\documentclass[a4paper,10pt]{scrartcl}
\usepackage[utf8]{inputenc}
\usepackage[english]{babel}
\usepackage[fixlanguage]{babelbib}
\usepackage{cite}
\usepackage{amssymb, amsmath, amstext, amsopn, amsthm, amscd, amsxtra, amsfonts}
\usepackage{yfonts, mathrsfs}
\usepackage{colonequals}
\usepackage{enumerate}
\usepackage{graphicx}%
\usepackage{colonequals}
\usepackage{tikz}
\usetikzlibrary{cd}
\usepackage{hyperref}%
\hypersetup{
urlcolor = red,
colorlinks = true,
linkcolor = blue,
citecolor = blue,
linktocpage = true,
pdftitle = {Lie groupoids of mappings taking values in a Lie groupoid},
pdfauthor = {Habib Amiri, Helge Gl{\"o}ckner, Alexander Schmeding},
bookmarksopen = true,
bookmarksopenlevel = 1,
unicode = true,
hypertexnames =false,
pdfencoding=auto,
unicode
}%
\swapnumbers
\newtheorem{thm}{Theorem}[section]
\newtheorem{la}[thm]{Lemma}
\newtheorem{prop}[thm]{Proposition}
\newtheorem{cor}[thm]{Corollary}

\theoremstyle{definition}
\newtheorem{defn}[thm]{Definition}
\newtheorem{exa}[thm]{Example}
\newtheorem{rem}[thm]{Remark}
\newtheorem{numba}[thm]{\!\!}

\newcommand{\Punkt}{\nopagebreak\hspace*{\fill}$\Box$}
\newcommand{\wb}{\overline}
\newcommand{\ve}{\varepsilon}

\newcommand{\impl}{\Rightarrow}
\newcommand{\mto}{\mapsto}
\newcommand{\N}{{\mathbb N}}
\newcommand{\R}{{\mathbb R}}

\newcommand{\bS}{{\mathbb S}}
\newcommand{\Sph}{{\mathbb S}}

\newcommand{\cL}{{\mathcal L}}

\newcommand{\cE}{{\mathcal E}}
\newcommand{\cA}{{\mathcal A}}

\newcommand{\cK}{{\mathcal K}}

\DeclareMathOperator{\GL}{GL}

\DeclareMathOperator{\id}{id}

\newcommand{\pl}{{\displaystyle \lim_{\longleftarrow}\, }}
\newcommand{\dl}{{\displaystyle \lim_{\longrightarrow}\, }}

\newcommand{\toto}{\ensuremath{\nobreak\rightrightarrows\nobreak}}
\newcommand{\coloneq}{\colonequals}

\DeclareMathOperator{\Diff}{Diff}

\DeclareMathOperator{\Iso}{Iso}

\DeclareMathOperator{\Supp}{supp}

\DeclareMathOperator{\pr}{pr}
\DeclareMathOperator{\ev}{ev}
\DeclareMathOperator{\graph}{graph}

\newcommand{\Frechet}{Fr\'{e}chet}
\newcommand{\LB}[1][\cdot \hspace{1pt} , \cdot]{\left[\hspace{1pt} #1 \hspace{1pt} \right]}
\newcommand{\Lf}{\mathbf{L}}
\date{}
\begin{document}
\title{Lie groupoids of mappings taking values in a Lie groupoid}
 \author{Habib Amiri\footnote{University of Zanjan, Iran
\href{mailto:h.amiri@znu.ac.ir}{h.amiri@znu.ac.ir}},\ Helge Gl\"{o}ckner\footnote{University of Paderborn, Germany, \href{mailto:glockner@math.upb.de}{glockner@math.upb.de}}\ \ ~~and Alexander
Schmeding\footnote{TU Berlin, Germany
\href{mailto:schmeding@tu-berlin.de}{schmeding@tu-berlin.de}
}%
}
{\let\newpage\relax\maketitle}

\begin{abstract}
Endowing differentiable functions from a compact manifold to a Lie group with the pointwise group operations one obtains the so-called current groups and, as a special case, loop groups. These are prime examples
of infinite-dimensional Lie groups modelled on locally convex spaces. 
In the present paper, we generalise this construction and show that differentiable mappings on a compact manifold (possibly with boundary) with values in a Lie groupoid form infinite-dimensional Lie groupoids which we call current groupoids.
We then study basic differential geometry and Lie theory for these Lie groupoids of mappings. In particular, we show that certain Lie groupoid properties, like being a proper \'{e}tale Lie groupoid, are inherited by the current groupoid. Furthermore, we identify the Lie algebroid of a current groupoid as a current algebroid (analogous to the current Lie algebra associated to a current Lie group).

To establish these results, we study superposition operators
\[
C^\ell(K,f)\colon C^\ell(K,M)\to C^\ell(K,N),\;\, \gamma\mto f\circ \gamma
\]
between manifolds of $C^\ell$-functions. Under natural
hypotheses, $C^\ell(K,f)$ turns out to be a submersion (an immersion, an embedding, proper, resp., a local diffeomorphism) if so is the underlying map $f\colon M\to N$. These results are new in their generality and of independent interest.
\end{abstract}

\medskip

% pseudogroups and differentiable groupoids

\tableofcontents

\section*{Introduction and statement of results} \addcontentsline{toc}{section}{Introduction and statement of results}
It is a well-known fact that the set
$C^\ell(K,G)$ of $C^\ell$-maps (for $\ell \in \N_0 \cup \{\infty\}$) from a compact manifold~$K$ to a Lie group~$G$ is again a Lie group (compare, e.g., \cite{Mil,Mic,PaS,Nee,GCX,GaN}).
Such mapping groups, often called current groups,
are prominent examples of infinite-dimensional Lie groups (notably loop groups $C^\ell(\bS,G)$, \cite{PaS}).
We perform an analogous construction for Lie groupoids, and study basic differential geometry and Lie theory for these current groupoids. In particular, we identify the Lie algebroid of a current groupoid as the corresponding current algebroid. Here, in analogy to the current Lie group/current Lie algebra picture
\cite{Nee,NaW,PaS}, a
current algebroid is a bundle of algebroid-valued differentiable maps whose Lie algebroid structure is induced by the pointwise operations. Moreover, we show that certain properties of Lie groupoids, such as being an \'{e}tale Lie groupoid, lift to the (infinite-dimensional) current groupoid.  
The key observation driving our approach is that superposition operators between manifolds of mappings inherit many properties from the underlying mappings. These results are new and of independent interest as they constitute a versatile tool to deal with some of the basic building blocks in infinite-dimensional geometry.\smallskip%\\[4mm]

Let us now describe our results in a bit more detail. Our construction is based on the fact that a manifold structure can be constructed on $C^\ell(K,M)$ whenever the target manifold~$M$ has a local addition (see Appendix~\ref{map-mfd}; cf.\
\cite{Eel,Ham,KaM,Mic,Mil}). Here the compact source manifold $K$ may have a smooth boundary, corners, or more generally a ``rough boundary'' as defined in~\cite{GaN} (and recalled in~\ref{def-rough}). For a smooth map $f \colon M \rightarrow N$, it is known that the manifold structures turn the superposition operator
$C^\ell (K,f) \colon C^\ell (K,M) \rightarrow C^\ell (K,N), C^\ell (K,f)(\gamma) \coloneq f\circ \gamma$
into a smooth map.
Note that the manifolds, Lie groups and Lie groupoids we study can be infinite-dimensional (in particular,
this is the case for the manifolds of mappings).
To deal with manifolds modelled on locally convex spaces beyond the Banach setting we work in the framework of the so-called Bastiani (or Keller $C^k_c$-) calculus \cite{Bas}, recalled in Section~\ref{sec:prelim}.
Throughout the following, we shall always consider a Lie groupoid $\mathcal{G} = (G\toto M)$ modelled on locally convex spaces with source map~$\alpha$ and target map~$\beta$ such that~$G$ and~$M$ admit local additions. Our results subsume the following theorem.\\[4mm]
\textbf{ Theorem~A.}
\emph{Assume that~$M$ is a smooth Banach manifold, $K$ a compact smooth manifold} (\emph{possibly with rough boundary}),
\emph{and $\ell\in\N$.
Then the pointwise operations turn $C^\ell (K,\mathcal{G}) \coloneq  (C^\ell(K,G) \toto C^\ell(K,M))$ into a Lie groupoid with source map $C^\ell(K,\alpha)$ and target map $C^\ell(K,\beta)$.
The same conclusion holds if $\ell=0$ and all modelling spaces of~$M$ are finite dimensional,
or if $\ell=\infty$ and all modelling spaces of~$G$ and~$M$ are finite dimensional.}\\[4mm]
Lie groupoids of the form $C^\ell(K,\mathcal{G})$ shall be referred to as
\emph{Lie groupoids of Lie groupoid-valued
mappings}, or \emph{current groupoids}. Since every Lie group can be interpreted as a Lie groupoid (over the one point manifold), current groupoids generalise current Lie groups and loop groups.

We then study basic differential geometry for current groupoids. For example, we identify
Lie subgroupoids and Lie groupoid actions which are induced by subgroupoids and actions of the target groupoids. Further, we investigate whether current groupoids inherit typical properties of Lie groupoids. To this end, recall the following typical properties of Lie groupoids.\\[4mm]
\textbf{ Definition} Consider a Lie groupoid $G$ with source map $\alpha\colon G\to M$
and target map $\beta\colon G\to M$. The Lie groupoid~$G$ is called
\begin{description}
\item[$\;$\'{e}tale] if $\alpha$ is a local $C^\infty$-diffeomorphism;
\item[$\;$proper] if $(\alpha,\beta)\colon G\to M\times M$ is a proper map;
\item[$\;$locally transitive] if $(\alpha,\beta)\colon G\to M\times M$ is a submersion;
\item[$\;$transitive] if $(\alpha,\beta)\colon G\to M\times M$ is a surjective submersion.\footnote{Our usage of "transitive" is as in \cite{MaM} and \cite{BaGaJaP} (but differs from
the notion of transitivity in \cite{Mac}).
Our concept of local transitivity
is as in~\cite{BaGaJaP} (where only Banach-Lie groupoids are considered).}
\end{description}
Concering local transitivity, we observe:\\[4mm]
\textbf{ Theorem B.}
\emph{If $\mathcal{G}$ is locally transitive in the situation of Theorem} A,
\emph{then also $C^\ell(K,\mathcal{G})$ is locally transitive.}\\[4mm]
We mention that $C^\ell(K,\mathcal{G})$ need not be transitive if~$\mathcal{G}$ is transitive (Example~\ref{not-tra}).
Likewise, $C^\ell(K,\mathcal{G})$ need not be proper if~$\mathcal{G}$ is proper (Example~\ref{not-proper}).
The situation improves if~$\mathcal{G}$ is \'{e}tale,
and we can even get around some hypotheses of Theorem~A
in this case.\\[4mm]
\textbf{ Theorem C.}
\emph{Let $\mathcal{G}=(G\toto M)$ be an \'{e}tale Lie groupoid modelled on locally convex spaces over a smooth manifold~$M$ modelled on locally convex spaces, $K$ a compact smooth manifold} (\emph{possibly with rough boundary}),
\emph{and $\ell\in\N_0\cup\{\infty\}$.
If the topological space underlying~$G$ is regular, then $C^\ell(K,\mathcal{G})$ is an \'{e}tale Lie groupoid.
If, moreover, $\mathcal{G}$ is proper, then also $C^\ell(K,\mathcal{G})$ is proper.}\\[4mm]
Analogs to Theorem~C are also available for topological groupoids
(see Corollary~\ref{top-eta}). As a consequence of Theorem C, the current groupoid of a proper \'{e}tale Lie groupoid will again be a proper \'{e}tale Lie groupoid. It is well known that proper \'{e}tale Lie groupoids are linked to orbifolds (cf.\ \cite{MaP,MaM,Sch}), whence they are also often called orbifold groupoids. In light of Theorem C, we can thus view the construction of Lie groupoids of orbifold Lie groupoid-valued mappings as a construction of infinite-dimensional orbifolds of mappings. However,
current groupoids are too simple in general to model spaces of orbifold morphisms \cite{RaV,RaV2,Chen}. This is discussed in detail in
Appendix~\ref{app:orbi}.\\[2.3mm]
Note that for $\ell \in \N_0$ and $\mathcal{G}$ a Banach-Lie groupoid, also the current groupoid will be a Banach-Lie groupoid (for $\ell = \infty$ and $\mathcal{G}$ a Banach-Lie groupoid,
the current groupoid is modelled on \Frechet\, spaces).
Basic (Lie) theory for Banach-Lie groupoids has recently been established in~\cite{BaGaJaP}.\\[2.3mm] 
On the infinitesimal level, one associates to a Lie groupoid a so-called Lie algebroid \cite{Mac,BaGaJaP}. 
The infinitesimal objects of current groupoids are as expected:\\[4mm]
\textbf{Theorem D} \emph{In the situation of Theorem} A, \emph{denote by $\mathcal{A} (\mathcal{G})$ the Lie algebroid associated to $\mathcal{G}$. Then there exists a canonical isomorphism of Lie algebroids such that}
$$\mathcal{A} (C^\ell (K,\mathcal{G})) \cong C^\ell (K,\mathcal{A} (\mathcal{G})),$$
\emph{where the Lie algebroid on the right hand side is given by the bundle $C^\ell (K,\mathcal{A} (G)) \rightarrow C^\ell (K,M)$ with the pointwise algebroid structure.}\\[2.3mm]
Again this generalises the case of current groups for which the construction yields (up to a shift in sign, see Remark \ref{rem:shiftsign}) the well-known construction of a current algebra \cite{NaW,KaW,Kac}.\\[4mm]
The main point of Theorem~A
is to see that $C^\ell(K,\alpha)$ and $C^\ell(K,\beta)$ are
submersions, and Theorem~B requires showing that $C^\ell(K,(\alpha,\beta))$ is a submersion.
Similarly, Theorem~C requires showing that $C^\ell(K,\alpha)$
is a local diffeomorphism (resp., that $C^\ell(K,(\alpha,\beta))$ is proper).
The following result provides these properties.\\[4mm]
\textbf{ Theorem E.}
\emph{Let $M$ and $N$ be smooth
manifolds modelled on locally convex spaces
such that~$M$ and~$N$ admit a local addition.
Let $k,\ell\in \N_0\cup\{\infty\}$,
$f\colon M\to N$ be a $C^{k+\ell}$-map
and $K$ be a compact manifold $($possibly with rough boundary$)$.
Then the~$C^k$-map}
\[
C^\ell(K,f)\colon C^\ell(K,M)\to C^\ell(K,N),\quad \gamma\mto f\circ\gamma
\]
\emph{has the following properties}:
\begin{itemize}
\item[(a)]
\emph{If $f$ is a submersion, $N$ is modelled on Banach spaces,
$k\geq 2$ and $\ell<\infty$,
then $C^\ell(K,f)$
is a submersion, assuming $\ell\geq 1$ if some modelling space of~$N$
is infinite-dimensional.}
\item[(b)]
\emph{If $f$ is an immersion, $M$ is modelled on Banach spaces,
$k\geq 2$ and $\ell<\infty$, then $C^\ell(K,f)$ is an immersion, assuming $\ell\geq 1$
if some modelling space of~$M$ is infinite-dimensional.}
\item[(c)]
\emph{If $f$ is a local $C^{k+\ell}$-diffeomorphism and $M$ is a regular topological space,
then $C^\ell(K,f)$ is a local $C^k$-diffeomorphism.}
\item[(d)]
\emph{If $f$ is a proper map, $M$ is a regular topological space and $N=N_1\times N_2$
with smooth manifolds~$N_1$ and~$N_2$ such that~$N_1$ admits a local addition
and
$\pr_1\circ\, f\colon M\to N_1$ is a local $C^{k+\ell}$-diffeomorphism, then $C^\ell(K,f)$ is proper.}
\end{itemize}
We remark that Theorem E (a) and (b) generalise a similar result by Palais
\cite[Theorem 14.10]{Pal2} for certain morphisms of (smooth) fiber bundles.

If $f$ in Theorem~E (b) is, moreover, a homeomorphism onto its image,
then so is $C^\ell(K,f)$ (see Lemma~\ref{embpfwd}),
whence $C^\ell(K,f)$ is an embedding of $C^k$-manifolds.
What is more, we have the following variant
(which also varies a result in~\cite{Mic}),
as a special case of Proposition~\ref{prop:emb}:\\[4mm]
\textbf{ Theorem F.} \emph{Let $e \colon M \rightarrow N$ be a smooth embedding between finite-dimensional manifolds,
$K$ be a compact smooth manifold $($possibly with rough boundary$)$, and
$\ell \in \N_0 \cup \{\infty\}$. Then $C^\ell (K,e) \colon C^\ell (K,M) \rightarrow C^\ell (K,N)$ is an embedding.}\\[2.3mm]
If $\ell=0$, then~$K$ can actually be replaced with an arbitrary compact topological space in Theorems~A--F (in view of \cite[Remark~4.9]{AaS}),
assuming moreover that~$K$ is locally connected for the conclusions concerning properness (cf.\ Proposition~\ref{ess-prop} and Corollary~\ref{top-eta}).

\section{Preliminaries}\label{sec:prelim}
We shall write $\N=\{1,2,\ldots\}$ and $\N_0:=\N\cup\{0\}$.
Hausdorff locally convex real topological vector spaces
will be referred to as locally convex spaces.
If $E$ and~$F$ are locally convex spaces, we let
$\cL(E,F)$ be the space of all continuous linear mappings from~$E$ to~$F$.
We write $\cL(E,F)_c$ and $\cL(E,F)_b$, respectively, if the topology
of uniform convergence on compact sets (resp., bounded sets) is used on~$\cL(E,F)$.
We write $\GL(E)$ for the group of automorphisms of~$E$
as a locally convex space;
if $E$ is a Banach space, then $\GL(E)$ is an open subset
of $\cL(E)_b:=\cL(E,E)_b$.
A subset $U$ of a locally convex space~$E$ is called \emph{locally convex}
if for each $x\in U$, there exists a convex $x$-neighborhood in~$U$.
Every open set $U\subseteq E$ is locally convex. We shall work in a setting of infinite-dimensional
calculus known as Bastiani calculus or Keller's $C^k_c$-theory, going back to~\cite{Bas},
and generalizations thereof (see \cite{GaN} and \cite{AaS},
also \cite{RES,Ham,Mic}, and~\cite{Mil}).

\begin{numba}
If $E$ and $F$ are locally convex spaces and $f\colon U\to F$
is a mapping on a locally convex subset $U\subseteq E$ with dense interior~$U^0$, we write
\[
(D_yf)(x):=\frac{d}{dt}\Big|_{t=0}f(x+ty)
\]
for the directional derivative of $f$ at $x\in U^0$ in the direction $y\in E$, if it exists.
A mapping $f\colon U\to F$ is called $C^k$ with $k\in \N_0\cup\{\infty\}$
if $f$ is continuous and there exist continuous mappings
$d^{\,(j)}f\colon U\times E^j\to F$ for all $j\in \N$ with $j\leq k$ such that
\[
d^{\,(j)}f(x,y_1,\ldots,y_j)=(D_{y_j}\cdots D_{y_1}f)(x)
\;\,\mbox{for all $x\in U^0$ and $y_1,\ldots, y_j\in E$.}
\]
\end{numba}
\begin{numba}\label{def-rough}
As compositions of $C^k$-maps are $C^k$, one can define $C^k$-manifolds
modelled on a set~$\cE$ of locally convex spaces as expected:
Such a manifold is a Hausdorff topological space~$M$,
together with a maximal set $\cA$ of homeomorphisms $\phi\colon U_\phi\to V_\phi$
(``charts'')
from an open subset $U_\phi\subseteq M$ onto an open subset $V_\phi\subseteq E_\phi$
for some $E_\phi\in\cE$ such that $\bigcup_{\phi\in\cA}U_\phi=M$
and $\phi\circ\psi^{-1}$ is $C^k$ for all $\phi,\psi\in\cA$.
If the sets $V_\phi$ in the definition of a $C^k$-manifold
are only required to be locally convex subsets with dense interior of some $E_\phi\in \cE$
(but not necessarily open), we obtain the more general concept
of a \emph{$C^k$-manifold with rough boundary} modelled on~$\cE$.

If all locally convex spaces in $\cE$ are Banach, \Frechet\, or finite-dimensional spaces, we say that $M$ is a Banach, or \Frechet\, or finite-dimensional manifold (possibly with rough boundary) respectively. Note that a priori all manifolds in this paper are allowed to be modelled on locally convex spaces and we suppress this in the notation (only emphasising the special cases
of Banach and \Frechet\, manifolds).
If $\cE=\{E\}$ consists of a single locally convex space, then $M$ is a
\emph{pure} $C^k$-manifold. Only this case is considered in~\cite{GaN},
but it captures the essentials as each connected component
of a $C^k$-manifold is open, and can be considered as a pure $C^k$-manifold.
However, the manifolds $C^\ell(K,M)$ we are about to consider
need not be pure (even if~$M$ is pure).
\end{numba}

\begin{rem}
Every ordinary finite-dimensional manifold with smooth boundary
also is a smooth manifold with rough boundary,
and so are finite-dimensional smooth manifolds with corners
in the sense of~\cite{Mic}.
\end{rem}
\begin{numba}
As usual, a map $f\colon M\to N$ between $C^k$-manifolds (possibly with rough boundary) is called~$C^k$ if it is continuous
and $\phi\circ f\circ\psi^{-1}$ is~$C^k$ for all charts $\psi$ and~$\phi$ of~$M$ and~$N$,
respectively. 

For a $C^k$-map $f\colon M\rightarrow N$ between manifolds \emph{without boundary}, we say 
(see \cite{Ham,SUB}) that $f$ is
\begin{enumerate}
\item a \emph{submersion} (or \emph{$C^k$-submersion}, for emphasis) if for each $x\in M$ we can choose a chart $\psi$
of~$M$ around~$x$ and a chart $\phi$ of~$N$
around $f(x)$ such that
$\phi\circ f\circ~\psi^{-1}$
is the restriction of a continuous linear map with continuous linear right inverse;
\item an \emph{immersion} (or \emph{$C^k$-immersion}) if for each $x\in M$ there are charts such that
$\phi\circ f\circ \psi^{-1}$ is the restriction of a continuous linear map admitting a continuous linear left inverse;
\item an \emph{embedding} (or \emph{$C^k$-embedding}) if $f$ is a $C^k$-immersion and a topological embedding;
\item a \emph{local $C^k$-diffeomorphism} if each $x\in M$ has an open neighborhood
$U\subseteq M$ such that $f(U)$ is open in~$N$ and $f|_U\colon U\to f(U)$ is a $C^k$-diffeomorphism.
\end{enumerate}
If the tangent map $T_xf\colon T_xM\to T_{f(x)}M$ has a continuous linear
right inverse\footnote{We then also say that $T_xf$ is a split linear surjection.}
for each $x\in M$, then~$f$ is called a \emph{na\"{\i}ve submersion}.
If each tangent map $T_xf$ has a continuous linear lect inverse,
then $f$ is called a \emph{na\"{\i}ve immersion.}
\end{numba}
It is essential for us to consider mappings on products with
different degrees of differentiability in the two factors, as in~\cite{AaS}
(or also~\cite{GaN}).
\begin{numba}
Let $E_1$, $E_2$, and $F$ be locally convex spaces and $f\colon U\times V\to F$
be a mapping on a product of locally convex subsets $U\subseteq E_1$ and $V\subseteq E_2$
with dense interior. Given $k,\ell\in\N_0\cup\{\infty\}$, we say that $f$ is $C^{k,\ell}$
if $f$ is continuous and there exist continuous mappings
$d^{\,(i,j)}f\colon U\times V\times E_1^i\times E_2^j\to F$
for all $i,j\in\N_0$ such that $i\leq k$ and $j\leq \ell$
such that
\[
d^{\,(i,j)}f(x,y,v_1,\ldots,v_i,w_1,\ldots,w_j)=(D_{(v_i,0)}\cdots D_{(v_1,0)}
D_{(0,w_j)}\cdots D_{(0,w_1)}f)(x,y)
\]
for all $x\in U^0$, $y\in V^0$ and $v_1,\ldots, v_i\in E_1$,
$w_1,\ldots, w_j\in E_2$.
\end{numba}
One can also define $C^{k,\ell}$-maps $M_1\times M_2\to N$
if~$M_1$ is a $C^k$-manifold, $M_2$ a $C^\ell$-manifold and~$N$
a $C^{k+\ell}$-manifold (all possibly with rough boundary),
checking the property in local charts.
\begin{numba}[Submanifolds]
Let $M$ be a $C^k$-manifold (possibly with rough boundary).
A subset $N\subseteq M$ is called a \emph{submanifold}
if, for each $x\in N$, there exists a chart $\phi\colon U_\phi\to V_\phi\subseteq E_\phi$
of~$M$ with $x\in U_\phi$ and a closed vector subspace $F\subseteq E_\phi$
such that $\phi(U_\phi\cap N)=V_\phi\cap F$ and $V_\phi\cap F$ has non-empty interior
in~$F$ (note that the final condition is automatic if~$M$ is a manifold without boundary).
Then~$N$ is a $C^k$-manifold
in the induced topology, using the charts $\phi|_{U_\phi\cap N}\colon U_\phi\cap N\to V_\phi\cap F$.
If $F$ can be chosen as a vector subspace of~$E_\phi$
which is complemented in~$E_\phi$ as a topological vector space,
then $N$ is called a \emph{split submanifold} of~$M$.
\end{numba}
The following two observations are well known (see, e.g., \cite{GaN}):
\begin{numba}\label{into-sub}
If $N$ is a submanifold of a $C^k$-manifold~$M$ (possibly with rough boundary)
and $f\colon L\to M$ a map on a $C^k$-manifold~$L$
such that $f(L)\subseteq N$, then $f$ is~$C^k$ if and only if its corestriction
$f|^N\colon L\to N$ is~$C^k$ for the $C^k$-manifold structure induced on~$N$.
\end{numba}
\begin{numba}\label{Ckellinto}
If $M$ is a $C^{k+\ell}$-manifold, $L_1$ is a $C^k$-manifold, $L_2$ a $C^\ell$-manifold
(all possibly with rough boundary)
and $f\colon L_1\times L_2\to M$ is a map with image in a submanifold~$N\subseteq M$,
then $f$ is $C^{k,\ell}$ if and only if $f|^N$ is~$C^{k,\ell}$.
\end{numba}
\begin{numba}
Let $k\in \N\cup\{\infty\}$. We mention that a $C^k$-map $f\colon M\to N$ between $C^k$-manifolds
without boundary is a $C^k$-embedding if and only if $f(M)$ is a split submanifold of~$N$
and $f|^{f(M)}\colon M\to f(M)$ is a $C^k$-diffeomorphism
(see \cite[Lemma 1.13]{SUB}). If $M$ and $N$ are $C^k$-manifolds
which may have a rough boundary, we take the latter property as the definition
of a $C^k$-embedding $f\colon M\to N$.
\end{numba}
\begin{numba}
If~$M$ is a $C^1$-manifold (possibly with rough boundary)
and $f\colon M\to U$ a $C^1$-map to an open subset $U$ of a locally convex space~$E$,
we identify the tangent bundle $TU$ with $U\times E$, as usual,
and let $df$ be the second component of the tangent map $Tf\colon TM\to TU=U\times E$.
\end{numba}
For Lie groups modelled on locally convex spaces, we refer to \cite{Mil,Nee,GaN}.
\begin{numba}
Consider a groupoid~$\mathcal{G} = (G\toto M)$, with source map $\alpha\colon G\to M$ and target map $\beta\colon G\to M$. If $G$ and~$M$ are smooth manifolds, $\alpha$ and~$\beta$ are $C^\infty$-submersions
and the multiplication map $G^{(2)}\to G$, the inversion map $G\to G$
and the identity-assigning map $M\to G$, $x\mto \mathbf{1}_x$ are smooth,
then~$\mathcal{G}$ is called a \emph{Lie groupoid}.
If $\mathcal{G}$ is a Lie groupoid and both of the manifolds $G$ and $M$ are modelled
on Banach spaces, then $\mathcal{G}$ is called a \emph{Banach-Lie groupoid}.
\end{numba}
\begin{numba}\label{the-topo}
Let $M$ and~$N$ be $C^k$-manifolds (possibly with rough boundary),
where $k\in \N_0\cup\{\infty\}$.
Given a $C^k$-map $\gamma\colon M\to N$,
we set $T^0M:=M$, $T^0N:=N$, $T^0\gamma:=\gamma$.
Recursively, we define iterated tangent maps $T^j\gamma:=T(T^{j-1}\gamma)$
from $T^jM:=T(T^{j-1}M)$ to $T^jN:=T(T^{j-1}N)$.
We endow the set $C^k(M,N)$ of all $N$-valued $C^k$-maps on~$M$
with the initial topology ${\mathcal O}$ with respect to the maps
\[
T^j\colon C^k(M,N)\to C(T^jM,T^jN),\;\, \gamma\mto T^j\gamma,
\]
for $j\in\N_0$ with $j\leq k$,
where $C(T^jM,T^jN)$ is endowed with the compact-open topology.
The topology ${\mathcal O}$ is called the \emph{compact-open $C^k$-topology}
on $C^k(M,N)$.
\end{numba}
We shall use the following fact.
\begin{la}\label{top-cover}
Let $M$ and $N$ be $C^k$-manifolds modelled on locally convex spaces
$($possibly with rough boundary$)$, where $k\in\N_0\cup\{\infty\}$.
If $(U_i)_{i\in I}$ is an open cover of~$M$, then the topology on $C^k(M,N)$ is initial
with respect to the restriction maps $\rho_i\colon C^k(M,N)\to C^k(U_i,N)$ for $i\in I$.
\end{la}
\begin{proof}
For each $j\in \N_0$ such that $j\leq k$,
the sets $T^jU_i$ form an open cover of $T^jM$ for $i\in I$,
whence the compact-open topology on $C(T^jM,T^jN)$ is initial with respect to
the restriction maps $\rho_{i,j}\colon C(T^jM,T^jN)\to C(T^jU_i,T^jN)$ for $i\in I$ (see
\cite[Lemma~A.5.11]{GaN}).
By transitivity of initial topologies \cite[Lemma~A.2.7]{GaN},
the topology on $C^k(M,N)$ is initial with respect
to the mappings $\rho_{i,j}\circ T^j$ for $i\in I$ and $j\in\N_0$ with $j\leq k$.
Again by transitivity of initial topologies, the initial topology on $C^k(M,N)$ with respect
to the maps $\rho_{i,j}\circ T^j=T^j\circ\rho_i$ coincides with the initial topology with respect to the mappings~$\rho_i$.
\end{proof}
If $M$, $N$, and $S$ are Hausdorff topological spaces and $f\colon S\to N$
is a continuous map, then also $C(M,f)\colon C(M,S)\to C(M,N)$
is continuous for the compact-open topologies; if $f$ is a topological embedding,
then so is $C(M,f)$
(see, e.g., \cite[Appendix~A]{GaN}). Likewise, the following holds:
\begin{la}\label{embpfwd}
Let $M$, $N$, and $S$ be $C^k$-manifolds modelled on locally convex spaces
$($possibly with rough boundary$)$, where $k\in\N\cup\{\infty\}$.
If $f\colon S\to N$ is a $C^k$-map,
then $C^k(M,f)\colon C^k(M,S)\to C^k(M,N)$ is continuous.
If $f$ is a $C^k$-embedding, then $C^k(M,f)$
is a topological embedding.
\end{la}
\begin{proof}
The first assertion follows from the continuity of the maps $T^j\circ C^k(M,f)=C(T^jM,T^jf)\circ T^j$.
If $f$ is a $C^k$-embedding,
then $f(S)$ is a $C^k$-submanifold of~$N$
and $f|^{f(S)}\colon S\to f(S)$ is a $C^k$-diffeomorphism, by \cite[Lemma 1.13]{SUB}.
Hence $$C^k(M,f|^{f(S)})\colon C^k(M,S)\to C^k(M,f(S))$$
is a homeomorphism. After replacing $S$ with $f(S)$, we may assume
that $S$ is a submanifold of~$N$ and $f\colon S\to N$ the inclusion map.
Since $T^jS$ is a submanifold of $T^jN$ for each $j\in\N_0$ with $j\leq k$
and the topology on the iterated tangent bundle coincides with the topology induced by $T^jN$,
we deduce that the topology on $C^k(M,S)$ is induced by $C^k(M,N)$.
\end{proof}
\begin{numba}
If $\pi\colon E\to M$ is a $C^k$-vector
bundle,\footnote{Thus $E$ and $M$ are $C^k$-manifolds (possibly with rough boundary),
$f$ is a surjective $C^k$-map and a vector space structure is given
on $E_x\coloneq f^{-1}(\{x\}$ for each $x\in M$ such that~$E$ is locally trivial
in the sense that each $x\in M$ has an open neighborhood $U\subseteq M$
for which there exists a $C^k$-diffeomorphism $\theta=(\theta_1,\theta_2)
\colon f^{-1}(U)\to U\times F$
for some locally convex space~$F$ such that $\theta_1=\pi|_{f^{-1}(U)}$
and $\theta_2|_{E_y}$ is linear for all $y\in U$.}
we write $\Gamma_{C^k}(E)$ for its space of $C^k$-sections
$\sigma\colon M\to E$ (thus $\pi\circ\sigma=\id_M$).
The topology induced by $C^k(M,E)$ on $\Gamma_{C^k}(E)$
makes the latter a locally convex space (see, e.g., \cite{GaN}).
If $U\subseteq M$ is an open subset, we write $E|_U:=\pi^{-1}(U)$. Finally, if all fibers $E_x$ of $E$ are Banach
spaces, \Frechet\, spaces or finite-dimensional, we say that $E$ is a Banach or \Frechet\, or finite rank bundle, respectively.
\end{numba}
\subsection*{Canonical manifolds of mappings}
  \addcontentsline{toc}{subsection}{Canonical manifolds of mappings}

We now define canonical manifold structures for manifolds of mappings. This allows us to identify the properties of manifolds of mappings necessary for our approach without having to deal with the details of the actual constructions (these are referenced in Appendix \ref{map-mfd}).

\begin{numba}[General Assumptions]
In the following we will (unless noted otherwise) use the following conventions and assumptions:
$K$ will be a compact smooth manifold (possibly with rough boundary),
$M, N$ will be smooth manifolds,
and $\ell, k\in\N_0\cup\{\infty\}$.
\end{numba}  
  
\begin{defn}\label{defcanon}
We say that a smooth manifold structure on the set $C^\ell(K,M)$ is \emph{canonical} if
its underlying topology is the compact-open $C^k$-topology and
the following holds: For each $k\in\N_0\cup\{\infty\}$,
each $C^k$-manifold $N$ (possibly with rough boundary) modelled on locally convex
spaces and each map $f\colon N\to C^\ell(K,M)$, the map $f$ is~$C^k$
if and only if
\[
f^\wedge\colon N\times K\to M,\;\, (x,y)\mto f(x)(y)
\]
is a $C^{k,\ell}$-map.
\end{defn}
\begin{rem}
A canonical manifold structure enforces a suitable version of the exponential law (cf.\ \cite{AaS,KaM}) which enables differentiability properties of mappings to be verified by computing them on the underlying manifolds. 
Thus we can avoid the (rather involved) manifold structure on manifolds
of mappings in many situations (similar ideas have been used in~\cite{NaW}). 
We hasten to remark that the usual constructions of manifolds of mappings yield canonical manifold structures (cf.\ the end of the present section and Appendix \ref{map-mfd}).
\end{rem}

\begin{la}\label{base-cano}
If $C^\ell (K,M)$ is endowed with a canonical manifold structure, then
\begin{itemize}
\item[\textup{(a)}]
the evaluation map $\ve\colon C^\ell(K,M)\times K\to M$,
$\ve(\gamma,x):=\gamma(x)$ is a $C^{\infty,\ell}$-map.
\item[\textup{(b)}]
Canonical manifold structures are unique in the following sense:
If we write $C^\ell(K,M)'$
for $C^\ell(K,M)$, endowed with another canonical manifold structure,
then $\id\colon C^\ell(K,M)\to C^\ell(K,M)'$, $\gamma\mto\gamma$
is a $C^\infty$-diffeomorphism.
\item[\textup{(c)}]
Let $N\subseteq M$ be a submanifold such that the set $C^\ell(K,N)$
is a submanifold of $C^\ell(K,M)$.
Then the submanifold structure on $C^\ell(K,N)$ is canonical.
\item[\textup{(d)}]
If $M_1$, and $M_2$ are smooth manifolds such that $C^\ell(K,M_1)$ and $C^\ell(K,M_2)$ have canonical
manifold structures, then the manifold structure
on the product manifold $C^\ell(K,M_1)\times C^\ell(K,M_2)$
$\cong  C^\ell(K,M_1\times M_2)$
is canonical.
\end{itemize}
\end{la}
\begin{proof}
(a) Since $\id\colon C^\ell(K,M)\to C^\ell(K,M)$ is~$C^\infty$ and $C^\ell(K,M)$
is endowed with a canonical manifold structure, it follows that $\id^\wedge\colon
C^\ell(K,M)\times K\to M$, $(\gamma,x)\mto \id(\gamma)(x)=\gamma(x)=\ve(\gamma,x)$
is $C^{\infty,\ell}$.

(b) The map $f:=\id\colon C^\ell(K,M)\to C^\ell(K,M)'$ satisfies
$f^\wedge=\ve$ where $\ve\colon C^\ell(K,M)\times K\to M$
is~$C^{\infty,\ell}$, by~(a). Since $C^\ell(K,M)'$ is endowed with a canonical manifold structure,
it follows that~$f$ is~$C^\infty$. By the same reasoning,
$f^{-1}=\id\colon C^\ell(K,M)'\to C^\ell(K,M)$ is~$C^\infty$.

(c) As $C^\ell(K,N)$ is a submanifold,
the inclusion $\iota\colon C^\ell(K,N)\to C^\ell(K,M)$, $\gamma\mto\gamma$ is~$C^\infty$.
Likewise, the inclusion map $j\colon N\to M$ is~$C^\infty$.
Let $L$ be a manifold (possibly with rough boundary)
modelled on locally convex spaces and $f\colon L\to C^\ell(K,N)$ be a map.
If~$f$ is~$C^k$, then $\iota\circ f$ is~$C^k$, entailing that
$(\iota\circ f)^\wedge\colon L\times K\to M$, $(x,y)\mto f(x)(y)$ is~$C^{k,\ell}$.
As the image of this map is contained in~$N$, which is a submanifold of~$M$,
we deduce that $f^\wedge=(\iota\circ f)^\wedge|^N$ is~$C^{k,\ell}$.
If, conversely, $f^\wedge\colon L\times K\to N$ is $C^{k,\ell}$,
then also $(\iota\circ f)^\wedge=j\circ (f^\wedge)\colon L\times K\to M$ is $C^{k,\ell}$.
Hence $\iota\circ f\colon L\to C^\ell(K,M)$ is~$C^k$ (the manifold structure on the range being
canonical). As $\iota\circ f$ is a $C^k$-map map with image in $C^\ell(K,N)$ which is a submanifold
of $C^\ell(K,M)$, we deduce that~$f$ is~$C^k$.

(d) If~$L$ is a $C^k$-manifold (possibly with rough boundary) and $f=(f_1,f_2)\colon
L\to C^\ell(K,M_1)\times C^\ell(K,M_2)$
a map, then~$f$ is~$C^k$ if and only if~$f_1$ and $f_2$ are~$C^k$.
As the manifold structures are canonical, this holds if and only if
$f_1^\wedge\colon L\times K\to M_1$ and $f_2^\wedge\colon L\times K\to M_2$
are $C^{k,\ell}$, which holds if and only if $f^\wedge=(f_1^\wedge,f_2^\wedge)$
is $C^{k,\ell}$.
\end{proof}
\begin{prop}\label{fstar-gen}
Assume that $C^\ell(K,M)$ and $C^\ell(K,N)$
admit canonical manifold structures.
If $\Omega\subseteq K\times M$ is an open subset
and $f\colon \Omega \to N$ is a $C^{k+\ell}$-map,
then
\[
\Omega':=\{\gamma\in C^\ell(K,M)\colon \graph(\gamma)\subseteq\Omega\}
\]
is an open subset of $C^\ell(K,M)$ and
\[
f_\star\colon \Omega' \to C^\ell(K,N),\;\, \gamma\mto f\circ (\id_K,\gamma)
\]
is a $C^k$-map.
\end{prop}
\begin{proof}
As the compact open topology on $C(K,M)$ coincides with the graph topology
(see, e.g., \cite[Proposition A.5.25]{GaN}),
$\{\gamma\in C(K,M)\colon \graph(\gamma)\subseteq \Omega\}$
is open in $C(K,M)$. As a consequence,
$\Omega'$ is open in $C^\ell(K,M)$,
exploiting that $C^\ell(K,M)$ is endowed with the compact-open
$C^\ell$-topology\footnote{We only included this requirement in Definition~\ref{defcanon}
to enable the current argument.}
which
is finer that the compact-open topology.
By Lemma~\ref{base-cano}\,(a), the map $\ve\colon C^\ell(K,M)\times K\to M$,
$(\gamma,x)\mto\gamma(x)$ is $C^{\infty,\ell}$ and hence~$C^{k,\ell}$,
whence also $C^\ell(K,M)\times K\to K\times M$, $(\gamma,x)\mto (x,\gamma(x))$
is $C^{k,\ell}$. Since $f$ is $C^{k+\ell}$, the Chain Rule shows that
\[
(f_\star)^\wedge\colon \Omega' \times K\to N, \;\,
(\gamma,x)\mto f_\star(\gamma)(x)
=f(x,\gamma(x))
\]
is $C^{k,\ell}$. So $f_\star$ is~$C^k$,
as the manifold structure on $C^\ell(K,N)$ is canonical.
\end{proof}
\begin{cor}\label{fstar-basic}
Assume that $C^\ell(K,M)$ and $C^\ell(K,N)$
admit canonical manifold structures.
If $f\colon K\times M\to N$ is a $C^{k+\ell}$-map,
then
\[
f_\star\colon C^\ell(K,M)\to C^\ell(K,N),\;\, \gamma\mto f\circ (\id_K,\gamma)
\]
is a $C^k$-map.\,\Punkt
\end{cor}
Applying Corollary~\ref{fstar-basic} with $f(x,y):=g(y)$, we get:
\begin{cor}\label{la-reu}
Assume that $C^\ell(K,M)$ and $C^\ell(K,N)$
admit canonical manifold structures.
If $g\colon M\to N$ is a $C^{k+\ell}$-map,
then
\[
g_*:=C^\ell(K,g) \colon C^\ell(K,M)\to C^\ell(K,N),\;\, \gamma\mto g\circ \gamma
\]
is a $C^k$-map.\,\Punkt
\end{cor}
We emphasise that the usual construction of manifolds of mappings using a local addition (see Definition \ref{def-loa}) produces canonical manifold structures.
This is recorded in the next proposition which slighly generalizes well-known constructions
(cf.\ \cite{Eel});\footnote{In this connection, the second author also profited from a talk by P.W.\ Michor at the 50th Seminar Sophus Lie, 2016 in Bedlewo (Poland).}
the proof can be found in Appendix~\ref{map-mfd}.
\begin{prop}\label{prop: can:locadd}
If~$M$ admits a local addition, then $C^\ell(K,M)$
admits a canonical smooth manifold structure, for each $\ell\in\N_0\cup\{\infty\}$.
Its underlying topology is as in \emph{\ref{the-topo}}
and the tangent bundle can be identified with the manifold $C^\ell (K,TM)$.
\end{prop}
Every paracompact finite-dimensional smooth manifold~$M$ admits
a local addition (e.g., one can choose a Riemannian metric on~$M$
and restrict the Riemannian exponential map to a suitable neighborhood
of the zero-section). 
Every Lie group $G$ modelled on a locally convex space~$E$
admits a local addition.\footnote{Given $g\in G$, let $\lambda_g\colon G\to G$, $x\mto gx$
be left translation by~$g$. Consider the smooth map
$\omega\colon TG\to T_eG=:{\mathfrak g}$, $v\mto (T\lambda_{\pi_{TG}(v)})^{-1}(v)$
and choose a $C^\infty$-diffeomorphism $\psi\colon V\to W$ from
an open $0$-neighborhood $V\subseteq {\mathfrak g}$ onto an open identity-neighborhood
$W\subseteq G$ such that $\psi(0)=e$. Define $U:=\bigcup_{g\in G}T\lambda_g (V)$.
Then $\Sigma\colon U\to G$, $v\mto \pi_{TG}(v)\psi(\omega(v))$
is a local addition for~$G$.}
We refer to \cite{SaW} for more information on local additions on (infinite-dimensional) Lie groupoids.

The tangent map of the push-forward map $C^\ell (K,f)$ can be identified in the presence of a local addition. As recalled in
Appendix \ref{map-mfd}, up to certain (explicit) bundle isomorphisms it has the following form.

\begin{numba}\label{thetamp}
Assume that $M,N$ admit local additions and $f \colon M \rightarrow N$ is a $C^{\ell +1}$ map. Then the identification $TC^\ell (K,M)\cong C^\ell (K,TM)$ induces a commuting diagram
\begin{equation}\label{eq:commdiag}
\begin{tikzcd}
TC^{\ell} (K,M) \arrow[r, "\cong"] \arrow[d, "TC^\ell(K{,}f)"]
& C^\ell (K,TM)  \arrow[d, "C^\ell (K{,}Tf)"]  \\
TC^\ell (K,N)  \arrow[r, "\cong"] & C^\ell (K,TN);
\end{tikzcd}
\end{equation}
see Corollary~\ref{formtang} for details.
\end{numba}
\section{Lifting properties of maps to manifolds of mappings}\label{sec: m:prop}
In this section, we prove that certain properties of mappings between manifolds
(e.g.\ the submersion property) are inherited by the corresponding
push-forward mappings between infinite-dimensional manifolds of mappings.
Notably, we shall prove Theorem~E, in several steps.
Whenever we refer to the theorem,
$K$ will denote a compact smooth manifold (possibly with rough boundary).
Moreover, $M$ and $N$ will be smooth manifolds admitting a local addition,
and $k,\ell \in \N_0 \cup\{\infty\}$.

\subsection*{Submersions between manifolds of mappings}\label{sec-sub}
  \addcontentsline{toc}{subsection}{Submersions between manifolds of mappings}
The proof of Theorem~E (a) relies on two lemmas,
which show that certain linear mappings between spaces of sections split.
To start with, we consider the case of
trivial vector bundles.
\begin{la}\label{triv-case}
Let $X$ and~$Y$ be locally convex spaces, $Z$ be a Banach space,
$U\subseteq X$ be a locally convex subset with dense interior and $f\colon U\times Y\to Z$ be a $C^\ell$-map
such that $f_x\coloneq f(x,\cdot)\colon Y\to Z$ is linear for each
$x\in U$. If $Z$ is infinite-dimensional and $\ell=0$, assume that~$f$ is~$C^1$.
If $x_0\in U$ such that $f_{x_0}$ is a split linear surjection,
then 
\[
(f|_{U_0})_\star\colon C^\ell(U_0,Y)\to C^\ell(U_0,Z),\;\, \gamma\mto f\circ (\id_{U_0},\gamma)
\]
is a split linear surjection for some open $x_0$-neighborhood
$U_0\subseteq U$.
\end{la}
\begin{proof}
As $f_{x_0}$ is a split linear surjection, there exists a closed
vector subspace $E\subseteq Y$ such that $f_{x_0}|_E\colon E\to Z$
is an isomorphism of topological vector spaces.
The inclusion map $j_E\colon E\to Y$ is continuous linear, entailing that
the restriction map
\[
\rho_E:=\cL(j_E,Z)\colon\cL(Y,Z)_b\to\cL(E,Z)_b,\;\, \alpha\mto\alpha\circ j_E=\alpha|_E
\]
is continuous linear.
If $f$ is $C^1$, then
\[
f^\vee\colon U\to\cL(Y,Z)_b, \;\, x\mto f_x
\]
is continuous, by \cite[Lemma~1.5.9]{GaN}.
If $Y$ is finite-dimensional and $\ell=0$, then $\cL(Y,Z)_c=\cL(Y,Z)_b$,
whence $f^\vee\colon U\to\cL(Y,Z)_c=\cL(Y,Z)_b$ is continuous by
\cite[Proposition~A.5.17]{GaN}.
In any case, $f^\vee\colon U\to\cL(Y,Z)_b$ is continuous, whence also
\[
h:=\rho_E\circ f^\vee \colon U\to\cL(E,Z)_b,\;\, x\mto f_x|_E
\]
is continuous.
Since the set $\Iso(E,Z)$ of invertible operators is open in $\cL(E,Z)_b$, there exists an open $x_0$-neighborhood
$U_0\subseteq U$ such that $h(U_0)\subseteq\Iso(E,Z)$. As the inversion map
\[
\iota\colon \cL(E,Z)_b\supseteq \Iso(E,Z) \to\cL(Z,E)_b, \;\, \alpha\mto\alpha^{-1}
\]
is continuous, we deduce that the map
$\iota\circ h|_{U_0}\colon U_0\to \cL(Z,E)_b$, $x\mto h(x)^{-1}=(f_x|_E)^{-1}$
is continuous and hence also the map
\[
g\colon U_0\times Z\to E\subseteq Y,\;\, g(x,z):=(\iota\circ h)(x)(z)=(f_x|_E)^{-1}(z)
\]
(as the evaluation map $\cL(Z,E)_b\times Z\to E$, $(T,y)\mto Ty$ is continuous).
Since~$f$ is~$C^\ell$, this entails that~$g$ is~$C^\ell$
(see \cite[Lemma~2.3]{DiK}; cf.\ also \cite[Exercise 1.3.10]{GaN}, \cite[Theorem 5.3.1]{Ham}).
As a consequence, the linear map
\[
g_\star\colon C^\ell(U_0,Z)\to C^\ell(U_0,Y),\, \gamma\mto g\circ (\id_{U_0},\gamma)
\]
is continuous (see \cite[Proposition 1.7.12]{GaN}). It only remains to check that~$g_\star$ is a right inverse to
$(f|_{U_0})_\star$. But $(f|_{U_0})_\star(g_\star(\gamma))(x)=
f(x,g(x,\gamma(x)))=f_x((f_x|_E)^{-1}(\gamma(x)))=\gamma(x)$
for all $\gamma\in C^\ell(U_0,Z)$ and $x\in U_0$,
entailing that $(f|_{U_0})_\star(g_\star(\gamma))=\gamma$
and thus $(f|_{U_0})_\star\circ g_\star=\id$.
\end{proof}
\begin{la}\label{lin-case}
Let $M$ be a manifold $($possibly with rough boundary$)$, such that~$M$ is smoothly paracompact.\footnote{Recall that a manifold is smoothly paracompact if it admits smooth partitions of unity, see \cite[Section 16]{KaM}.}
Let $\pi_E\colon E\to M$ be a $C^\ell$-vector bundle
over~$M$,
and $\pi_F\colon F\to M$
be a $C^\ell$-vector bundle over~$M$ whose fibers are Banach spaces.
Let $f\colon E\to F$ be a vector bundle map $($over $\id_M)$\footnote{Thus $f$ is a $C^\ell$-map
such that $f(E_x)\subseteq F_x$ for each $x\in M$ and $f|_{E_x}\colon E_x\to F_x$ is linear.}
of class~$C^\ell$
such that $f|_{E_x}\colon E_x\to F_x$ is a split linear surjection,
for each $x\in M$. If $\ell=0$ and~$F$ is not of finite rank, assume, moreover, that~$E$, $F$, and~$f$ are~$C^1$. Then also
\[
f_* \colon \Gamma_{C^\ell}(E)\to\Gamma_{C^\ell}(F),\;\,
\sigma\mto f\circ\sigma
\]
is a split linear surjection.
\end{la}
\begin{proof}
For each $x\in M$, there exists an open $x$-neighborhood
$U_x\subseteq M$ such that $E|_{U_x}$ and $F|_{U_x}$
are trivial. By Lemma~\ref{triv-case},
the continuous linear map
\[
(f|_{E|_{U_x}})_*\colon \Gamma_{C^\ell}(E|_{U_x})\to \Gamma_{C^\ell}(F|_{U_x}),\;\,
\sigma\mto f|_{E|_{U_x}}\circ\sigma
\]
admits a continuous linear right inverse $\rho_x\colon \Gamma_{C^\ell}(F|_{U_x})\to
\Gamma_{C^\ell}(E|_{U_x})$,
possibly after shrinking the open $x$-neighborhood $U_x\subseteq M$.
Let $(h_x)_{x\in M}$ be a smooth partition of unity on~$M$
such that $\Supp(h_x)\subseteq U_x$ for each $x\in M$.
Then
\[
\rho(\sigma):=\sum_{x\in M} (h_x\cdot \rho_x(\sigma|_{U_x}))\widetilde{\;}
\]
defines a linear right inverse $\rho\colon \Gamma_{C^\ell}(F)\to\Gamma_{C^\ell}(E)$
for~$f_*$ (where $\widetilde{\;}$ indicates the extension of the given
section to a global section of $E$ taking points outside $U_x$ to~$0\in E_x$).
Then~$\rho$
is continuous as each $x\in M$ has an open neighborhood
$V_x\subseteq M$ such that $\Phi_x:=\{y\in M\colon
V_x\cap \Supp(h_y)\not=\emptyset\}$ is finite, entailing that
\[
\Gamma_{C^\ell}(F)\to\Gamma_{C^\ell}(E|_{V_x}),\quad
\sigma\mto \rho(\sigma)|_{V_x}=\sum_{y\in \Phi_x} (h_y \cdot \rho_y(\sigma|_{U_y}))\widetilde{\;}
\]
is continuous (cf.\ Lemma~\ref{top-cover}).
\end{proof}
We are now in a position to prove Theorem E\,(a) whose statement we repeat here for the reader's convenience:

\begin{numba}[Theorem E (a)]
\emph{Let $f \colon M \rightarrow N$ be a $C^{\ell+k}$-submersion
and $N$ a Banach manifold, where $2\leq k\in \N\cup\{\infty\}$ and $\ell\in\N_0$. Then $C^\ell(K,f)$
is a $C^k$-submersion, assuming $\ell\geq 1$ if some modelling space of~$N$
is infinite-dimensional.}
\end{numba}

\begin{proof}[Proof of Theorem E\, (a)]
For each $\gamma\in C^\ell(K,M)$, the
map $g\colon \gamma^*(TM)\to (f\circ \gamma)^*(TN)$
taking $v\in \gamma^*(TM)_x=T_{\gamma(x)}(M)$ to
\[
T_{\gamma(x)}f(v)\in T_{f(\gamma(x))}N=(f\circ\gamma)^*(TN)_x
\]
is $C^\ell$ (as can be checked using local trivializations)\footnote{Let $\phi\colon U_\phi\to
V_\phi\subseteq F$
be a chart for~$N$, $\psi\colon U_\psi\to V_\psi\subseteq E$ be a chart for~$M$
with $f(U_\psi)\subseteq U_\phi$ and $W\subseteq K$ be an open subset such that
$\gamma(W)\subseteq U_\psi$. Then $\gamma^*(TM)|_W=\bigcup_{x\in W}\{x\}\times T_{\gamma(x)}M$
and the map
$\theta_\psi\colon \gamma^*(TM)|_W\to W\times E$, $(x,y)\mto (x,d\psi(y))$
is a local trivialization for $\gamma^*(TM)$. An analogous formula
yields a local trivialization $\theta_\phi\colon (f\circ \gamma)^*(TN)|_W\to W\times F$ of $(f\circ\gamma)^*(TN)$.  It remains to note that
$(\theta_\phi\circ g\circ \theta_\psi^{-1})(x,z)=
\big(x,d(\phi\circ f\circ \psi^{-1})(\psi(\gamma(x)),z)\big)$ is $C^\ell$ in $(x,z)$.}
and linear in $v\in \gamma^*(TM)_x$.
Moreover, $g(x,\cdot)$ corresponds to $T_{\gamma(x)}f$ (cf.\ (\ref{nearly}) in Corollary~\ref{formtang}),
whence it
is a split linear surjection. Now
\[
T_\gamma \, C^\ell(K,f)=g_*\colon \Gamma_{C^\ell}(\gamma^*(TM))\to \Gamma_{C^\ell}((f\circ\gamma)^*(TN))
\]
is a split linear surjection, by \ref{thetamp} and Lemma~\ref{lin-case}.
%
% here we use \ell \geq 1 if $TN$ is infinite-dimensional
%
The $C^k$-map $C^\ell(K,f)$ therefore is a na\"{\i}ve submersion in the sense of~\cite{SUB}.
As $\Gamma_{C^\ell}((f\circ\gamma)^*(TN))$
is a Banach space (cf.\ \cite[Section 3]{Wit}) and $k\geq 2$, we deduce from \cite[Theorem~A]{SUB}
that $C^\ell(K,f)$ is a submersion.\end{proof}

\begin{rem}
Note that the proof of Theorem E\, (a) used extensively the Banach manifold structure of $C^\ell (K,N)$ as we have established the submersion property by proving that the push-forward is a na\"{\i}ve submersion. 
Thus the proof will not generalise beyond the Banach setting, e.g.\ for $\ell=\infty$ or $K$ non-compact.
 
However, if $f \colon M \rightarrow N$ is a $C^\ell$-submersion between finite-dimensional manifolds the push-forward $f_* \colon C^\ell (K,M) \rightarrow C^\ell (K,N)$ can be proven to be a submersion in more general cases:
In \cite{AS17}, it was shown that for $\ell = \infty$ and $K$ (possibly non-compact and with smooth boundary or corners), the push-forward $f_*$ is a submersion. This theorem is known as the Stacey-Roberts Lemma. The proof can be generalised to $\ell \in \N_0$ (with $f$ being $C^{\ell+2}$). 
We note, however, that the results presented here are distinct from the Stacey-Roberts Lemma whose proof in \cite{AS17} does not generalise to infinite-dimensional target manifolds.
\end{rem}

\subsection*{Immersions between manifolds of mappings}\label{ec-imm}
\addcontentsline{toc}{subsection}{Immersions between manifolds of mappings}
\begin{la}\label{triv-case2}
Let $X$ and~$Z$ be locally convex spaces, $Y$ be a Banach space,
$U\subseteq X$ be a locally convex subset with dense interior
and $f\colon U\times Y\to Z$ be a $C^\ell$-map
such that $f_x:=f(x,\cdot)\colon Y\to Z$ is linear for each
$x\in U$. If $Y$ is infinite-dimensional and $\ell=0$, assume that~$f$ is~$C^1$.
If $x_0\in U$ such that $f_{x_0}$ admits a continuous linear left inverse,
then also
\[
(f|_{U_0})_\star\colon C^\ell(U_0,Y)\to C^\ell(U_0,Z),\;\, \gamma\mto f\circ (\id_{U_0},\gamma)
\]
admits a continuous linear left inverse, for an open $x_0$-neighborhood
$U_0\subseteq U$.
\end{la}
\begin{proof}
Let $\lambda\colon Z\to Y$ be a continuous linear map such that $\lambda\circ f_{x_0}=\id_Y$.
Then $\lambda_*:=\cL(Y,\lambda)\colon \cL(Y,Z)_b\to\cL(Y)$, $S\mto\lambda\circ S$
is a continuous linear map.
As in the proof of Lemma~\ref{triv-case}, we see that
$f^\vee\colon U\to\cL(Y,Z)_b$ is continuous, whence also
\[
h:=\lambda_*\circ f^\vee \colon U\to\cL(Y)_b,\;\, x\mto \lambda\circ f_x
\]
is continuous.
Since $\GL(Y)$ is open in $\cL(Y)_b$, there exists an open $x_0$-neighborhood
$U_0\subseteq U$ such that $h(U_0)\subseteq\GL(Y)$. As the inversion map
$\iota\colon \GL(Y)\to\GL(Y)$ is continuous,
we deduce that the map
\[
g\colon U_0\to \cL(Y)_b,\;\, x\mto (\lambda\circ f_x)^{-1}
\]
is continuous and hence also the map $g^\wedge\colon U_0\times Y\to Y$,
$(x,z)\mto g(x)(z)$ (as the evaluation map $\cL(Y)_b\times Y\to Y$ is continuous).
As $\lambda\circ f|_{U_0\times Y}$ is~$C^\ell$, we deduce that
$g^\wedge$ is~$C^\ell$ (see \cite[Lemma~2.3]{DiK}; cf.\ also \cite[Exercise 1.3.10]{GaN}, \cite[Theorem 5.3.1]{Ham}). Then also
\[
s\colon U_0\times Z\to Y,\quad s(x,z):=g^\wedge(x,\lambda(z))=(\lambda\circ f_x)^{-1}(\lambda(z))
\]
is~$C^\ell$, entailing that the linear map
\[
s_\star\colon C^\ell(U_0,Z)\to C^\ell(U_0,Y),\, \gamma\mto s\circ (\id_{U_0},\gamma)
\]
is continuous (see \cite[Proposition 1.7.12]{GaN}). It only remains to check that~$s_\star$ is a left inverse to
$(f|_{U_0})_\star$. But $s_\star((f|_{U_0})_*(\gamma))(x)=s(x,f(x,\gamma(x)))=(\lambda\circ f_x)^{-1}((\lambda\circ f_x)(\gamma(x)))=\gamma(x)$
for all $\gamma\in C^\ell(U_0,Y)$ and $x\in U_0$,
entailing that $s_\star((f|_{U_0})_\star(\gamma))=\gamma$
and thus $s_\star\circ (f|_{U_0})_\star=\id$.
\end{proof}
\begin{la}\label{lin-case2}
Let $M$ be a smooth manifold $($possibly with rough boundary$)$ which is smoothly paracompact, $\pi_E\colon E\to M$ be a $C^\ell$-vector bundle over~$M$ whose fibers are Banach spaces and $\pi_F\colon F\to M$
be a $C^\ell$-vector bundle over~$M$ whose fibers are locally convex spaces. Consider a bundle map
$f\colon E\to F$ of class~$C^\ell$
such that $f|_{E_x}\colon E_x\to F_x$ has a continuous linear left inverse,
for each $x\in M$. If $\ell=0$ and $E$ is not a finite rank bundle, assume, moreover, that~$E$, $F$, and~$f$ are~$C^1$. Then also
\[
f_* \colon \Gamma_{C^\ell}(E)\to\Gamma_{C^\ell}(F),\;\,
\sigma\mto f\circ\sigma
\]
has a continuous linear left inverse.
\end{la}
\begin{proof}
For each $x\in M$, there exists an open $x$-neighborhood
$U_x\subseteq M$ such that $E|_{U_x}$ and $F|_{U_x}$
are trivial. By Lemma~\ref{triv-case2},
the continuous linear map
\[
(f|_{E|_{U_x}})_*\colon \Gamma_{C^\ell}(E|_{U_x})\to \Gamma_{C^\ell}(F|_{U_x})
\]
admits a continuous linear left inverse $\lambda_x\colon \Gamma_{C^\ell}(F|_{U_x})\to
\Gamma_{C^\ell}(E|_{U_x})$,
possibly after shrinking the open $x$-neighborhood $U_x\subseteq M$.
Let $(h_x)_{x\in M}$ be a smooth partition of unity on~$M$
such that $\Supp(h_x)\subseteq U_x$ for each $x\in M$.
As in the proof of Lemma~\ref{lin-case}, we see that
\[
\Lambda(\sigma):=\sum_{x\in M} (h_x\cdot \lambda_x(\sigma|_{U_x}))\widetilde{\;}
\]
defines a continuous linear map $\Lambda \colon \Gamma_{C^\ell}(F)\to\Gamma_{C^\ell}(E)$.
It remains to check that $\Lambda$ is a left inverse for $f_*$.
Now, given $\tau\in \Gamma_{C^\ell}(E)$,
we have $\lambda_x(f_*(\tau)|_{U_x})=\lambda_x((f\circ\tau)|_{U_x})=\lambda_x((f|_{E|_{U_x}})_*(\tau|_{U_x}))=\tau|_{U_x}$ for all $x\in M$ and hence
\[
\Lambda(f_*(\tau))=\sum_{x\in M} (h_x\cdot \lambda_x(f_*(\tau)|_{U_x}))\widetilde{\;}
=\sum_{x\in M} (h_x\cdot \tau|_{U_x})\widetilde{\;}=\tau,
\]
which completes the proof.
\end{proof}
We now deduce Theorem E (b), which we repeat for the reader's convenience

\begin{numba}[Theorem E (b)]
\emph{Let $f \colon M \rightarrow N$ be an $C^{\ell +k}$-immersion with $\ell\in\N_0$ and
$2\leq k\in \N\cup\{\infty\}$, where $M$ is a Banach manifold, and $\ell \geq 1$ if some modelling space of $M$ is infinite-dimensional. Then $C^\ell (K,f)$ is a $C^k$-immersion.}
\end{numba}
\begin{proof}[Proof of Theorem E\,(b)]
For each $\gamma\in C^\ell(K,M)$, the
map $g\colon \gamma^*(TM)\to (f\circ \gamma)^*(TN)$
taking $v\in \gamma^*(TM)_x=T_{\gamma(x)}(M)$ to
\[
T_{\gamma(x)}f(v)\in T_{f(\gamma(x))}N=(f\circ\gamma)^*(TN)_x
\]
is $C^\ell$ 
and linear in $v\in \gamma^*(TM)_x$ (which can be verified as in the proof of Theorem~E\,(a)).
Moreover, $g(x,\cdot)$ corresponds to $T_{\gamma(x)}f$ for $x\in K$,
(cf.\ (\ref{nearly}) in Corollary~\ref{formtang}),
whence it has a continuous linear left inverse. Now
\[
T_\gamma \, C^\ell(K,f)=g_*\colon \Gamma_{C^\ell}(\gamma^*(TM))\to \Gamma_{C^\ell}((f\circ\gamma)^*(TN))
\]
has a continuous linear left inverse, by Lemma~\ref{lin-case2}.
The $C^k$-map $C^\ell(K,f)$ therefore is a na\"{\i}ve immersion in the sense of~\cite{SUB}.
As $\Gamma_{C^\ell}((f\circ\gamma)^*(TM))$
is a Banach space and $k\geq 2$, we deduce from \cite[Theorem~H]{SUB}
that $C^\ell(K,f)$ is an immersion.
\end{proof}

\subsection*{Local diffeomorphisms between manifolds of mappings}\label{sec-eta}  \addcontentsline{toc}{subsection}{Local diffeomorphisms between manifolds of mappings}
Let us turn to local diffeomorphisms between manifolds of mappings. It turns out that this property can be established immediately using some topological data.

\begin{numba}[Theorem E\, (c)]\label{proofDc}\emph{
 If $f \colon M \rightarrow N$ is a local $C^{k+\ell}$ diffeomorphism and $M$ is a regular topological space, then $C^\ell (K,f)$ is a local $C^k$-diffeomorphism.}
\end{numba}
\begin{proof}[Proof of Theorem E (c)]
Let $\gamma\in C^\ell(K,M)$. For each $y\in M$,
there exists an open $y$-neighborhood $U_y\subseteq M$ such that
$f(U_y)$ is open in~$N$ and $f|_{U_y}\colon U_y\to f(U_y)$ is a $C^{k+\ell}$-diffeomorphism.
By \cite[Lemma~2.1]{DGS},
there exists $n\in\N$ and open subsets $V_1,\ldots, V_n\subseteq M$
such that $\gamma(K)\subseteq V_1\cup\cdots\cup V_n$ and
\[
(\forall i,j\in\{1,\ldots, n\})\;\,
V_i\cap V_j\not=\emptyset \impl (\exists y\in M)\; V_i\cup V_j \subseteq U_y.
\]
Each $x\in K$ has an open neighborhood $W_x\subseteq K$ whose closure
$L_x:=\wb{W_x}$ is contained in $\gamma^{-1}(V_{i(x)})$ for some $i(x)\in\{1,\ldots, n\}$.
Since $K$ is compact, there exist $m\in\N$ and $x_1,\ldots, x_m\in K$
such that $K=W_{x_1}\cup\cdots\cup W_{x_m}$.
Then
\[
P:=\{\eta\in C^\ell(K,M)\colon (\forall k\in\{1,\ldots,m\})\; \eta(L_{x_k})\subseteq V_{i(x_k)}\}
\]
is an open neighborhood of~$\gamma$ in $C^\ell(K,M)$. The restriction $C^\ell(K,f)|_P$
is injective, as we now verify:
Let $\eta_1,\eta_2\in P$ such that $f\circ \eta_1=f\circ \eta_2$.
Given $x\in K$, we have $x\in L_{x_k}$ for some $k\in \{1,\ldots, m\}$.
Since $\eta_1(x),\eta_2(x)\in V_{i(x_k)}$
and $f|_{V_{i(x_k)}}$ is injective, we deduce from $f(\eta_1(x))=f(\eta_2(x))$
that $\eta_1(x)=\eta_2(x)$. Thus $\eta_1=\eta_2$.\\[2.3mm]
The set $\Omega:=\bigcup_{k=1}^m (W_{x_k}\times f(V_{i(x_k)}))$
is open in $K\times N$ and contains the graph of $f\circ\gamma$.
Moreover,
\[
Q:=\{\theta\in C^\ell(K,N)\colon (\forall k\in\{1,\ldots,m\})\; \theta(L_{x_k})\subseteq f(V_{i(x_k)})\}
\]
is an open neighborhood of $f\circ\gamma$ in $C^\ell(K,N)$.
We claim that the map
\[
g\colon \Omega\to M,\;\, (x,z)\mto (f|_{V_{i(x_k)}})^{-1}(z)\;\,\mbox{if $x\in L_{x_k}$
and $z\in f(V_{i(x_k)})$}
\]
is well defined. If this is true, we observe that~$g$ is $C^{k+\ell}$ on each of the sets
$W_{x_k}\times f(V_{i(x_k)})$ which form an open cover for~$\Omega$,
whence $g$ is~$C^{k+\ell}$. As a consequence, the map
\[
g_\star\colon Q \to C^\ell(K,M),\;\, \theta\mto g\circ (\id_K,\theta)
\]
is~$C^k$. By construction, $g_\star(Q)\subseteq P$ and
\begin{equation}\label{thusbij}
C^\ell(K,f)|_P\circ g_\star=\id_Q,
\end{equation}
whence $C^\ell(K,f)(P)\supseteq Q$ and thus $C^\ell(K,f)(P)=Q$
(the converse inclusion being obvious). As $C^\ell(K,f)|_P$ is injective,
we see that $C^\ell(K,f)|_P\colon P\to Q$
is a bijection. We now infer from~(\ref{thusbij})
that $(C^\ell(K,f)|_P)^{-1}=g_*$, which is~$C^k$.
Thus $C^\ell(K,f)|_P\colon P\to Q$ is a $C^k$-diffeomorphism.\\[2.3mm]
It only remains to verify the claim. If $k,h\in\{1,\ldots, m\}$
such that $x\in L_{x_k}\cap L_{x_h}$
and $z\in f(V_{i(x_k)})\cap f(V_{i(x_h)})$,
then $\gamma(x)\in V_{i(x_k)}\cap V_{i(x_h)}$,
whence there exists $y\in Y$ such that
\[
V_{i(x_k)}\cup V_{i(x_h)}\subseteq U_y.
\]
Thus
$(f|_{V_{i(x_k)}})^{-1}(z)=(f|_{U_y})^{-1}(z)=(f|_{V_{i(x_h)}})^{-1}(z)$.
\end{proof}
We mention a variant of Theorem~E (c)
for spaces of continuous mappings between topological spaces.
\begin{prop}\label{map-eta}
Let $K$ be a compact Hausdorff topological space, $M$ and $N$ be Hausdorff
topological spaces and $f\colon M\to N$ be a local homeomorphism.
If $M$ is a regular topological space, then
\[
C(K,f)\colon C(K,M)\to C(K,N),\;\, \gamma\mto f\circ\gamma
\]
is a local homeomorphism.
\end{prop}
\begin{proof}
The proof is analogous to that of \ref{proofDc},
except that $C^\ell(K,M)$ and $C^\ell(K,N)$
have to be replaced with $C(K,M)$ and $C(K,N)$, respectively. Furthermore, 
the words ``$C^k$-diffeomorphism'' and ``$C^{k+\ell}$-diffeomorphism''
have to be replaced with ``homeomorphism'',
and the properties ``$C^k$'' and ``$C^{k+\ell}$'' have to be replaced with
continuity.
\end{proof}
In light of the results in this section we can now adapt a classical result by Michor to our setting .

\begin{prop}\label{prop:emb}
Let $\iota \colon M \rightarrow N$ be a $C^{\ell+k}$ embedding between manifolds modelled on locally convex spaces with local addition, where $\ell,k \in \N_0 \cup \{\infty\}$ and we assume one of the following:
\begin{enumerate}
\item both $M$ and $N$ are finite dimensional (no restriction on $\ell,k$),
\item $M$ is a Banach manifold whose model space is infinite-dimensional and $1 \leq \ell < \infty$, $k\geq 2$.
\end{enumerate}
Then for every compact manifold $K$ (possibly with rough boundary), 
$$\iota_* \coloneq C^\ell (K,f) \colon C^\ell (K,M) \rightarrow C^\ell (K,N),\quad g \mapsto f\circ g$$
is a $C^k$-embedding.  
\end{prop}

\begin{proof}
We only have to establish the first case as the second case follows by combining Lemma \ref{embpfwd} and Theorem E (b) (cf.\ introduction). Now, due to Corollary \ref{la-reu} the push-forward $\iota_*$ is a $C^k$-map which is clearly injective with image $\iota_* (C^\ell (K,M)) = C^\ell (K,\iota(M))$ (as sets).
 Considering $\iota(M)$ as a submanifold of $N$, Theorem E\, (c) shows that $\iota_*$ corestricts to a $C^k$-diffeomorphism $\iota_*|^{C^\ell (K,\iota (M))} \colon C^\ell (K,M) \rightarrow C^\ell (K,\iota(M))$.
Arguing as in the proof of \cite[Proposition 10.8]{Mic}\footnote{In loc.cit. $\ell = \infty$ is assumed. However, the proof generalises verbatim to $\ell \in \N_0$ due to the canonical manifold structure on $C^\ell (K,M)$ and $C^\ell (K,N)$ from Appendix \ref{map-mfd}.} to prove that $C^\ell (K,\iota(M)) $ is a split submanifold of $C^\ell (K,N)$.
Then \cite[Lemma 1.13]{SUB} entails that $\iota_*$ is a $C^k$-embedding.
\end{proof}
\subsection*{Proper maps between manifolds of maps}\label{sec-pro}
  \addcontentsline{toc}{subsection}{Proper maps between manifolds of maps}
In this section, we investigate conditions under which the pushforward of a proper map yields a proper map between manifolds of mappings. To this end, we recall first:

\begin{defn}\label{defn: proper}
Consider a continuous map $f\colon X\to Y$
between Hausdorff topological spaces. Then $f$ is called
\begin{itemize}
\item[(a)] \emph{proper} if $f^{-1}(K)$ is a compact subset of~$X$
for each compact subset $K\subseteq Y$ (see \cite{Pal}).
\item[(b)] \emph{perfect} if~$f$ is a closed map and
$f^{-1}(\{y\})$ is a compact subset of~$X$ for each $y\in Y$
(see \cite[p.\,182]{Eng}).
\end{itemize}
\end{defn}
Every perfect map is proper (see \cite[Theorem 3.7.2]{Eng}).
If $Y$ is a $k$-space,\footnote{A Hausdorff topological space~$Y$
is called a \emph{$k$-space} if subsets $A\subseteq Y$ are closed
if and only if $A\cap K$ is closed for each compact subset $K\subseteq Y$.}
  then a continuous map $f\colon X\to Y$ is proper
if and only if it is perfect (as proper maps to $k$-spaces are closed
mappings, see~\cite{Pal}). Note that
every manifold (possibly with rough boundary) modelled
on a metrizable locally convex space is a $k$-space
(notably every Banach manifold).

In contrast to our findings concerning submersions and immersions,
the push-forward of proper maps will in general not be a proper map as the next example shows.

\begin{exa}\label{exa:nonproper}
Let $\{\star\}$ be the one-point manifold and consider the (smooth) map from the circle $f\colon \Sph \rightarrow \{\star\}$. Then $f$ is proper as $\Sph$ is compact. Now  $C^\ell (\Sph,\{\star\}) = \{\star\}$ and we observe that $f_* \colon C^\infty (\Sph , \Sph) \rightarrow \{\star\}$ cannot be proper as $C^\infty(\Sph , \Sph)$ is an infinite-dimensional manifold (hence non-compact). 
\end{exa}

However, properness of the push-forward is preserved under additional assumptions. 
Three lemmas will be useful for these discussions.

\begin{la}\label{hence-Ck}
Let $M$, $N$, and $L$ be $C^k$-manifolds $($possibly
with rough boundary$)$
and $q\colon M\to N$ be a local $C^k$-diffeomorphism.
If $f\colon L\to M$ is a continuous map such that $q\circ f$ is~$C^k$,
then~$f$ is~$C^k$.
\end{la}
\begin{proof}
Given $x\in L$, let $V$ be an open neighborhood of~$f(x)$ in~$M$ such that
$q(V)$ is open in~$N$ and $q|_V\colon V\to q(V)$ is a $C^k$-diffeomorphism.
Let $U\subseteq L$ be an open $x$-neighborhood such that $f(U)\subseteq V$.
Then $f|_U=(q|_V)^{-1}\circ (q\circ f)|_U$ is~$C^k$.
\end{proof}
\begin{la}\label{equaliza}
Let $X$ be a connected topological space, $Y$ be a Hausdorff topological space
and $q\colon Y\to Z$ be a locally injective map to a set~$Z$.
Let $f\colon X\to Y$ and $g\colon X\to Y$ be continuous mappings such that
\[
q\circ f=q\circ g.
\]
If $f(x_0)=g(x_0)$ for some $x_0\in X$,
then $f=g$.
\end{la}
\begin{proof}
The subset $E:=\{x\in X\colon f(x)=g(x)\}$ of~$X$ is nonempty by hypothesis
and closed as $Y$ is Hausdorff and both $f$ and $g$ are continuous.
If $x\in E$, let~$V$ be a neighborhood of $f(x)=g(x)$ in~$Y$ such that
$q|_V$ is injective. By continuity of~$f$ and~$g$,
there is an $x$-neighborhood $U\subseteq X$ such that $f(U)\subseteq V$ and
$g(U)\subseteq V$. For each $y\in U$, we deduce from $q|_V(f(y))=q|_V(g(y))$
that $f(y)=g(y)$, whence $U\subseteq E$ and~$E$ is open. As~$X$ is connected, $E=X$ follows.
\end{proof}
If $X$ is a set and $V\subseteq X\times X$ is a set containing the diagonal
$\Delta_X:=\{(x,x)\colon x\in X\}$, we set $V[x]:=\{y\in X\colon (x,y)\in V\}$.
\begin{la}\label{like-L-del}
Let $X$ be a topological space and $(U_j)_{j\in J}$
be an open cover of~$X$.
Let $(A_j)_{j\in J}$ be a locally finite cover of~$X$
by closed subsets~$A_j$ of~$X$ such that $A_j\subseteq U_j$ for each $j\in J$.
Then there exists a neighborhood~$V$ of the diagonal $\Delta_X$
in $X\times X$ such that $V[x]\subseteq U_j$
for all $j\in J$ and $x\in A_j$.
\end{la}
\begin{proof}
We set
\[
V:=\bigcup_{x\in X}\left(\{x\}\times \bigcap_{j\in J : x\in A_j}U_j \right).
\]
To see that~$V$ is a neighborhood of~$\Delta_X$, let $x\in X$.
Since $(A_j)_{j\in J}$ is locally finite, the union
\[
\bigcup_{j\in J\colon x\not\in A_j}A_j
\]
of closed sets is closed. Thus
\[
W:=X\setminus \bigcup_{j\in J\colon x\not\in A_j}A_j=\bigcap_{j\in J: x\not\in A_j} (X\setminus A_j)
\]
is an open neighborhood of~$x$ in~$X$.
If $w\in W$, then
$\{j\in J \colon w\in A_j\}\subseteq \{j\in J\colon x\in A_j\}$,
whence $V$ contains the $(x,x)$-neighborhood
\[
W\times \bigcap_{j\in J: x \in A_j}U_j.
\]
Thus~$V$ is a neighborhood of~$\Delta_X$ in $X\times X$. It remains to observe
that $V[x]\subseteq U_j$
for each $j\in J$ such that $x\in A_j$, by definition of~$V$.
\end{proof}
\begin{prop}\label{ess-prop}
Let $S$ be a Hausdorff topological space which is a $k$-space
and admits a cover $(K_i)_{i\in I}$
of compact, locally connected subsets~$K_i$ such that each compact subset $K\subseteq S$
is contained in $\bigcup_{i\in\Phi}K_i$ for some finite subset $\Phi\subseteq I$. 
Let
$X$, $Y$, and~$Z$ be Hausdorff topological spaces,
$\alpha\colon X\to Y$ be a local homeomorphism
and $\beta\colon X\to Z$ be a continuous map such that
$(\alpha,\beta)\colon X\to Y\times Z$ is a proper map.
We assume that the topological space~$X$ is regular.
Consider the local homeomorphism
\[
\alpha_*:=C(S,\alpha) \colon C(S,X)\to C(S,Y),\;\, \gamma\mto \alpha\circ\gamma
\]
and the continuous map $\beta_*:=C(S,\beta)\colon C(S,X)\to C(S,Z)$, $\gamma\mto \beta\circ \gamma$.
Then
$g:=(\alpha_*,\beta_*)\colon C(S,X)\to C(S,Y)\times C(S,Z)$,
$\gamma\mto (\alpha\circ\gamma,\beta\circ \gamma)$
is proper.
\end{prop}
\begin{rem}
(a) If $S$ is any finite-dimensional manifold (possibly with rough boundary)
which is locally compact (which is automatic if~$S$ has no boundary),
then~$S$ admits a cover $(K_i)_{i\in I}$ as described in
Proposition~\ref{ess-prop}.\medskip

\noindent
(b) Consider an ascending sequence $S_1\subseteq S_2\subseteq\cdots$
of finite-dimensional manifolds (possibly with rough boundary) which are locally compact,
such that each inclusion map $S_n\to S_{n+1}$ is a topological embedding.
Endow $S:=\bigcup_{n\in\N}S_n$ with the direct limit topology.
Then $S$ admits a cover $(K_i)_{i\in I}$ as in Proposition~\ref{ess-prop}
(composed of those of the $S_n$, as in~(a)).
For example, this applies to $S:=\bigoplus_{n\in\N}\R=\dl\,\R^n$ with $K_n:=[{-n},n]^n$
for~$n\in\N$.
\end{rem}
\begin{proof}[Proof of Proposition~\ref{ess-prop}]
We start with the special case that $K:=S$ is compact.
Let $L\subseteq C(K,Y)\times C(K,Z)\sim C(K,Y\times Z)$ be a compact set;
we have to show that $g^{-1}(L)$ is compact.
As the evaluation map $\ve\colon C(K,Y\times Z)\times K\to Y\times Z$
is continuous,
\[
\{\gamma(x)\colon \gamma\in L, x\in K\}=\ve(L\times K)
\]
is a compact subset of~$Y\times Z$. Let $C$ be a compact subset of $Y\times Z$
which contains $\ve(L\times K)$. Since $(\alpha,\beta)$ is proper,
$B:=(\alpha,\beta)^{-1}(C)$ is a compact subset of~$X$.
Then
\[
g^{-1}(L)\subseteq C(K,B);
\]
in fact, for $\theta\in g^{-1}(L)$ and $x\in K$
we have $g(\theta)=(\alpha\circ\theta,\beta\circ \theta)\in L$,
whence $(\alpha,\beta)(\theta(x))=g(\theta)(x)\in C$
and thus $\theta(x)\in B$.
Note that $C':=(\alpha,\beta)(B)$ is a compact subset of~$C$ and $B=(\alpha,\beta)^{-1}(C')$.
Also, $g^{-1}(L)=g^{-1}(L')$ using the compact set
$L':=\{\gamma\in L\colon (\forall x\in K)\;\gamma(x)\in C'\}$.
Finally, $\ve(L'\times K)\subseteq C'$.

In view of Ascoli's Theorem,
$g^{-1}(L)$ will be compact if we can show that $g^{-1}(L)\subseteq C(K,B)$
is equicontinuous (with respect to the unique uniform structure on
the compact Hausdorff space~$B$ which is compatible with its topology).
Let~$W$ be a neighborhood of~$\Delta_B$ in $B\times B$.
For each $b\in B$, we find an open $b$-neighborhood
$U_b$ in~$X$ such that $(U_b\times U_b)\cap (B\times B)\subseteq W$ holds,
$\alpha(U_b)$ is open in~$Y$, and $\alpha|_{U_b}\colon U_b\to\alpha(U_b)$
is a homeomorphism.
For $b\in B$, let $A_b$ be a compact neighborhood of~$b$ in~$B$
such that $A_b\subseteq U_b$. By compactness of~$B$, there exists a finite subset $I\subseteq B$
such that $B=\bigcup_{b\in I}A_b$.

Let $\pr_1\colon Y\times Z\to Y$, $(y,z)\mto y$ be the projection onto the first component.
Then $D:=\pr_1(C')=\pr_1((\alpha,\beta)(B))=\alpha(B)$ is a compact subset
of~$Y$, and $(D\cap \alpha(U_b))_{b\in I}$ is a finite open cover of the compact
topological space~$D$. Moreover, the compact sets $L_b:=\alpha(A_b)$
cover~$D$ for $b\in I$ and $L_b\subseteq D\cap \alpha(U_b)$ holds for all $b\in I$.

By Lemma~\ref{like-L-del}, there exists
a neighborhood~$V$ of~$\Delta_D$ in $D\times D$ such that $V[z]\subseteq \alpha(U_b)\cap D$
for each $z\in D$ and each $b\in I$ such that $z\in L_b$.
Since $\{\pr_1\circ \,\gamma\colon\gamma\in L'\}=C(K,\pr_1)(L')\subseteq C(K,D)$ is compact and hence 
equicontinuous, each $x\in K$ has a neighborhood $Q\subseteq K$ such that
\[
\gamma_1(Q)\subseteq V[\gamma_1(x)]
\]
for each $\gamma=(\gamma_1,\gamma_2)\in L'$.
After shrinking~$Q$, we may assume that~$Q$ is connected.
If $\eta\in g^{-1}(L)=g^{-1}(L')$,
then $\gamma:=g(\eta)\in L'$ with $\gamma_1:=\pr_1\circ\,\gamma=\alpha\circ \eta$.
We have $\eta(x)\in B$ and thus $\eta(x)\in A_b$ for some $b\in I$.
Then $\gamma_1(x)=\alpha(\eta(x))\in \alpha(A_b)=L_b$,
whence $V[\gamma_1(x)]\subseteq \alpha(U_b)\cap D$.
Now $\zeta:=(\alpha|_{U_b})^{-1}\circ \gamma_1|_Q$
and $\eta|_Q$ are continuous maps $Q\to X$ such that $\alpha\circ \zeta=\gamma_1|_Q=
\alpha\circ \eta|_Q$ and $\zeta(x)=\eta|_Q(x)$ as $\eta(x),\zeta(x)\in U_b$
and $\alpha|_{U_b}$ is injective. Hence $\eta|_Q=\zeta$, by Lemma~\ref{equaliza}.
Thus $\eta(Q)=(\alpha|_{U_b})^{-1}(\gamma_1(Q))\subseteq U_b$
and hence $\{\eta(x)\}\times \eta(Q)\subseteq A_b\times U_b$.
As $\eta(K)\subseteq B$, we deduce that $\{\eta(x)\}\times\eta(Q)\subseteq (U_b\cap B)\times (U_b\cap B)\subseteq W$
and thus $\eta(Q)\subseteq W[\eta(x)]$. Hence $g^{-1}(L)$ is equicontinuous.

The general case: Assume now that~$S$ is a $k$-space
admitting a family $(K_i)_{i\in I}$ as specified in the proposition.
Since~$S$ is a $k$-space, we have
\[
C(S,X)=\pl_{K\in\cK(S)} C(K,X)
\]
as a topological space, where $\cK(S)$ is the set of compact subsets of~$S$
(directed under inclusion). The limit maps are the restriction maps
\[
\rho_K\colon C(S,X)\to C(K,X),\;\, \gamma\mto\gamma|_k.
\]
Let $\cL:=\{\bigcup_{i\in \Phi}K_i\colon \mbox{$\Phi\subseteq I$, $\Phi$ finite}\}$.
Then $\cL$ is cofinal in~$\cK(S)$, whence
\[
C(S,X)=\pl_{K\in\cL}C(K,X).
\]
As a consequence,
\[
\rho:=(\rho_K)_{K\in\cL}\colon C(S,X)\to\prod_{K\in\cL}C(K,X),\;\,\gamma\mto(\gamma|_K)_{K\in\cL}
\]
is a topological embedding onto a closed subset.
Given $K\in\cL$, there is a finite subset $\Phi_K\subseteq I$ such that $K=\bigcup_{i\in\Phi_K}K_i$.
The map
\[
\sigma_K\colon C(K,X)\to\prod_{i\in\Phi_K}C(K_i,X),\;\,
\gamma\mto (\gamma|_{K_i})_{i\in \Phi_K}
\]
is continuous, injective, and its image is the set
\[
\{(\gamma_i)_{i\in\Phi_K}\colon (\forall i,j\in\Phi_K)\;\gamma_i|_{K_i\cap K_j}=\gamma_j|_{K_i\cap K_j}\}
\]
(by the Glueing Lemma), which is closed in $\prod_{i\in \Phi_K}C(K_i,X)$.
If $H\subseteq K$ is compact and $O\subseteq X$ an open set, then
$\gamma\in C(K,X)$ satisfies
$\gamma(H)\subseteq O$ if and only if $\gamma|_{K_i}$ satisfies $\gamma|_{K_i}(K_i\cap H)\subseteq O$
for all $i\in\Phi_K$. As a consequence, $\sigma_K$ is a topological embedding.
Thus
\[
\Big(\prod_{K\in\cL}\sigma_K\Big)\circ \rho\, \colon \, C(S,X)\to\prod_{K\in\cL}\prod_{i\in\Phi_K}C(K_i,X)
\]
is a topological embedding with closed image. Let $L\subseteq C(S,Y\times Z)$ be a compact set.
Then $g^{-1}(L)$ is closed in $C(S,X)$. By the preceding, $g^{-1}(L)$
will be compact if we can show that $\rho_{K_i}(g^{-1}(L))$ is relatively compact
in $C(K_i,X)$ for each $i\in I$.
We now use that $R_i\colon C(X,Y\times Z)\to C(K_i,Y\times Z)$, $\gamma\mto \gamma|_{K_i}$
is continuous, whence $R_i(L)$ is compact. From the above special case, we know that
the map
\[
g_i\colon C(K_i,X)\to C(K_i,Y\times Z),\;\, \gamma\mto (\alpha\circ \gamma,\beta\circ\gamma)
\]
is proper, whence $g_i^{-1}(R_i(L))$ is compact in $C(K_i,X)$.
Now $g_i\circ \rho_{K_i}=R_i\circ g$.
To see that $\rho_{K_i}(g^{-1}(L))$
is relatively compact, it only remains to note that $\rho_{K_i}(g^{-1}(L))\subseteq g_i^{-1}(R_i(L))$
as $g_i(\rho_{K_i}(g^{-1}(L)))=R_i(g(g^{-1}(L)))\subseteq R_i(L)$.
\end{proof}
After these preparations, we are now in a position to prove Theorem E (d), which we repeat here for the reader's convenience.

\begin{numba}[Theorem E\, (d)]\emph{If $f\colon M \rightarrow N$ is a proper $C^{k+\ell}$-map, $M$ is a regular topological space and $N=N_1\times N_2$
with smooth manifolds~$N_1$ and~$N_2$ such that~$N_1$ admits a local addition
and
$\pr_1\circ\, f\colon M\to N_1$ is a local $C^{k+\ell}$-diffeomorphism, then $C^\ell(K,f)$ is proper.}
\end{numba}

\begin{proof}[Proof of Theorem E (d)]
Abbreviate $g:=C^\ell(K,f)$.
Let $L\subseteq C^\ell(K,N)$ be a compact subset; we have to verify that $g^{-1}(L)\subseteq C^\ell(K,M)$
is compact. We show that each net $(\gamma_j)_{j\in J}$ in $g^{-1}(L)$ has
a convergent subnet in $g^{-1}(L)$. Since $g(\gamma_j)\in L$ and~$L$ is compact,
after passage to a subnet we may assume that $g(\gamma_j)\to\eta$ in $C^\ell(K,N)$
for some $\eta\in L$. Write $\eta=(\eta_1,\eta_2)$ with $\eta_i\in C^\ell(K,N_i)$ for
$i\in\{1,2\}$. The map
\[
h\colon C(K,M)\to C(K,N),\;\, \gamma\mto(\alpha\circ \gamma,\beta\circ\gamma)
\]
is proper, by Proposition~\ref{ess-prop}.
Since $L$ also is a compact subset of $C(K,N)$,
\[
C:=h^{-1}(L)=\{\gamma\in C(K,M)\colon (\alpha\circ\gamma,\beta\circ\gamma)\in L\}
\]
is compact in $C(K,M)$. Hence, after passage to a subnet we may assume that
$\gamma_j$ converges to some $\gamma\in C$
with respect to the compact-open topology on $C(K,M)$.
As~$h$ is continuous, we must have $h(\gamma_j)\to h(\gamma)$ in $C(K,N)$.
But $h(\gamma_j)=g(\gamma_j)\to\eta$ in $C^\ell(K,N)$ and hence also in $C(K,N)$.
Hence $h(\gamma)=\eta\in L\subseteq C^\ell(K,N_1)\times C^\ell(K,N_2)$,
whence $\alpha\circ\gamma =\pr_1\circ \, h(\gamma)=\pr_1\circ\, \eta\in C^\ell(K,N_1)$.
Using Lemma~\ref{hence-Ck}, we deduce that $\gamma \in C^\ell(K,M)$.

We  show that $\gamma_j\to\gamma$ in $C^\ell(K,M)$:
Given $x\in K$, let $W\subseteq M$ be an open $\gamma(x)$-neighborhood
such that $\alpha(W)\subseteq N_1$ is open and $\alpha|_W\colon W\to\alpha(W)$ is
a $C^\ell$-diffeomorphism. Let $\phi\colon U_\phi\to V_\phi\subseteq E$
be a chart of~$K$ around~$x$ and $K_x:=\phi^{-1}(C_x)$,
where $C_x$ is a compact convex $\phi(x)$-neighborhood in~$V_\phi$;
choosing~$C_x$ small enough, we may assume that $\gamma(K_x)\subseteq W$
and thus $\eta_1(K_x)\subseteq \alpha(W)$. Let~$U_x$ be the interior of $K_x$ in~$K$.
We then get a continuous restriction map
\[
\rho_{K_x}\colon C^\ell(K,N_1)\to C^\ell(K_x,N_1),\;\, \theta\mto \theta|_{K_x}.
\]
% (!)
There exists an index $j_x$ such that $\gamma_j(K_x)\subseteq
W$ for all $j\geq j_x$. Then
\[
\gamma_j|_{K_x}=(\alpha|_W)^{-1}\circ (\alpha\circ\gamma_j)|_{K_x}
=C^\ell(K_x,(\alpha|_W)^{-1})(\rho_{K_x}(\alpha\circ \gamma_j)),
\]
which converges to $C^\ell(K_x,(\alpha|_W)^{-1})(\rho_{K_x}(\eta_1))=
(\alpha|_W)^{-1}\circ\eta_1|_{K_x}$ in $C^\ell(K_x,M)$ as $\alpha_1\circ\gamma_j$
converges to $\eta_1$ in $C^\ell(K,N_1)$.
In particular, $\gamma_j|_{K_x}\to (\alpha|_W)^{-1}\circ \eta_1|_{K_x}$ in $C(K_x,M)$.
Since also $\gamma_j|_{K_x}\to \gamma|_{K_x}$ in $C(K_x,M)$, we deduce that
$\gamma|_{K_x}=(\alpha|_W)^{-1}\circ\eta_1|_{K_x}$. Hence
\[
\gamma_j|_{K_x}\to \gamma|_{K_x}\;\,\mbox{in $C^\ell(K_x,M)$,}
\]
whence $\gamma_j|_{U_x}\to \gamma|_{U_x}$ in $C^\ell(U_x,M)$
a fortiori. As $(U_x)_{x\in K}$ is an open cover of~$K$, we deduce with Lemma~\ref{top-cover}
that $\gamma_j\to\gamma$ in $C^\ell(K,M)$. The proof is complete.
\end{proof}
\section{Current groupoids}\label{sec:currentgpd}
In this section, we deal with the Lie groupoids of mappings on
a manifold with values in a (possibly infinite-dimensional) manifold. These were defined in the introduction and we briefly recall the construction and prove Theorems~A--C.

\begin{numba}[Current groupoids]
We let $\mathcal{G} = (G\toto M)$ be a Lie groupoid, $K$ be a compact manifold and $\ell \in \N_0 \cup \{\infty\}$. 
Define now the \emph{current groupoid} $C^\ell (K,\mathcal{G})$ as the groupoid given by the following data 
\begin{itemize}
\item $C^\ell (K,G)$ (space of arrows), $C^\ell (K,M)$ (space of units)
\item pointwise groupoid operations, i.e.\ the pushforwards of the groupoid maps $\alpha_*$, $\beta_*$, $m_*$, $\iota_*$ and $\mathbf{1}_*$ are the source, target, multiplication, inversion and unit maps.
\end{itemize}
\end{numba}

Clearly a current groupoid is a groupoid. The following theorem, which encompasses Theorem A of the introduction, will now establish that current groupoids are Lie groupoids.

\begin{thm}\label{thm:A:ext}
Let $\mathcal{G} = (G\toto M)$ be a Lie groupoid, where~$M$ is a Banach manifold. Fix a compact manifold $K$ (possibly with rough boundary),
and $\ell\in\N_0 \cup \{\infty\}$. If $\ell=0$, assume that all modelling spaces of~$M$ are finite-dimensional and if $\ell =\infty$ we assume in addition that all modelling spaces of $G$ are finite-dimensional.
If~$G$ and~$M$ admit a local addition,
then the current groupoid $C^\ell(K,\mathcal{G})$ is a Lie groupoid.
\end{thm}
\begin{proof}
As we assume that $G$ and $M$ admit a local addition,
$C^\ell(K,G)$ and $C^\ell(K,M)$ admit canonical smooth manifold structures.
By Theorem~E (a) (or for $\ell = \infty$ by the Stacey-Roberts Lemma \cite[Lemma 2.4]{AS17}),
the mappings
\[
\alpha_*:=C^\ell(K,\alpha)\colon C^\ell(K,G)\to C^\ell(K,M)
\]
and $\beta_*:=C^\ell(K,\beta)$ are submersions.
As a consequence, the fiber product
\[
C^\ell(K,G)^{(2)}:=\{(\gamma_1,\gamma_2)\in C^\ell(K,G)\times C^\ell(K,G)\colon
\alpha_*(\gamma_1)=\beta_*(\gamma_2)\}
\]
is a submanifold of $C^\ell(K,G)\times C^\ell(K,G)$ (see \cite[Theorem~B]{SUB}).
Now $C^\ell(K,G)^{(2)}$ $=C^\ell(K,G^{(2)})$ as a set, which
enables a groupoid multiplication on $C^\ell(K,G)$ to be defined via
\[
\mu_*\colon C^\ell(K,G)^{(2)}\to C^\ell(K,G),\;\, (\gamma_1,\gamma_2)\mto
\mu\circ(\gamma_1,\gamma_2),
\]
where $\mu\colon G^{(2)}\to G$ is the smooth multiplication in the groupoid~$G$.
By Lemma~\ref{la-reu}, $\mu_*$ is smooth.
Since~$\mathcal{G}$ is a Lie groupoid, the map $\mathbf{1}\colon M\to G$,
$x\mto e_x\in G_x$ is smooth. Then $e_\gamma:=\mathbf{1}\circ \gamma$ is the neutral element in
$C^\ell(K,G)_\gamma$ for $\gamma\in C^\ell(K,M)$, and the map
$C^\ell(K,\mathbf{1})\colon C^\ell(K,M)\to C^\ell(K,G)$, $\gamma\mto e_\gamma=\mathbf{1}\circ \gamma$
is smooth. As the inversion map $\iota\colon G\to G$
is smooth, also $\iota_*\colon C^\ell(K,G)\to C^\ell(K,G)$ is smooth.
Thus $C^\ell(K,G)$ is a Lie groupoid.
\end{proof}

\begin{exa}
Recall that a locally convex Lie group $G$ can be made into a Lie groupoid $G \toto \{\star\}$ over the one-point manifold $\{\star\}$ (which is trivially a Banach manifold). Note that $C^\ell (K,\{\star\}) = \{\star\}$. 
Moreover, $G$ admits a local addition \cite[42.4]{KaM} whence Theorem A yields a current groupoid $C^\ell (K, G) \toto \{\star\}$ which can be canonically identified with the current group $C^\ell (K,G)$ from \cite{NaW}.
Thus for the circle $K = \Sph$ our construction recovers the loop groups from \cite{PaS}.
\end{exa}

\begin{exa}\label{exa:unitgpd}
Let $M$ be a smooth manifold with local addition. 
Recall the following groupoids associated to $M$:
\begin{itemize}
\item the \emph{unit groupoid} $\mathfrak{u}(M) = (M\toto M)$ where all structure maps are the identity.
\item the \emph{pair groupoid} $\mathcal{P} (M) = (M\times M \toto M)$ where the groupoid multiplication is given by $(a,b)\cdot (b,c) \coloneq (a,c)$ (and the other structure maps are obvious).
\end{itemize}
For $K$ compact and $\ell \in \N_0 \cup \{\infty\}$ (with $M$ finite dimensional if $\ell = \infty$) our construction yields $C^\ell (K,\mathfrak{u}(M)) = \mathfrak{u}(C^\ell (K,M))$. Collapsing the groupoid structure in this (trivial) example, the current groupoid encodes only the manifold of
$C^\ell$-maps $K \rightarrow M$, \cite{Wit,Mic}.
In view of Lemma \ref{base-cano} (d) we further have $C^\ell (K,\mathcal{P}(M)) \cong \mathcal{P}(C^\ell (K,(M)))$ as Lie groupoids.
\end{exa}

\begin{rem}
Note that in the situation of Theorem \ref{thm:A:ext}, the current groupoid $C^\ell (K,\mathcal{G})$ of a
Banach-Lie groupoid $\mathcal{G}$, is a Banach-Lie groupoid if $\ell < \infty$ and a Fr\'{e}chet-Lie
groupoid if $\ell = \infty$. Basic Lie theory and differential geometry for Banach-Lie groupoids have
recently been studied in \cite{BaGaJaP} (also see \cite{MaZ} for a discussion in a categorical framework).
\end{rem}

If $\Omega \subseteq M$ is open, the restriction $\mathcal{G}|_\Omega \coloneq (G|_{\Omega} \coloneq \alpha^{-1} (\Omega) \cap \beta^{-1} (\Omega) \toto \Omega)$ becomes a Lie groupoid called \emph{restriction of $\mathcal{G}$ to $\Omega$} \cite[p.\ 14]{Mac}.

\begin{exa}\label{exa: restriction}
In the situation of Theorem \ref{thm:A:ext} consider $\Omega \subseteq M$ open. Then we define $\lfloor K , \Omega\rfloor_\ell \coloneq \{f\in C^\ell (K,M) \mid f(K)\subseteq \Omega\}$ and note that it is an open subset of $C^\ell (K,M)$ due to Proposition \ref{fstar-gen}.
One immediately computes that the restriction of the current groupoid to $\lfloor K,\Omega\rfloor_\ell$ satisfies 
$$C^\ell (K,\mathcal{G})|_{\lfloor K,\Omega\rfloor_\ell} = C^\ell (K, \mathcal{G}|_\Omega) \text{ (as Lie groupoids)}.$$
Thus $C^\ell (K,\mathcal{G}|_\Omega)$ is an open subgroupoid of the current groupoid $C^\ell (K,\mathcal{G})$. 
\end{exa}

\begin{exa}\label{exa:action}
Consider a left Lie group action\footnote{Thus $\Lambda$ is
a left $G$-action on a smooth manifold $M$ and $\Lambda$ is smooth}
$\Lambda \colon G \times M \rightarrow M, (g,m)\mapsto g.m$. Then the action groupoid $G\ltimes M \coloneq (G\times M \toto M)$ is the Lie groupoid defined by the structure maps $\alpha (g,m)\coloneq m$, $\beta(g,m)\coloneq g.m$, $\mu((g,h.m),(h,m)))\coloneq (gh,m)$ and $\iota (g,m)\coloneq (g^{-1},g.m)$.
If $M$ admits a local addition, so does $G\times M$\footnote{$G$ admits a local addition as a Lie group and the product $G\times M$ inherits a local addition by the product of the local additions on $G$ and $M$.}, whence we can consider the current groupoid $C^\ell (K, G\ltimes M)$.

As the manifolds of mappings are canonical, Lemma \ref{base-cano} shows that
$$\Lambda_* \colon C^\ell (K,G\times M) \cong C^\ell (K,G)\times C^\ell (K,M) \rightarrow C^\ell (K,M)$$ 
is a Lie group action of the current Lie group $C^\ell (K,G)$ on $C^\ell (K,M)$ and moreover, the associated action groupoid satisfies 
$$C^\ell (K,G) \ltimes C^\ell (K,M) \cong C^\ell (K, G\ltimes M)$$
as Lie groupoids.
\end{exa}

We will now study some specific classes of current groupoids in the next sections. There Theorems B and C from the introduction will be established as immediate consequences of Theorem E in Section~\ref{sec: m:prop}. 

\subsection*{Transitivity and local transitivity of current groupoids}
 \addcontentsline{toc}{subsection}{Transitivity and local transitivity of current groupoids}

In this section, we investigate whether the current groupoid inherits the transitivity of the target groupoid
(resp., local transitivity).
To this end recall the following definitions.

\begin{defn} 
Let $\mathcal{G} = (G\toto M)$ be a Lie groupoid. Then we call the map
$$(\alpha,\beta) \colon G \rightarrow M\times M,\quad g \mapsto (\alpha(g),\beta(g))$$ 
the \emph{anchor} of $\mathcal{G}$.
We call the Lie groupoid $\mathcal{G}$
\begin{itemize}
\item \emph{transitive} if the anchor is a surjective submersion,
and \emph{totally intransitive} if the image of the anchor is the diagonal in $M\times M$;
\item \emph{locally transitive} if the anchor is a submersion.
\end{itemize}
\end{defn}

The next example shows that transitivity is not inherited by current groupoids:

\begin{exa}\label{not-tra}
Consider the left action of $\R$ on~$\Sph$ via $t.z:=e^{it}z$
and the corresponding action groupoid
$G\coloneq \R\ltimes \Sph$ over $M\coloneq\Sph$ with $\alpha\colon G\to\Sph$, $(t,z)\mto z$
and $\beta\colon G\to\Sph$, $(t,z)\mto e^{it}z$. Then $\alpha$, $\beta$, and $(\alpha,\beta)$
are submersions and $G$ is a transitive groupoid as the $\R$-action on~$\Sph$ is
transitive.
Taking $K:=\Sph$ and $\ell\in\N_0$,
we obtain a current groupoid $C^\ell(\Sph,G)$.
Let $c_1\colon \Sph\to\Sph$ be the constant map taking each element to $1\in\Sph$.
Then $(\alpha_*,\beta_*)\colon C^\ell(\Sph,G)\to C^\ell(\Sph,\Sph) \times C^\ell(\Sph,\Sph)$
is not surjective (whence $C^\ell(\Sph,G)$ is not a transitive Lie groupoid),
as $(\id_{\Sph},c_1)$ is not contained in its image.
In fact, if there was $\gamma=(\gamma_1,\gamma_2)\in C^\ell(\Sph,G)=C^\ell(\Sph,\R\times \Sph)$
with $(\alpha_*,\beta_*)(\gamma)=(\id_{\Sph},c_1)$, then
\[
\gamma_2=\alpha\circ \gamma =\id_{\Sph}.
\]
Hence $1=\beta(\gamma(z))=e^{i\gamma_1(z)}z$ for all $z\in \Sph$
and thus $z=e^{-i\gamma_1(z)}$, contradicting the fact that $\id_{\Sph}$
does not admit a continuous lift for the covering map $\R\to\Sph$, $t\mto e^{it}$.
\end{exa}

\begin{thm}[Theorem B]
If $\mathcal{G}$ is locally transitive in the situation of Theorem \ref{thm:A:ext}, then also $C^\ell(K,\mathcal{G})$ is locally transitive.
\end{thm}
\begin{proof}
Identifying the manifold $C^\ell(K,M)\times C^\ell(K,M)$ with $C^\ell(K,M\times M)$ (Lemma \ref{base-cano}),
the map $(C^\ell(K,\alpha),C^\ell(K,\beta))$ can be identified with
$C^\ell(K,(\alpha,\beta))\colon C^\ell(K,G)\to C^\ell(K,M\times M)$, $\gamma\mto (\alpha,\beta)\circ \gamma$. As $(\alpha,\beta)$ is a submersion, also $C^\ell(K,(\alpha,\beta))$
is a submersion (by Theorem~E\,(a)) and hence also $(C^\ell(K,\alpha),C^\ell(K,\beta))$.
Thus $C^\ell(K,G)$ is locally transitive.
\end{proof}

\subsection*{Current groupoids of proper and \'{e}tale Lie groupoids}
 \addcontentsline{toc}{subsection}{Current groupoids of proper and \'{e}tale Lie groupoids}

In this section, we study proper and \'{e}tale Lie groupoids. These Lie groupoids are closely connected to orbifolds and we review this connection in Appendix \ref{app:orbi} together with a discussion of how the groupoids we construct are connected to morphisms of orbifolds (see e.g.\ \cite{Sch}). 
Recall that a Lie groupoid is proper if the anchor is a proper map and \'{e}tale if the source map is a local diffeomorphism.
As the following example shows, current groupoids of proper groupoids need not be proper.

\begin{exa}\label{not-proper}
Consider $G:=\Sph\times\Sph$ as a Lie groupoid over $M:=\Sph$
with $\alpha=\beta:=\pr_1\colon \Sph\times\Sph\to\Sph$, $(z,w)\mto z$
and $(z,w_1)(z,w_2):=(z,w_1w_2)$ for $z,w_1,w_2\in\Sph$
(using the multiplication in the circle group).
Thus the Lie groupoid we obtain is a Lie group bundle, which is a totally intransitive.

Then $(\alpha,\beta)\colon G\to M\times M$ is a proper map (as $G$ is compact).
Hence~$G$ is a proper Lie groupoid. However, the Lie groupoid $C(K,G)$
is not proper for any compact smooth manifold~$K$ of positive dimension.
To see this, note that $(\alpha_*,\beta_*)\colon C(K,G)\to C(K,M)\times C(K,M)$
is the map taking $(\gamma_1,\gamma_2)$ to $(\gamma_1,\gamma_1)$.
If we fix $\eta\in C(K,M)$, then
\[
(\alpha_*,\beta_*)^{-1}(\{\eta\})=\{\eta\}\times C(K,\Sph)
\]
in the topological space $C(K,\Sph)\times C(K,\Sph)\sim C(K,\Sph\times\Sph)=C(K,G)$.
As the singleton $\{\eta\}$ is compact but $C(K,\Sph)$ is an infinite-dimensional
manifold and hence not compact, we deduce that $(\alpha_*,\beta_*)$
is not a proper map.
\end{exa}

Though properness is not preserved in general, there are special situations, outlined in Theorem C, in which properness is preserved.

\begin{numba}[Theorem C]
\emph{Let $\mathcal{G}$ be an \'{e}tale Lie groupoid such that $G$ and~$M$ admit a local addition.
Let $K$ be a compact smooth manifold} (\emph{possibly with rough boundary}), \emph{and $\ell\in\N_0\cup\{\infty\}$.
If the topological space underlying~$G$ is regular, then $C^\ell(K,\mathcal{G})$ is an \'{e}tale Lie groupoid.
If, moreover, $\mathcal{G}$ is proper, then also $C^\ell(K,\mathcal{G})$ is proper.}
\end{numba} 

\begin{proof}
 Since $\mathcal{G}$ is an \'{e}tale Lie groupoid, $\alpha$ and $\beta$
are local $C^\infty$-diffeomorphisms. Using Theorem~E~(c), we deduce that
$C^\ell(K,\alpha)$ and $C^\ell(K,\beta)$ are local $C^\infty$-diffeomorphisms
and hence submersions. We now find as in the proof of Theorem~A that $C^\ell(K,G)$ is a Lie groupoid. As $C^\ell(K,\alpha)$ is a local $C^\infty$-diffeomorphism,
$C^\ell(K,G)$ is \'{e}tale. If $G$ is \'{e}tale and proper,
then $\alpha$ is a local $C^\infty$-diffeomorphism and $(\alpha,\beta)$
is proper, whence $C^\ell(K,(\alpha,\beta))$ (and hence also $(C^\ell(K,\alpha), C^\ell(K,\beta))$)
is a proper map, by Theorem E (d).  Thus $C^\ell(K,G)$ is a proper Lie groupoid in this case.\end{proof}
Finite-dimensional proper and \'{e}tale Lie groupoids are locally isomorphic to action groupoids, see \cite[Proposition 2.23]{Ler}. We establish a suitable version in our setting. 

\begin{prop}\label{prop:loc:actgpd}
Let $\mathcal{G} = (G\toto M)$ be a proper \'{e}tale Lie groupoid such that $M\times M$ is a $k$-space. Then $\mathcal{G}$ is locally isomorphic to an action groupoid, i.e.\ every $x\in M$ has an open neighborhood $U_x \subseteq M$ with an action of the isotropy group $G_x \coloneq \alpha^{-1} (x) \cap \beta^{-1} (x)$ such that there is an isomorphism of \'{e}tale Lie groupoids 
$$\mathcal{G}|_{U_x} \cong G_x \ltimes U_x.$$
\end{prop}

\begin{rem}
The proof of Proposition \ref{prop:loc:actgpd} follows closely the classical proof in \cite[Theorem 4]{MaP} with some added detail (see proof of claim below). In addition, loc.cit.\ assumes that the Lie groupoids are effective and finite dimensional. Both assumptions are not necessary for this part of the proof of \cite[Theorem 4, (4) $\Rightarrow$ (1)]{MaP}. To highlight this, we chose to provide full details. 
\end{rem}

\begin{proof}[{Proof of Proposition \ref{prop:loc:actgpd}}]
Let us show that for a fixed $x\in M$ there exists an open neighborhood on which the groupoid $\mathcal{G}$ restricts to an action groupoid.

Note first that since $(\alpha,\beta) \colon G \rightarrow M\times M$ is a proper local diffeomorphism, the group $G_x \coloneq (\alpha,\beta)^{-1} \{x\}$ is finite. For all $g \in G_x$ we choose an open $g$-neighborhood $\Omega_g$ in $G$ such that $\alpha|_{\Omega_g}$ and $\beta|_{\Omega_g}$ are diffeomorphisms onto their (open) image in $M$. Shrinking the $\Omega_g$ we may assume that they are pairwise disjoint. 

\textbf{Claim:} There are open $g$-neighborhoods $W_g \subseteq \Omega_g$ such that
\begin{equation}\label{eq:multsubset}
\forall g,h\in G_x \text{ and }(x,y) \in W_g \times W_h \text{ with } \alpha(x) = \beta(y), \text{ we have }xy \in \Omega_{gh} . 
\end{equation}
If this is true, then the proof can be concluded as follows: Consider the open $x$-neighborhood $U_x \coloneq \bigcap_{g\in G_x} \alpha (W_g)$. 
Since $(\alpha,\beta)$ is proper and $M\times M$ a $k$-space, $(\alpha,\beta)$ is a perfect map, whence closed (cf.\ remarks after Definition \ref{defn: proper}). Thus we can apply \cite[3.2.10. Wallace Theorem]{Eng} to obtain an open $x$-neighborhood $V_x \subseteq U_x$ with
\begin{equation}\label{eq: VX} \begin{aligned}
V_x \times V_x \cap (\alpha,\beta) \left(G\setminus \cup_{g\in G_x} W_g\right) =\emptyset, &\text{ i.e. } h \in G \text{ with } \alpha (h),\beta(h) \in N_x \\ &\Rightarrow h \in W_g \text{ for some } g \in G_x.
\end{aligned}
\end{equation}
As $W_g \subseteq \Omega_g$ and $\alpha,\beta$ restrict to diffeomorphisms on $\Omega_g$, we can now define for $g\in G_x$ the diffeomorphism 
$$\delta_g \colon \alpha (W_g) \rightarrow \beta (W_g),\quad \delta_g \coloneq \beta \circ (\alpha|_{W_g})^{-1}.$$ 
As every $\delta_g$ is defined on $V_x \subseteq U_x$, we can define an open $x$-neighborhood via 
$$N_x \coloneq \{ y\in V_x \mid \delta_g(y) \in V_x,\ \forall g \in G_x\} = V_x \cap \bigcap_{g\in G_x} \delta_g^{-1} (V_x).$$
We claim that $\delta_g(N_x) \subseteq N_x$ for all $g\in G_x$. To see this, note that since $y$ and $\delta_{g_1} (y)$ are contained in $V_x$, both $\delta_{g_2} \circ \delta_{g_1}(y)$ and $\delta_{g_1g_2}(y)$ are defined. 
By construction of $\delta_g$, $\delta_{g} (z)$ is the target of the unique arrow in $W_{g}$ starting at $z$. Thus $\delta_{g_2} \circ \delta_{g_1} (z)$ is the target of a product of arrows in $W_{g_2} \times W_{g_1}$ and by \eqref{eq:multsubset} this arrow is the unique arrow in $\Omega_{g_2g_1}$ starting at $z$. Now $\delta_{g_1g_2} (z)$ is the target of an arrow in $W_{g_2g_1} \subseteq \Omega_{g_2g_1}$ starting at $z$ and by uniqueness $\delta_{g_2} \circ \delta_{g_1} (z) = \delta_{g_2g_1}(z) \in V_x$ holds. Hence we obtain a group action 
\begin{equation}\label{eq: delta}
\delta \colon G_x \times N_x \rightarrow N_x,\quad  (g,y) \mapsto \delta_g (y).
\end{equation}
Now we define for $g\in G_x$ the open $g$-neighborhood 
$$O_g \coloneq W_g \cap \alpha^{-1} (N_x) = W_g \cap (\alpha,\beta)^{-1} (N_x\times N_x),$$
where the last identity follows from \eqref{eq: delta} as $\delta_g(y) \in N_x$ for all $y\in N_x, g\in G_x$. From \eqref{eq: VX} we deduce that $(\alpha,\beta)^{-1} (N_x\times N_x) = \sqcup_{g\in G_x} O_g$ is the disjoint union of the open sets $O_g$.
We can thus consider the open Lie subgroupoid $\mathcal{G}|_{N_x} = \sqcup_{g\in G_x} O_g \toto N_x$. Using that $\alpha$ restricts to a diffeomorphism on every $O_g$, we obtain a diffeomorphism 
$$\Phi \colon \sqcup_{g\in G_x} O_g \rightarrow G_x \times N_x ,\ \gamma \mapsto (g,\alpha(\gamma)), \text{ if } \gamma \in O_g.$$
From the definition of the $\delta$-action of $G_x$ on $N_x$ it is then clear, that $\Phi$ induces a Lie groupoid isomorphism $\mathcal{G}|_{N_x} \cong (G_x \ltimes N_x \toto N_x)$ onto the action groupoid associated to \eqref{eq: delta}.\smallskip

\textbf{Proof of the claim:} As $\mathcal{G}$ is a Lie groupoid, the multiplication $m \colon G\times_{\alpha,\beta} G \rightarrow G$ is continuous. By \cite[Theorem B]{SUB}, the domain $G\times_{\alpha,\beta} G$ is a split submanifold of $G\times G$ such that the projections $\text{pr}_i \colon G\times_{\alpha,\beta} G \rightarrow G, i \in \{1,2\}$ onto the $i$th component are submersions, whence open mappings. For every choice $g,h \in G_x$ we thus obtain open subsets 
$$L_{g,h} \coloneq \pr_{1} (\Omega_g\times \Omega_h \cap m^{-1} (\Omega_{gh})), \quad R_{g,h} \coloneq \pr_{2} (\Omega_g\times \Omega_h \cap m^{-1} (\Omega_{gh})).$$
By construction $g\in L_{g,h}\subseteq \Omega_g, h \in R_{g,h} \subseteq \Omega_h$. Let now $(x,y) \in L_{g,h} \times R_{g,h}$ such that $\alpha (x) = \beta(y)$. As $\alpha, \beta$ restrict to bijections on $\Omega_g$ for every $g\in G_x$, $(x,y)$ is the unique pair in $\Omega_g\times \Omega_h$ with $\alpha (x) = \beta(y)$. Now by construction of $L_{g,h}$, there must be (at least) one pair in $(x,z) \in  \Omega_g \times \Omega_h \cap m^{-1} (\Omega_{gh})$. By definition of this set, $(x,z) \in \Omega_g \times \Omega_h$ with $\alpha(x) = \beta(z)$, whence $z=y$. This entails $m(x,y) \in \Omega_{gh}$ whenever a pair of arrows in $L_{g,h}\times R_{g,h}$ is composable. Since $G_x$ is finite we obtain open $g$-neighborhoods $W_g \coloneq \bigcap_{h\in G_x} L_{g,h} \cap R_{h,g}$. By construction the $W_g$ satisfy \eqref{eq:multsubset}.  
\end{proof}

In the situation of Theorem C, the current groupoid of a proper \'{e}tale Lie groupoid is again proper \'{e}tale such that its base, $C^\ell (K,M)$, is a \Frechet\, manifold. Thus $C^\ell (K,M) \times C^\ell (K,M)$ is a \Frechet\, manifold, hence a $k$-space and we obtain:
\begin{cor}
In the situation of Theorem C, the proper \'{e}tale Lie groupoid $C^\ell (K,\mathcal{G})$ is locally isomorphic to an action groupoid.
\end{cor}
\subsubsection*{Analogues to Theorem C for topological groupoids}

Let $G$ be a (Hausdorff) topological groupoid over the Hausdorff topological space~$M$,
with initial point map $\alpha\colon G\to $ and terminal point map $\beta\colon G\to M$.
Given a Hausdorff topological space~$K$,
we endow $C(K,G)$ and $C(K,M)$ with the compact-open topology.
Then $C(K,G)$ is a topological groupoid over the base $C(K,M)$,
with initial point map $C(K,\alpha)$ and terminal point map
$C(K,\beta)$ (as the latter maps, the groupoid multiplication
and the map taking a base point to its corresponding identity element
are continuous by standard results concerning the compact-open topology, like
\cite[Lemma~A.5.3]{GaN}).
\begin{numba}
A topological groupoid~$G$ is called \emph{\'{e}tale}
if $\alpha\colon G\to M$
is a local homeomorphism.
If $(\alpha,\beta)\colon G\to M\times M$ is a proper map,
then~$G$ is called \emph{proper}.
\end{numba}

The following is immediate from Propositions~\ref{map-eta}
and~\ref{ess-prop}.
\begin{cor}\label{top-eta}
Let $K$ be a compact Hausdorff topological space
and $G$ be a topological groupoid over a Hausdorff topological space~$M$.
Assume that the topological space underlying~$G$ is regular.
If $G$ is \'{e}tale,
then also the mapping groupoid
$C(K,G)$ is \'{e}tale.
If $G$ is \'{e}tale and proper
and $K$ is locally connected, then $C(K,G)$ is proper.\Punkt
\end{cor}
The results obtained in this section on proper \'{e}tale topological/Lie groupoids are used in Appendix \ref{app:orbi} to discuss (infinite-dimensional) orbifolds.

\subsection*{Subgroupoids and groupoid actions}
\addcontentsline{toc}{subsection}{Subgroupoids and groupoid actions}

In this section we explore subgroupoids and groupoid actions of current groupoids. To this end, let us observe first that the construction of current groupoids is functorial.

\begin{rem}[Functoriality of the current groupoid construction]
Let $F \colon \mathcal{G} \rightarrow \mathcal{H}$ be a morphism\footnote{That is a smooth map $F \colon H \rightarrow G$ (where $H$ and $G$ are the arrow manifolds) which is compatible with the groupoid multiplication and inversion in the obvious way and maps units to units. Since $F$ maps units to units, it descends to a smooth map $f$
between
the bases.} of Lie groupoids between Lie groupoids which satisfy the assumptions of Theorem \ref{thm:A:ext}. Then the push-forward induces a groupoid morphism 
$$C^\ell (K,F) \colon C^\ell (K,\mathcal{G}) \rightarrow C^\ell (K,\mathcal{H})$$
which is smooth due to Corollary \ref{la-reu}. Similarly, one can prove that the construction takes natural transformations between morphisms of Lie groupoids to natural transformations (cf.\ \cite[3.5]{MaZ}). 
In conclusion, we obtain for every compact manifold $K$ and $\ell \in \N_0 \cup \{\infty\}$ a ($2$-)functor between suitable ($2$-)categories of groupoids. In the present paper we will not investigate this further.
\end{rem}

\begin{defn}
Let $F \colon \mathcal{H} \rightarrow \mathcal{G}$ be a morphism of Lie groupoids. We call $H$ an
\begin{itemize}
\item \emph{immersed subgroupoid} of $\mathcal{G}$ if $F$ and the induced map on the base are injective immersions.
\item \emph{embedded subgroupoid} of $\mathcal{G}$ if $F$ and the induced map on the base are embeddings. 
\end{itemize}
\end{defn}

We have already seen in Example \ref{exa: restriction} that the restriction of a Lie groupoid to an open set gives rise to a corresponding restriction of the current groupoids. More generally, one immediately concludes from Theorem E (b) and Proposition \ref{prop:emb} the following.

\begin{cor}
Let $\mathcal{H}$ be an immersed Banach subgroupoid of the Banach groupoid $\mathcal{G}$ and $\ell \in \N$.
 Then $C^\ell (K,\mathcal{H})$ is an immersed Banach subgroupoid of $C^\ell (K,\mathcal{G})$.
If in addition $\mathcal{H}$ is an embedded subgroupoid of $\mathcal{G}$, then $C^\ell (K,\mathcal{H})$ is an embedded Banach subgroupoid of $C^\ell (K,\mathcal{G})$. 
\end{cor}

Another way to construct subgroupoids of current groupoids from open subsets of the manifold base will be discussed now.

\begin{numba}\label{numba: restriction}
For $\Omega \subseteq M$ open we define the set 
$$\mathcal{I}_K(\Omega) \coloneq \{f\in C^\ell (K,M) \mid f(K) \cap \Omega \neq \emptyset \} = \bigcup_{x \in K} \varepsilon (\cdot,x)^{-1} (\Omega)$$
As the evaluation is continuous by Lemma \ref{base-cano}, $\mathcal{I}_K(\Omega)$ is an open subset of the base of the current groupoid $C^\ell (K,\mathcal{G})$ and we can consider the restriction $C^\ell (K,\mathcal{G})|_{\mathcal{I}_K(\Omega)}$.
\end{numba}

For the next result we restrict ourselves to $\ell=\infty$. Though the authors believe that the statement is also true for $\ell \in \N_0$, the proof uses a result which, to our knowledge, has so far only been established in the $\ell = \infty$ case.

\begin{prop}
In the situation of Theorem \ref{thm:A:ext}, consider an open subset $E \subseteq C^\infty (K,M)$.
Then the restriction $C^\infty (K,\mathcal{G})|_E$ is an (open) embedded Lie subgroupoid of $C^\infty (K, \mathcal{G}|_{\varepsilon (E\times K)})$.
\end{prop}

\begin{proof}
Observe first that $\varepsilon \colon C^\infty (K,M) \times K \rightarrow M$ is a smooth, surjective submersion by \cite[Corollary 2.9]{SaW2}. Hence $\varepsilon$ is open and so is
$\Omega \coloneq \varepsilon(E\times K)$.
Thus it makes sense to consider the restriction $C^\infty (K, \mathcal{G}|_{\varepsilon (E\times K)})$ as an open Lie subgroupoid of $C^\infty (K,\mathcal{G})$.

By the definition of the restriction, $f \in C^\infty (K,G)|_E$ satisfies $\alpha_* (f), \beta_* (f) \in E$, whence $f \in C^\infty (K,G|_{\varepsilon (K\times E)})$. We conclude that $C^\infty (K,\mathcal{G})|_E \subseteq C^\infty (K, \mathcal{G}|_{\varepsilon (E\times K)})$ as open sets, hence as open Lie subgroupoids. 
\end{proof}
In the rest of this section we study current groupoids related to groupoid actions. This generalises Example \ref{exa:action} of the current groupoid of an action groupoid. We recall first the definition of a Lie groupoid action.

\begin{defn}
An action of a Lie groupoid $\mathcal{G} = (G\toto M)$ on a smooth map $q \colon X \rightarrow M$ is given by a smooth action map
$$\mathcal{A} \colon G \times_{\alpha,q} X \rightarrow X, \quad g.x \coloneq \mathcal{A}(g,x) $$
where $G \ltimes X := \{(g,x) \in G \times X\mid \alpha (g)= q(x)\}$ is the fiber product. 
We call $X$ a \emph{(left) $\mathcal{G}$-manifold} if the action map satisfies $q (g.x) = \beta (g)$, as well as $(g_1g_2).x = g_1.(g_2.x)$ and $\mathbf{1}_m . x = x$ for all $g_i \in G, x \in X, m \in M$, and whenever the composition is defined.
The map $q$ is called the \emph{moment map} of the action.

We define the \emph{action groupoid} $\mathcal{G} \ltimes X \coloneq (G\ltimes X \toto X)$ as the Lie groupoid with $\alpha_\ltimes (g,x)\coloneq x$, $\beta_\ltimes (g,x) \coloneq g.x$ and multiplication, inversion and unit map induced by the corresponding mappings in $\mathcal{G}$ \cite[1.6]{Mac}. 
\end{defn}

Note that if the Lie groupoid $\mathcal{G}$ is a Lie group, then a Lie groupoid action coincides with a Lie group action and the action groupoid just defined is the one discussed in Example \ref{exa:action}.

\begin{prop}\label{prop:findim:actiongpd}
Let $\mathcal{G}$ be a finite dimensional Lie groupoid and $X$ be a finite dimensional manifold. If $X$ is a $\mathcal{G}$-manifold, then $C^\ell (K,X)$ is a $C^\ell (K,\mathcal{G})$-manifold and we obtain an isomorphism of Lie groupoids 
$$C^\ell (K,\mathcal{G})\ltimes C^\ell (K,X) \cong C^\ell (K,\mathcal{G} \ltimes X).$$
\end{prop}
  
\begin{proof}
Let $q$ be the moment map and $\mathcal{A} \colon G\ltimes X \rightarrow X$ be the action map of the groupoid action. 
By Theorem E, the fiber product $C^\ell (K, G) \ltimes C^\ell (K,X) = (\alpha_*,q_*)^{-1} (\{(f,f) \mid f \in C^\ell (K,M)\}$ is a splitting submanifold of $C^\ell (K,G) \times C^\ell (K, X)$ \cite[Theorem B]{SUB}. 
Further, we deduce from Proposition \ref{prop:emb} that $C^\ell (K,G\ltimes X)$ is a splitting submanifold of $C^\ell (K,G\times X)$.
It is easy to see that the isomorphism  $C^\ell (K,G) \times C^\ell (K,X) \cong C^\ell (K,G\times X)$ and its inverse factors through the split submanifolds.
Thus $C^\ell (K, G\ltimes X) \cong C^\ell (K,G) \ltimes C^\ell (K,X)$ as sets and also as manifolds, since smoothness is inherited by the (co-)restriction of the smooth maps to the split submanifolds. 
In particular, the push-forward of the action map $\mathcal{A}$ induces a smooth action 
\begin{align*}C^\ell (K,G)\ltimes C^\ell (K,X) \rightarrow C^\ell (K,X) , \quad (g,f) \mapsto \mathcal{A} \circ (g , f). \tag*{\qedhere}
\end{align*}
\end{proof}
A finite dimensional proper \'{e}tale Lie groupoid is locally around $x \in M$ isomorphic to an action groupoid $\mathcal{G}|_{U_x} \cong (G_x \ltimes U_x \toto U_x)$ \cite[Theorem 2.23]{Ler} (cf.\ Proposition \ref{prop:loc:actgpd}). Hence combining Example \ref{exa: restriction} and Proposition \ref{prop:findim:actiongpd} immediately yields:

\begin{cor}
Let $\mathcal{G}$ be a finite dimensional proper \'{e}tale Lie groupoid locally isomorphic to an action groupoid $G_x \ltimes U_x \toto U_x$. Then $C^\ell (K,\mathcal{G})|_{\lfloor K,U_x\rfloor_\ell}$ is isomorphic to an embedded Lie subgroupoid of the action groupoid $C^\ell (K,G_x)\ltimes C^\ell (K,U_x)$.
\end{cor}  

\newpage
\section{Current algebroids}
In this section, we study the Lie algebroid associated to a current groupoid. Lie algebroids are infinitesimal counterparts of Lie groupoids akin to the role Lie algebras play to Lie groups. Let us first recall from \cite{BaGaJaP} the definition of an (infinite-dimensional) Lie algebroid:

\begin{defn}
Fix a locally convex vector bundle $\mathcal{A} \rightarrow M$ over a locally convex manifold together with a vector bundle morphism $a\colon \mathcal{A} \rightarrow TM$ covering the identity ($a$ is called \emph{anchor}, $(\mathcal{A},M,a)$ \emph{anchored bundle}). Note that the anchor induces a map $a \colon \Gamma (\mathcal{A}) \rightarrow \Gamma (TM)$ by post-composition.
\begin{enumerate} 
\item A Lie bracket on the anchored bundle is a skew-symmetric $\mathbb{R}$-linear map\\ $[\cdot,\cdot]_\mathcal{A} \colon \Gamma (\mathcal{A}) \times \Gamma (\mathcal{A}) \rightarrow \Gamma (\mathcal{A})$ satisfying the following conditions 
\begin{enumerate}
\item $[X,fY]_{\mathcal{A}} = f + Tf a(X)X$ for all $f\in C^\infty (M)$ and $X,Y \in \Gamma (\mathcal{A})$.
\item $[X,[Y,Z]_{\mathcal{A}}]_{\mathcal{A}} + [Z,[X,Y]_{\mathcal{A}}]_{\mathcal{A}} + [Y,[Z,X]_{\mathcal{A}}]_{\mathcal{A}} = 0$ for all $X,Y,Z\in \Gamma (\mathcal{A})$.
\end{enumerate}
A Lie bracket is \emph{localisable} if for every non-empty open set $U\subseteq M$ there is a Lie bracket $[\cdot,\cdot]_{U}$ on the restriction of the anchored bundle $\mathcal{A}|_U$ such that for $U=M$ we have $[\cdot,\cdot]_U=[\cdot,\cdot]_\mathcal{A}$
and $[X|_V,Y|_V]_V = [X,Y]_U|_V$ for all open subsets $V\subseteq U\subseteq M$
and $X,Y\in \Gamma (\mathcal{A}|_{U})$.
\item An anchored bundle with a Lie bracket is called a \emph{Lie algebroid} if the Lie bracket is localisable and the map $a \colon\Gamma(\mathcal{A}) \rightarrow \Gamma (TM)$ is a Lie algebra morphism.   
\end{enumerate}
\end{defn}

\begin{rem}\label{rem:localisation}
Localisability of the Lie bracket is a new feature which is automatic for finite-dimensional algebroids. The Lie brackets associated to a Lie groupoid (see \ref{setup: tan:alpha} below) are automatically localisable (this was proved in \cite[Theorem 4.17]{BaGaJaP} for Banach Lie algebroids, however the proof carries over to our more general situation). Hence we chose to include it in the definition. Up to now no example of a non-localisable Lie bracket is known.  
\end{rem}

\begin{numba}\label{numba:cur:alg}
Let $K$ be a compact manifold, $\ell \in \N_0 \cup \{\infty\}$ and $(\mathcal{A}, a, \LB )$ be a Lie algebroid over $M$. Assume that $M$ and $\mathcal{A}$ admit local additions.
 Then the push-forward of the bundle projection turns $C^\ell (K,\mathcal{A})$ into a vector bundle over $C^\ell (K,M)$. Using the canonical manifold structure, we see that the pointwise application of
the Lie bracket of the Lie algebroid yields a Lie bracket on the section algebra $\Gamma (C^\ell (K,M),C^\ell (K,\mathcal{A}))$ which is compatible with the push-forward of the anchor of $\mathcal{A}$. 
We obtain an anchored bundle which becomes a Lie algebroid $(C^\ell(K,\mathcal{A}), a_*, [\cdot,\cdot]_{pw})$ if the Lie bracket is localisable. A Lie algebroid of this form is called a \emph{current algebroid}.
\end{numba}

Note that for a Lie algebroid $\Lf (\mathcal{G})$ associated to a Lie groupoid $\mathcal{G}$ (to be recalled in \ref{setup: tan:alpha} below) the anchored bundle in \ref{numba:cur:alg} will always have a localisable Lie bracket. This follows a posteriori from Theorem \ref{thm: currentalg} and Remark \ref{rem:localisation}. As we are only interested in the Lie algebroid associated to current groupoids, we shall not investigate the localisability further. However, it is well known that there are Lie algebroids which do not integrate to Lie groupoids \cite{CaF} (i.e.\ are not of the form $\Lf (\mathcal{G})$) and for these we do not know whether the anchored bundle is localisable.\footnote{Note that to establish localisability of general algebroids \cite{{BaGaJaP}} requires smooth bump functions on the base. However, for $\ell <\infty$ the mapping spaces $C^\ell (K,M)$ do not admit such bump functions, see \cite[Section 14]{KaM}.}

Let us now recall how to associate such a Lie algebroid to a Lie groupoid $\mathcal{G} = (G \toto M)$, e.g.\ \cite[Section 3.5]{Mac} or \cite{BaGaJaP,SaW}.

\begin{numba} \label{setup: tan:alpha} Consider the subset
 $T^\alpha G = \bigcup_{g\in G} T_g \alpha^{-1} (\alpha (g))$ of $TG$. 
 For all $x \in T^\alpha_g G$ the definition implies
 $T\alpha (x) = 0_{\alpha (g)} \in T_{\alpha (g)} M$, i.e.\ fiber-wise we have
 $T^\alpha_g G = \ker T_g\alpha$. Since $\alpha$ is a submersion, the
 same is true for $T\alpha$. Computing in submersion charts, the kernel of
 $T_g \alpha$ is a direct summand of the model space of $TG$. Furthermore, the
 submersion charts of $T \alpha$ yield submanifold charts for $T^\alpha G$
 whence $T^\alpha G$ becomes a split submanifold of $T G$. Restricting the
 projection of $TG$, we thus obtain a subbundle
 $\pi_\alpha \colon T^\alpha G \to G$ of the tangent bundle $TG$. 
 
 An element in $\Gamma (T^\alpha G)$ is called \emph{vertical vector field}. A vertical vector field $Y$ is called \emph{right-invariant} if
 for all $(h,g) \in G \times_{\alpha, \beta} G$ the equation
 $Y(hg) = T_h (R_g)(Y(h))$ holds. Due to \cite[Lemma 3.5.5]{Mac} the set of all right invariant vector fields $\Gamma^R(T^\alpha G)$ is a Lie subalgebra of $\Gamma(TG)$.
\end{numba}

\begin{numba}[The Lie algebroid associated to a Lie groupoid]
 Define $\Lf(\mathcal{G})$ to be the pullback bundle $1^{*}T^{\alpha}G$ where $1 \colon M \rightarrow G$ is the unit embedding (we will think of the pullback $1^*$ as restriction). The
 anchor $a_{\Lf(\mathcal{G})} \colon \Lf(\mathcal{G}) \rightarrow TM$ is defined as the composite 
 \begin{displaymath}
  \Lf(G) \rightarrow T^\alpha G \xrightarrow{\subseteq} TG \xrightarrow{T\beta}
  TM.
 \end{displaymath}
 Let $g$ be an element of $G$. We define the smooth map
 $R_g \colon \alpha^{-1}(\beta (g)) \rightarrow G$, $h \mapsto hg$.  Then \cite[Corollary
 3.5.4]{Mac} (or \cite[Propositions 4.13 (i) and 4.14 (ii)]{BaGaJaP}) shows 
 \begin{equation}\label{eq: RI:VF}
  \Gamma(\Lf(\mathcal{G})) \rightarrow \Gamma^R (G),\quad X \mapsto \overrightarrow{X} , \quad \text{ with } \quad\overrightarrow{X} (g) = T(R_g) (X(\beta (g)))
 \end{equation}
 is an isomorphism of $C^\infty (G)$-modules. Its inverse is
  $\Gamma^R (G) \rightarrow \Gamma (\Lf(\mathcal{G}))$, $X \mapsto X \circ 1$. Now define the Lie bracket on $\Gamma (\Lf(\mathcal{G}))$ via
 \begin{equation}\label{eq: LB}
  [X,Y] \coloneq [\overrightarrow{X},\overrightarrow{Y}] \circ 1.
 \end{equation}
 Then $(\Lf(\mathcal{G}), M,a_{\Lf (\mathcal{G})}, [\cdot,\cdot])$ is the Lie algebroid \cite[Theorem 4.17]{BaGaJaP} associated to $\mathcal{G}$.
\end{numba}

We can now identify the Lie algebroid associated to a current groupoid.

\begin{thm}[The Lie algebroid associated to a current groupoid]\label{thm: currentalg}
In the situation of Theorem \ref{thm:A:ext} consider the current groupoid $C^\ell (K,\mathcal{G})$. Then the Lie algebroid associated to the current groupoid is canonically isomorphic to a current algebroid,
$$(\Lf (C^\ell (K,\mathcal{G}), a_{\Lf (C^\ell (K,\mathcal{G})} , \LB ) \cong (C^\ell (K,\Lf (\mathcal{G}), (a_{\Lf (\mathcal{G}})_*, [\cdot,\cdot]_{pw}),$$
where the Lie bracket is given by the pointwise application of the bracket on $\Lf (\mathcal{G})$. 
\end{thm}

\begin{proof}
By Theorem \ref{thetabu} $TC^\ell (K,G) \cong C^\ell (K,TG)$ holds. Then \eqref{eq:commdiag} lets us conclude that $T^{\alpha_*} C^\ell (K,G) = \bigcup_{f\in C^{r} (K,G)} \text{Ker}\, (T_f \alpha_*) \cong C^\ell (K,T^\alpha G)$. 
Thus
\begin{equation}\label{ident: algebroid}
\Lf (C^\ell (K,\mathcal{G})) \cong (1_*)^* T^{\alpha_*} C^\ell (K,G) \cong C^\ell (K, 1^* (T^\alpha G)) = C^\ell (K,\Lf (\mathcal{G}))
\end{equation}
Analogously \eqref{eq:commdiag} yields $a_{\Lf (C^\ell (K,\mathcal{G}))} = (a_{\Lf (\mathcal{G})})_* \colon C^\ell (K,\Lf (\mathcal{G})) \rightarrow C^\ell (K,TM)$.
Finally, we observe that the Lie groupoid operations of the current groupoid are given by pointwise application of the groupoid operations of $\mathcal{G}$. Thus a section of the bundle $C^\ell (K, T^\alpha G) \rightarrow C^\ell (K,G)$ is right invariant if and only if it satisfies the right invariance property pointwise. We conclude from \eqref{ident: algebroid} together with \eqref{eq: RI:VF} and \eqref{eq: LB} that the Lie bracket on
$C^\ell(K,\Lf (\mathcal{G})$ induced from the Lie algebroid $\Lf (C^\ell (K,\mathcal{G})$ is given by the pointwise application bracket $[\cdot , \cdot ]_{\Lf (\mathcal{G}}$. 
Summing up, we can identify $\Lf (C^\ell (K,\mathcal{G}))$ with the Lie algebroid $(C^\ell(K,\Lf (\mathcal{G}), (a_{\Lf (\mathcal{G})_*}, [\cdot,\cdot]_{pw})$ where $[\cdot,\cdot]_{pw}$ is the pointwise Lie bracket.
Thus the Lie algebroid associated to a current groupoid is a current algebroid.
\end{proof}

The construction of the current algebroid in Theorem \ref{thm: currentalg} recovers the
construction of current algebras.

\begin{rem}\label{rem:shiftsign}
A locally convex topological Lie algebra of the form
$C^\ell (K,{\mathfrak h})$ with pointwise Lie bracket, where ${\mathfrak h}$
is a locally convex topological Lie algebra, is called a \emph{current algebra}.

As was noted in \cite[Warning after 1.7]{SaW}, a
Lie group $H$ with Lie algebra $\Lf (H)$ gives rise to a
Lie groupoid $H\toto \{\star\}$ over the one point manifold, but $\Lf (H \toto \{\star\}) \neq \Lf (H) \toto \{\star\}$. The reason for this is that due to conventions the Lie bracket of one of these Lie algebras is the negative of the other.
Thus the current algebra $C^\ell (K, \Lf(H) \toto \{\star\})$ of \cite{NaW} is only anti-isomorphic to $C^\ell (K, \Lf (H\toto \{\star\}))$.  
\end{rem}

We have restricted ourselves to compact $K$ in this section. For $\ell=\infty$ and $K$ without rough boundary (but possibly with smooth boundary and non compact) one can obtain a similar identification of the current algebroid if $\mathcal{G}$ is a Banach Lie groupoid.

To see this, note that the identification becomes $TC^\infty (K,G) \cong \mathcal{D} (K,TG)$ (where $\mathcal{D}$ denotes smooth mappings constant outside some compact set, \cite[Section 10]{Mic}). Then the above proof carries over using the results contained in \cite[Sections 10 and 11]{Mic}. We chose to suppress this more complicated case. 
\paragraph{Acknowledgements} The authors wish to thank D.M.\ Roberts (Adelaide) for useful comments which helped improve this article. H.A.\ was supported by a travel grant of the University of Zanjan.
Further funding was provided by DFG grant GL 357/9-1.
H.A.\ and A.S.\ wish to thank H.G.\ and the mathematical institute in Paderborn for their hospitality during their stay in summer 2018. 
\newpage
\appendix
\section{Manifold structure on \texorpdfstring{{\boldmath$C^\ell(K,M)$}}{differentiable mappings}}\label{map-mfd}
In this appendix, we present a construction of the canonical manifold structure on the spaces
$C^\ell (K,M)$ for $K$ a compact manifold (possibly with rough boundary),
$\ell\in \N_0\cup\{\infty\}$, and $M$ a (possibly infinite-dimensional) smooth manifold
which admits a local addition (in the sense recalled in Definition~\ref{def-loa}).
Constructions of the manifold structure
are well known in special cases; see \cite{Eel,Mic,Ham,KaM} (for $\ell=\infty$ and $K$ a manifold with corners), \cite{RaS} (for $\ell = \infty$ and $K$ with rough boundary), and \cite{Wit}
(for $K$ without boundary, finite~$\ell$ and $M$ of finite dimension).
In all approaches mentioned (with the exception of \cite{RaS}), the construction hinges on a version of the
so-called $\Omega$-Lemma. This result is not currently available for manifolds with rough boundary\footnote{But will be contained in \cite{GaN}.}. However, as suggested
by the work of Michor and carried out in \cite[Section 5]{RaS}, one can circumvent this problem for compact source manifolds by using exponential laws for $C^{k,\ell}$-functions (as provided in \cite{AaS}).
We work out this approach here and mention that the authors are not aware of a published source for the results if $\ell < \infty$,
in the current generality.
\subsection*{Pullbacks of vector bundles and their spaces of sections}
The smooth manifold structure on $C^\ell(K,M)$ we strive to construct
is modelled on spaces of $C^\ell$-sections of certain pullback bundles
$f^*(TM)$. We therefore give some explanations concerning such pullbacks
and their sections, before turning to $C^\ell(K,M)$.
\begin{numba}\label{bundleconv}
Let $M$ be a smooth manifold, $\ell\in\N_0\cup\{\infty\}$
and $K$ be a $C^\ell$-manifold (possibly with rough boundary).
If $\pi\colon E\to M$ is a smooth vector bundle over~$M$
and $f\colon K\to M$ is a $C^\ell$-map, then
\[
f^*(E):=\bigcup_{x\in K}\{x\}\times E_{f(x)}
\]
is a submanifold of the $C^\ell$-manifold $K\times E$
(as it locally looks like $\mbox{graph}(f)\times E_x$
inside $K\times M\times E_x$ around points in $\{x\}\times E_x$).
We endow $f^*(E)$ with the submanifold structure.
Together with the natural vector space structure on $\{x\}\times E_{f(x)}\cong E_{f(x)}$
and the map $\pi_f\colon f^*(E)\to K$, $(x,y)\mto x$,
we obtain a $C^\ell$-vector bundle $f^*(E)$ over~$K$,
the so-called \emph{pullback of $E$ along~$f$}.
For each local trivialization $\theta=(\pi|_{E|_U},\theta_2)
\colon E|_U\to U\times F$ of~$E$
and $W:=f^{-1}(U)$, the map
\[
f^*(E)|_W \to W\times F,\quad (x,y)\mto (x,\theta_2(y))
\]
is a local trivialization of $f^*(E)$.
We endow
\[
\Gamma_f:=\{\tau\in C^\ell(K,E)\colon \pi\circ \tau=f\}
\]
with the topology induced by $C^\ell(K,E)$. With pointwise operations,
$\Gamma_f$ is a vector space and the map
\[
\Psi\colon \Gamma_{C^\ell}(f^*(E))\to\Gamma_f,\;\, \sigma\mto \pr_2\circ \, \sigma
\]
is a bijection with inverse $\tau\mto (\id_K,\tau)$.
As $C^\ell(K,\pr_2)\colon C^\ell(K,K\times E)\to C^\ell(K,E)$ is a continuous map and also $\tau\mto (\id_K,\tau)\in
C^\ell(K,K) \times C^\ell(K,E)\cong C^\ell(K,K\times E)$ is continuous,
we deduce that $\Gamma_f$ is a locally convex topological vector space and
$\Psi$ is an isomorphism of topological vector spaces.
If we wish to emphasize the dependence on~$E$, we also write $\Gamma_f(E)$
instead of $\Gamma_f$.
\end{numba}
The following Exponential Law is essential for us, in the preceding situation.
\begin{la}\label{expsect}
Let $k\in \N_0\cup\{\infty\}$ and $g\colon N\to \Gamma_f$ be a map,
where $N$ is a $C^k$-manifold $($possibly with rough boundary$)$.
If $K$ is locally compact,
then $g$ is $C^k$ if and only if
\[
g^\wedge\colon N\times K\to E,\quad (x,y)\mto g(x)(y)
\]
is a $C^{k,\ell}$-map.
\end{la}
\begin{proof}
Let $(K_j)_{j\in J}$ be a cover of~$K$ by open sets $K_j$
such that $f(K_j)\subseteq U_j$ for an open subset $U_j\subseteq M$ such that
$E|_{U_j}$ is trivializable; let $\theta_j\colon E|_{U_j}\to U_j\times F_j$
be a local trivialization.
Write $\theta_j=(\pi,\theta_{j,2})$ with a smooth map $\theta_{j,2}\colon E|_{U_j}\to F_j$.
By Lemma~\ref{top-cover}, the topology on $\Gamma_f$ is initial with respect to the maps
\[
\rho_j\colon\Gamma_j\to C^\ell(K_j,E),\quad \tau\mto\tau|_{K_j}
\]
for $j\in J$. We may regard these as maps to $C^\ell(K_j,E|_{U_j})$
as $\tau(K_j)\subseteq E|_{U_j}$ for each $\tau\in\Gamma_f$ (cf.\ Lemma~\ref{embpfwd}).
As the map $C^\ell(K_j,\theta_j)\colon C^\ell(K_j,E|_{U_j})\to C^\ell(K_j,U_j\times E_j)
\cong C^\ell(K_j,U_j)\times C^\ell(K_j,F_j)$ is a homeomorphism,
the topology on $\Gamma_f$ is also initial with respect to the maps
$
C^\ell(K_j,\theta_j)\circ\rho_j\colon \Gamma_f\to C^\ell(K_j,U_j)\times C^\ell(K_j,F_j)
$
sending $\tau\in\Gamma_f$ to
\begin{equation}\label{const}
(f|_{K_j},\theta_{j,2}\circ \tau|_{K_j}).
\end{equation}
As the first component in (\ref{const}) is independent of $\tau$,
the topology on $\Gamma_f$ is also initial with respect to the mappings
\[
\psi_j\colon \Gamma_f\to C^\ell(K_j,F_j),\quad \tau\mto \theta_{j,2}\circ\tau|_{K_j}
\]
for $j\in J$, which are linear maps. As a consequence, the map
\[
\psi:=(\psi_j)_{j\in J}\colon\Gamma_f\to\prod_{j\in J}C^\ell(K_j,F_j)
\]
is linear and a topological embedding. The vector subspace
\begin{eqnarray*}
\lefteqn{\{(\gamma_j)_{j\in J}\in\prod_{j\in J}C^\ell(K_j,F_j)\colon}\qquad\qquad\\
&& (\forall i,j\in J)(\forall x\in K_i\cap K_j)\;
\gamma_i(x)=\theta_{i,2}(\theta_j^{-1}(f(x),\gamma_j(x)))\}
\end{eqnarray*}
of $\prod_{j\in J}C^\ell(K_j,F_j)$
is closed and coincides with the image of~$\psi$. [It clearly contains the image
and given $(\gamma_j)_{j\in J}$, we have $(\gamma_j)_{j\in J}=\psi(\tau)$
if we set $\tau(x):=\theta_j^{-1}(f(x),\gamma_j(x))$ for $j\in J$ and $x\in K_j$.]\\[2.3mm]
As a consequence, a map $g\colon N\to \Gamma_f$
as in the lemma is $C^k$ if and only if $\psi_j\circ g\colon N\to C^\ell(K_j,F_j)$ is
$C^k$ for all $j\in J$ (see \cite[Lemmas 1.4.5 and 1.4.15]{GaN}; cf.\ \ref{into-sub}).
By the Exponential Law \cite[Theorem~4.6]{AaS},
the latter holds if and only if
\[
(\psi_j\circ g)^\wedge\colon N\times K_j\to F_j,\quad (x,y)\mto \theta_{j,2}(g(x)(y))
\]
is $C^{k,\ell}$ for each $j\in J$. The latter holds if and only if
\[
N\times K_j\to U_j\times F_j,\quad (x,y)\mto (f(y),\theta_{j,2}(g(x)(y)))=\theta_j(g^\wedge(x,y))
\]
is $C^{k,\ell}$ for all $j\in J$. This in turn holds if and only if $g^\wedge|_{N\times K_j}\colon
N\times K_j\to E|_{U_j}\subseteq E$ is $C^{k,\ell}$
for all $j\in J$, which holds if and only if $g^\wedge$ is $C^{k,\ell}$.
\end{proof}
\begin{rem}
(a) Note that $C^{k,\ell}$-maps to $f^*(E)$ (which is only a $C^\ell$-manifold)
do not make sense, whence an exponential law for $\Gamma_{C^\ell}(f^*(E))$
would not make sense in na\"{\i}ve form. Because we do have an exponential law
for $\Gamma_f$, it is essential for us to work with $\Gamma_f$ rather than
$\Gamma_{C^\ell}(f^*(E))$.

(b) Dropping the hypothesis of local compactness, we might
assume instead, e.g., that $N\times K$ is metrizable in Lemma~\ref{expsect},
as the Exponential Law \cite[Theorem 4.6]{AaS} 
also applies in this situation (further variants involve $k$-spaces, see \emph{loc.cit.}).\smallskip

(c) If all fibers of~$E$ are Fr\'{e}chet spaces and $K$ is $\sigma$-compact
and locally compact, then $\Gamma_f$ is a Fr\'{e}chet space;
if all fibers of~$E$ are Banach spaces, $K$ is compact and $\ell<\infty$,
then $\Gamma_f$ is a Banach space. To see this, take
a countable (resp., finite) family $(K_j)_{j\in J}$ of compact subsets $K_j\subseteq K$ (instead of open ones)
whose interiors cover~$K$
with $f(K_j)\subseteq U_j$ as in the proof of Lemma~\ref{expsect},
such that $K_j$ admits a smooth manifold structure with rough boundary
making the inclusion map $K_j\to K$ a smooth map and such that $K_j$ and $K$
induce the same smooth manifold structure on the interior $K_j^0$ relative~$K$
(e.g., the $K_j$ can be chosen as preimages of compact convex sets with non-empty interior
under charts of~$K$).
Then
\[
\psi\colon \Gamma_f\to\prod_{j\in J}C^\ell(K_j,F_j),\quad \tau\mto (\theta_{j,2}\circ\tau|_{K_j})_{j\in J}
\]
is linear and a topological embedding with closed image. If all $F_j$ are Fr\'{e}chet
spaces, so is each $C^\ell(K_j,F_j)$ (see, e.g., \cite{GaN}) and hence also $\Gamma_f$. If all $F_j$ are Banach spaces and
$\ell$ as well as $J$ is finite, then each $C^\ell(K_j,F_j)$ is a Banach space
(\emph{loc.\ cit.})
and hence also $\Gamma_f$.
\end{rem}
\begin{la}\label{evalsect}
If $K$ is locally compact in the preceding situation,
then the evaluation map
\[
\ve\colon\Gamma_f\times K\to E,\quad (\tau,x)\mto \tau(x)
\]
is $C^{\infty,\ell}$.
\end{la}
\begin{proof}
Since $\id\colon \Gamma_f\to\Gamma_f$, $\tau\mto \tau$ is a $C^\infty$-map,
$\id^\wedge\colon\Gamma_f\times K\to E$, $(\tau,x)\mto\id(\tau)(x)=\tau(x)=\ve(\tau,x)$
is $C^{\infty,\ell}$, by Lemma~\ref{expsect}.
\end{proof}
\begin{rem}
The conclusions of Lemmas~\ref{expsect} and \ref{evalsect}
remain valid if $K$ is only a $C^\ell$-manifold
(possibly with rough boundary) and $N$ as well as the vector bundle
$\pi\colon E\to M$ are only $C^{k+\ell}$; however, we shall not need the added
generality.\footnote{The evaluation $\ve$ will be $C^{k,\ell}$ then
and hence $C^{\infty,\ell}$, being linear in its first argument
(cf.\ \cite[Lemma~3.14]{AaS}).}
\end{rem}
\begin{la}\label{Gammfunct}
Let $\pi_1\colon E_1 \to M$ and $\pi_2\colon E_2\to M$ be smooth vector bundles
over a smooth manifold~$M$. Let $\ell\in\N_0\cup\{\infty\}$ and $f\colon K\to M$
be a $C^\ell$-map from a $C^\ell$-manifold~$K$ $($possibly with rough boundary$)$ to~$M$.
Then the following holds:
\begin{itemize}
\item[\textup{(a)}]
If $\psi\colon E_1\to E_2$ is a mapping of smooth vector bundles over~$\id_M$,
then $\psi\circ\tau\in\Gamma_f(E_2)$ for each $\tau\in\Gamma_f(E_1)$ and
\[
\Gamma_f(\psi)\colon \Gamma_f(E_1)\to\Gamma_f(E_2),\quad \tau\mto \psi\circ \tau
\]
is a continuous linear map.
\item[\textup{(b)}]
$\Gamma_f(E_1\oplus E_2)$ is canonically isomorphic to $\Gamma_f(E_1)\times \Gamma_f(E_2)$.
\end{itemize}
\end{la}
\begin{proof}
(a) If $\tau\in \Gamma_f(E_1)$, then $\psi\circ\tau\colon K\to E_2$ is $C^\ell$ and $\pi_2\circ
\psi\circ \tau =\pi_1\circ\tau=f$,
whence $\psi\circ \tau\in \Gamma_f(E_2)$. Evaluating at points we see that the map
$\Gamma_f(\psi)$ is linear; being a restriction of the continuous map $C^\ell(K,\psi)\colon
C^\ell(K,E_1)\to C^\ell(K,E_2)$, it is continuous.

(b) If $\rho_j\colon E_1\oplus E_2\to E_j$
is the map taking $(v_1,v_2)\in E_1 \times E_2$ to $v_j$ for $j\in \{1,2\}$
and $\iota_j\colon E_j\to E_1\oplus E_2$ is the map taking $v_j\in E_j$
to $(v_1,0)$ and $(0,v_2)$, respectively, then
\[
(\Gamma_f(\rho_1),\Gamma_f(\rho_2))\colon \Gamma_f(E_1\oplus E_2)\to\Gamma_f(E_1)\times
\Gamma_f(E_2)
\]
is a continuous linear map
which is a homeomorphism as it has the continuous map $(\sigma,\tau)\mto \Gamma_f(\iota_1)(\sigma)+
\Gamma_f(\iota_2)(\tau)$ as its inverse.
\end{proof}
\subsection*{Construction of the canonical manifold structure}
We now construct the canonical manifold structure on $C^\ell(K,M)$,
assuming that~$M$ admits a local addition.We recall the concept.
\begin{defn}\label{def-loa}
Let $M$ be a smooth manifold.
A \emph{local addition} is a smooth map
\[
\Sigma \colon U \to M,
\]
defined on an open neighborhood $U \subseteq TM$ of the zero-section
$0_M:=\{0_p\in T_pM\colon p\in M\}$
such that $\Sigma(0_p)=p$ for all $p\in M$,
\[
U':=\{(\pi_{TM}(v),\Sigma(v))\colon v\in U\}
\]
is open in $M\times M$ (where $\pi_{TM}\colon TM\to M$ is the bundle
projection) and the map
\[
\theta:=(\pi_{TM},\Sigma)\colon U \to U'
\]
is a $C^\infty$-diffeomorphism. If
\begin{equation}\label{bettersigma}
T_{0_p}(\Sigma|_{T_pM})=\id_{T_pM}\;\,
\mbox{for all $p\in M$,}
\end{equation}
we say that the local addition $\Sigma$ is \emph{normalized}.
\end{defn}
Until Lemma~\ref{la:cano:mfdmap}, we fix the following setting, which allows a canonical
manifold structure on $C^\ell(K,M)$ to be constructed.
\begin{numba}\label{thesetA}
We consider a compact smooth manifold $K$ (possibly with rough boundary),
a smooth manifold $M$ which admits a local addition $\Sigma \colon TM\supseteq U \rightarrow M$,
and $\ell \in \N_0 \cup \{\infty\}$.
\end{numba}
\begin{numba}[Manifold structure on $C^\ell (K,M)$]\label{numba:mfdstruct}
For $f\in C^\ell(K,M)$, the locally convex space of $C^\ell$-sections
of the pullback-vector bundle $f^*(TM)$ is isomorphic to
\[
\Gamma_f:=\{\tau\in C^\ell(K,TM)\colon \pi_{TM}\circ \tau=f\},
\]
as explained in~\ref{bundleconv}.
Then
\[
O_f:=\Gamma_f\cap C^\ell(K,U)
\]
is an open subset of~$\Gamma_f$,
\[
O_f':=\{g\in C^\ell(K,M)\colon (f,g)(K)\subseteq U'\}
\]
is an open subset of $C^\ell(K,M)$ and the map
\begin{equation}\label{thephif}
\phi_f\colon O_f\to O'_f,\quad\tau\mto \Sigma\circ \tau
\end{equation}
is a homeomorphism with inverse $g\mto \theta^{-1}\circ (f,g)$.
By the preceding, if also $h\in C^\ell(K,M)$, then $\psi:=\phi_h^{-1}\circ \phi_f$
has an open domain $D\subseteq \Gamma_f$ and is a smooth map $D\to\Gamma_h$
by Lemma~\ref{expsect}, as $\psi^\wedge\colon D\times K\to TM$,
\[
(\tau,x)\mto (\phi_h^{-1}\circ \phi_f)(\tau)(x)=\theta^{-1}(h(x),\Sigma(\tau(x)))
=\theta^{-1}(h(x),\Sigma(\ve(\tau,x)))
\]
is a $C^{\infty,\ell}$-map (exploiting that the evaluation map
$\ve\colon \Gamma_f\times K\to TM$
is $C^{\infty,\ell}$, by Lemma~\ref{evalsect}).
Hence $C^\ell(K,M)$ has a smooth manifold structure for which each of the maps
$\phi_f^{-1}$ is a local chart.
\end{numba}
We now prove that the manifold structure on $C^\ell (K,M)$ is canonical.
Together with Lemma~\ref{base-cano}\,(b), this implies
that the smooth manifold structure on $C^\ell(K,M)$ constructed in \ref{numba:mfdstruct}
is independent of the choice of local addition.
\begin{la}\label{la:cano:mfdmap}
The manifold structure on $C^\ell (K,M)$ constructed in \ref{numba:mfdstruct} is canonical.
\end{la}
\begin{proof}
We first show that the evaluation map $\ev\colon C^\ell(K,M)\times K\to M$
is $C^{\infty,\ell}$. 
It suffices to show that $\ev(\phi_f(\tau),x)$ is $C^{\infty,\ell}$
in $(\tau,x)\in O_f\times K$
for all $f\in C^\ell(K,M)$. This follows from
\[
\ev(\phi_f(\tau),x)=\Sigma(\tau(x))=\Sigma(\ve(\tau,x)),
\]
where $\ve\colon\Gamma_f\times K\to TM$, $(\tau,x)\mto\tau(x)$
is $C^{\infty,\ell}$ by Lemma~\ref{evalsect}.
Now let $k\in\N_0\cup\{\infty\}$ and $h\colon N\to C^\ell(K,M)$ be a map, where
$N$ is a $C^k$-manifold modelled on locally convex spaces
(possibly with rough boundary). If $h$ is~$C^k$, then $h^\wedge=\ev\circ (h\times\id_K)$
is~$C^{k,\ell}$.
If, conversely, $h^\wedge$ is $C^{k,\ell}$,
then $h$ is continuous as a map to $C(K,M)$ with the compact-open topology
(see \cite[Proposition~A.5.17]{GaN})
and $h(x)=h^\wedge(x,\cdot)\in C^\ell(K,M)$ for each $x\in N$.
Given $x\in N$, let $f:=h(x)$. Then $\psi_f\colon C(K,M)\to C(K,M)\times C(K,M)\sim C(K,M\times M)$,
$g\mto (f,g)$ is a continuous map.
Since $\psi_f(g)$ is $C^\ell$ if and only if~$g$ is $C^\ell$, we see that
\begin{eqnarray*}
W &:=& h^{-1}(O_f')=h^{-1}(\psi_f^{-1}(C^\ell(K,U')))=(\psi_f\circ h)^{-1}(C^\ell(K,U'))\\
&=&(\psi_f\circ h)^{-1}(C(K,U'))
\end{eqnarray*}
is an open $x$-neighborhood in~$N$. As the map $(\phi_f^{-1}\circ h|_W)^\wedge\colon
W\times K\to TM$,
\[
(y,z)\mto ((\phi_f)^{-1}\circ h|_W)^\wedge(y,z)=(\theta^{-1}\circ (f,h(y)))(z)=\theta^{-1}(f(z),h^\wedge(y,z))
\]
is $C^{k,\ell}$ by \cite[Lemma 3.18]{AaS},
the map $\phi_f^{-1}\circ h|_W\colon W\to\Gamma_f$ (and hence also $h|_W$)
is~$C^k$, by Lemma~\ref{expsect}.
\end{proof}
\subsection*{The tangent bundle of {\boldmath$C^\ell(K,M)$}}
We now identify the tangent bundle of a manifold of mappings,
as well as the tangent maps of superposition operators between
such manifolds. We start with a description of the main steps and ideas;
three more technical proofs
(of Lemma~\ref{addTM}, Lemma~\ref{cannormalize},
and Theorem~\ref{thetabu}) are relegated to the following subsection.
An observation is crucial:
\begin{la}[cf.\ {\cite[Lemma 7.5]{SaW} or \cite[10.11]{Mic}}]\label{addTM}
If a smooth manifold $M$ admits a
local addition, then also its tangent manifold $TM$
admits a local addition.
\end{la}
Since $M$ admits a local addition in the setting of~\ref{thesetA},
we deduce from Lemmas~\ref{addTM} and \ref{la:cano:mfdmap}
that also $C^\ell(K,TM)$
admits a canonical smooth manifold structure. By
Proposition~\ref{fstar-gen},
the map
\[
C^\ell(K,\pi_{TM})\colon C^\ell(K,TM)\to C^\ell(K,M),\quad\tau\mto \pi_{TM}\circ \tau
\]
is smooth. For each $f\in C^\ell(K,M)$, we have
$C^\ell(K,\pi_{TM})^{-1}(\{f\})=\Gamma_f$,
which we endow with a vector space structure as in~\ref{bundleconv}.
Given $x\in K$, let $\ve_x\colon C^\ell(K,M)\to M$, $f\mto f(x)$
be the point evaluation at~$x$.
Then the tangent bundle
of $C^\ell(K,M)$ can be described as follows:
\begin{thm}\label{thetabu}
In the situation of \ref{thesetA},
\[
C^\ell(K,\pi_{TM})\colon C^\ell(K,TM)\to C^\ell(K,M)
\]
is a smooth vector bundle with fiber $\Gamma_f$ over $f\in C^\ell(K,M)$.
For each $v\in T(C^\ell(K,M))$, we have $\Phi(v):=(T\ve_x(v))_{x\in K}\in C^\ell(K,TM)$
and the map
\[
\Phi\colon TC^\ell(K,M)\to C^\ell(K,TM),\quad v\mto \Phi(v)
\]
is an isomorphism of smooth vector bundles $($over the identity$)$.
\end{thm}
If we wish to emphasize the dependence on $M$, we write $\Phi_M$
instead of~$\Phi$.
\begin{rem}\label{alsothis}
(a) Assume that the local addition $\Sigma\colon U\to M$ is normalized
in the sense of (\ref{bettersigma}). 
Then the proof of Theorem~\ref{thetabu}
will show that
\[
\Phi\circ T\phi_f(0,\cdot)\colon\Gamma_f\to C^\ell(K,TM)
\]
is the inclusion map $\tau\mto \tau$,
for each $f\in C^\ell(K,M)$ (where $\phi_f$ is as in (\ref{thephif})).\smallskip

(b) Compare \cite[Theorem 10.13]{Mic} for a special case of Theorem~\ref{thetabu}
for finite-dimensional~$M$ and $\ell=\infty$
(and \cite[Theorem 7.9]{SaW} for additional explanations concerning Michor's discussion).
\end{rem}
By the preceding, it is useful to work with normalized
local additions. This is no further restriction:
\begin{la}\label{cannormalize}
If a smooth manifold $M$ admits a local addition,
then~$M$ also admits a normalized local addition.
\end{la}
Theorem~\ref{thetabu} allows us to identify the tangent maps
of superposition operators.
\begin{cor}\label{formtang}
Let $K$ be a compact smooth manifold $($possibly with rough boundary$)$
and $g\colon M\to N$ be a $C^{\ell+1}$-map between smooth manifolds~$M$
and $N$ admitting local additions.
Then the tangent map of the $C^1$-map
\[
C^\ell(K,g)\colon C^\ell(K,M)\to C^\ell(K,N),\quad f\mto g\circ f
\]
is given by $T(C^\ell(K,g))=\Phi_N^{-1}\circ C^\ell(K,Tg)\circ \Phi_M$.
For each $f\in C^\ell(K,M)$, we have $\Phi_M(T_f(C^\ell(K,M)))=\Gamma_f(TM)$,
$\Phi_N(T_{g\circ f}(C^\ell(K,N)))=\Gamma_{g\circ f}(TN)$
and $C^\ell(K,Tg)$ restricts to the map
\begin{equation}\label{nearly}
\Gamma_f(TM)\to\Gamma_{g\circ f}(TN),\quad \tau\mto Tg\circ \tau
\end{equation}
which is continuous linear and corresponds to $T_f(C^\ell(K,g))$.
\end{cor}
\begin{rem}
In the following proof and also in the following subsection,
we find it convenient to consider the tangent space $T_pM$ of a smooth
manifold~$M$ at $p\in M$ as a \emph{geometric}
tangent space.
Thus, the elements of $T_pM$
are geometric tangent vectors, i.e., equivalence classes $[\gamma]$
of $C^1$-curves $\gamma\colon \;]{-\ve},\ve[\,\to M$ with $\gamma(0)=p$
(where two such curves are considered equivalent if their velocity at $t=0$ coincides
in each chart). In the case of an open subset $U$ of a vector space~$E$,
we shall also identify $TU$ with $U\times E$ (as usual).
\end{rem}
\begin{proof}[Proof of Corollary \ref{formtang}]
Let $\Sigma$ be a normalized local addition on~$M$ (cf.\ Lemma~\ref{cannormalize}).
Given $f\in C^\ell(K,M)$, the map $T\phi_f(0,\cdot)\colon \Gamma_f\to T_f(C^\ell(K,M))$
is an isomorphism of vector spaces. The assertions follow from
\begin{align*}
&\Phi_N(T(C^\ell(K,g))T\phi_f(0,\sigma))\\
=& ([t\mto g(\Sigma(t\sigma(x)))])_{x\in K}
= \left(Tg{\textstyle \frac{d}{dt}}\big|_{t=0}\Sigma(t\sigma(x))\right)_{x\in K}\\
=& (Tg(\sigma(x)))_{x\in K}
=Tg\circ \sigma=C^\ell(K,Tg)(\sigma)
= C^\ell(K,Tg)\Phi_MT\phi_f(0,\sigma),
\end{align*}
where we used (\ref{bettersigma})
for the third equality and Remark~\ref{alsothis}\,(a) for the last.
\end{proof}
We mention a further consequence.
\begin{rem}
If $M$ has the local addition $\Sigma\colon U\to M$,
then $C^\ell(K,U)$ is an open subset of $C^\ell(K,TM)$,
which we identify with $T(C^\ell(K,M))$ by means of the map $\Phi$
from Theorem~\ref{thetabu}.
Using this identification, the map $C^\ell(K,\Sigma)\colon C^\ell(K,U)\to C^\ell(K,M)$,
$\gamma\mto \Sigma\circ\gamma$ is easily seen to be a local addition for
$C^\ell(K,M)$.
If $N$ is a compact smooth manifold (possibly with rough boundary) and $k\in \N_0\cup\{\infty\}$,
this enables a canonical smooth manifold structure to be defined on $C^k(N,C^\ell(K,M))$.
By a variant of the construction described in this appendix,
it is possible to endow also $C^{k,\ell}(N\times K,M)$ with a suitable canonical
smooth manifold structure, and to obtain an exponential law of the form
\[
C^{k,\ell}(N\times K,M)\cong C^k(N,C^\ell(K,M));
\]
notably, $C^k(N,C^\ell(K,M))\cong C^\ell(K,C^k(N,M))$ (joint work in progress
by the second and third author).
\end{rem}
\subsection*{Proof of Lemma \ref{cannormalize}, Lemma~\ref{addTM}, and Theorem~\ref{thetabu}}
We now fill in the three proofs which have been postponed
in the preceding subsection.

The following notation is useful: If $U_j$ is an open subset of
a locally convex space $E_j$ for $j\in\{1,2\}$
and $f\colon U_1\times U_2\to F$
is a $C^1$-map to a locally convex space~$F$, we abbreviate
\[
d_2f(x,y,z):=df((x,y),(0,z))
\]
for $(x,y)\in U_1\times U_2$ and $z\in E_2$.%\\[2.3mm]
\begin{proof}[Proof of Lemma~\ref{cannormalize}]
Given a local addition
$\Sigma\colon U\to M$, let us use the notation from~\ref{def-loa}.
The $C^\infty$-diffeomorphism
$\theta:=(\pi_{TM},\Sigma)\colon U\to U'\subseteq M\times M$
takes the open set $T_pM\cap U\subseteq T_pM$ to the submanifold
$(\{p\}\times M)\cap U'$ of~$M\times M$ for each $p\in M$
and restricts to a $C^\infty$-diffeomorphism between these sets.
Hence $\Sigma|_{T_pM\cap U}$ is a $C^\infty$-diffeomorphism
onto an open subset of~$M$, whence
$\alpha_p:=T_{0_p}(\Sigma|_{T_pM})\in\GL(T_pM)$ for each $p\in M$. Define
\[
h \colon TM\to TM,\;\, h(v):= \alpha_{\pi_{TM}(v)}^{-1}(v).
\]
We claim that $h$ is smooth. If this is true,
then
$\Sigma\circ h\colon h^{-1}(U)\to M$
is a normalized local addition. To prove the claim,
let us show smoothness of $T\phi\circ h \circ T\phi^{-1}$
for a given chart $\phi\colon U_\phi\to V_\phi\subseteq E_\phi$ of~$M$.
Abbreviate
\[
P:=\theta^{-1}(U_\phi\times U_\phi) \cap TU_\phi
\]
and $Q:=T\phi(P)$. Then $0_p\in P$ for each $p\in U_\phi$
as $\theta(0_p)=(p,p)\in U_\phi\times U_\phi$
and $0_p\in TU_\phi$. As a consequence, $V_\phi\times \{0\}\subseteq Q$.
Then
\[
f:=(\phi\times \phi )\circ \theta|_P \circ T\phi^{-1}|_Q^P\colon Q\to E_\phi\times E_\phi
\]
is a $C^\infty$-diffeomorphism onto the open subset $(\phi\times\phi)(\theta(P))$
of~$E_\phi\times E_\phi$ and $f$ is of the form $f(x,y)=(x,g(x,y))$ with a smooth function
$g\colon Q\to E_\phi$. By the following lemma, $\beta_x:=d_2g(x,0,\cdot)\in \GL(E_\phi)$
for all $x\in V_\phi$ and the map
\[
V_\phi\times E_\phi\to E_\phi,\quad (x,z)\mto \beta_x^{-1}(z)
\]
is smooth. Given $(x,y)\in V_\phi\times E_\phi$, let $w:=T\phi^{-1}(x,y)$,
$p:=\phi^{-1}(x)$, $v:=\alpha_p^{-1}(w)$ and $z:=d\phi(v)$ (whence $v=T\phi^{-1}(x,z)$).
Since
\[
g(x,u)=\phi(\Sigma|_{T_pM}T\phi^{-1}(x,\cdot)|_{E_\phi}^{T_pM}(u))
\]
for $u\in E_\phi$
close to~$0$, we deduce that
\[
\beta_x(z)=d_2g(x,0,z)=T\phi(\alpha_p(T\phi^{-1}(x,z)))=T\phi(\alpha_p(v))=T\phi(w)=(x,y).
\]
Thus,
$d\phi(v)=z=\beta_x^{-1}(x,y)$ and $(T\phi\circ h\circ T\phi^{-1})(x,y)
=T\phi (\alpha_p^{-1}(w))=T\phi(v)=(x,\beta_x^{-1}(x,y))$,
which is a smooth $E_\phi\times E_\phi$-valued function of $(x,y)\in V_\phi\times E_\phi$.
\end{proof}
\begin{la}\label{localdiffinv}
Let $E$ be a locally convex space, $W\subseteq E$ be an open subset
and $Q\subseteq E\times E$ be an open subset such that $W\times\{0\}\subseteq Q$.
Let $g\colon Q\to E$ be a smooth mapping such that
\[
f\colon Q\to E\times E,\quad (x,y)\mto (x,g(x,y))
\]
has open image and is a $C^\infty$-diffeomorphism onto its image.
Then $\beta_x:=d_2g(x,0,\cdot)\in \GL(E)$ for each $x\in W$ and the map
$
W\times E\to E,\quad (x,z)\mto \beta_x^{-1}(z)
$
is smooth.
\end{la}
\begin{proof}
Given $x\in W$, the hypotheses imply that the map
\[
g_x:=g(x,\cdot)\colon
\{y\in E\colon (x,y)\in Q\}\to E, \;\,
y\mto g(x,y)
\]
is a $C^\infty$-diffeomorphism onto $\{z\in E\colon (x,z)\in f(Q)\}$.
Hence $\beta_x:=d_2g(x,0,\cdot)=d(g_x)(0,\cdot)\in \GL(E)$.
Write $h$ for the second component of $f^{-1}$.
Given $x\in W$, we have
$h(x,g(x,y))=y$ for all $y\in E$ such that $(x,y)\in Q$ (notably for $y=0$),
whence $d_2h(x,g(x,0),d_2g(x,0,v))=v$
for all $v\in E$ and thus
\[
d_2h(x,g(x,0),\cdot)\circ \beta_x=\id_E.
\]
Hence $\beta_x^{-1}(z)=d_2h(x,g(x,0),z)$, which is smooth in $(x,z)\in W\times
E$.
\end{proof}
As a tool for the proofs of Lemma~\ref{addTM} and Theorem~\ref{thetabu},
we recall the definition of the canonical flip on $T^2M:=T(TM)$,
and some of its properties.
\begin{numba}
Consider a smooth manifold $M$
and the bundle projection $\pi_{TM}\colon TM\to M$.
Then $T^2M$ is a smooth vector bundle over~$TM$;
we write $\pi_{T^2M}\colon T^2M\to TM$ for its bundle
projection.
Given a chart $\phi\colon U_\phi\to V_\phi\subseteq E_\phi$ of~$M$
and $(x,y,z,w)\in V_\phi\times E_\phi\times E_\phi\times E_\phi$, we define
\[
\kappa (T^2(\phi^{-1})(x,y,z,w)):=T^2(\phi^{-1})(x,z,y,w).
\]
It is easy to check that a well-defined smooth map
\[
\kappa\colon T^2M\to T^2M
\]
is obtained in this way
(the \emph{canonical flip}), such that $\kappa\circ\kappa=\id_{T^2M}$.\\[2.3mm]
[In fact, if $T^2(\phi^{-1})(x,y,z,w)=T^2(\psi^{-1})(x',y',z',w')$,
then
\begin{eqnarray*}
(x',y',z',w') &=& T^2(f)(x,y,z,w)\\
&=& (f(x),df(x,y),df(x,z),df(x,w)+d^{(2)}f(x,y,z))
\end{eqnarray*}
with $f:=\psi\circ \phi^{-1}$ and thus
\[
T^2(\psi^{-1})(x',z',y',w')
=T^2(\psi^{-1})(f(x),df(x,z),df(x,y),df(x,w)+d^{(2)}f(x,y,z)),
\]
which coincides with
\begin{eqnarray*}
T^2(\phi^{-1})(x,z,y,w)&=&T^2(\psi^{-1})T^2(f)(x,z,y,w)\\
&=& T^2(\psi^{-1})(f(x),df(x,z),df(x,y),df(x,w)+d^{(2)}f(x,z,y)).]
\end{eqnarray*}
Using a local chart, one readily verifies that
\begin{equation}\label{firstkapp}
\pi_{T^2M}=(T\pi_{TM})\circ \kappa.
\end{equation}
[As $(\phi\circ\pi_{TM}\circ T(\phi^{-1}))(x,y)=x$,
we have $T(\phi\circ\pi_{TM}\circ T(\phi^{-1}))(x,y,z,w)=
(x,z)$ and thus
$(T(\phi)\circ T(\pi_{TM})\circ \kappa \circ T^2(\phi^{-1}))(x,y,z,w)
=T(\phi\circ \pi_{TM}\circ T(\phi^{-1}))(x,z,y,w)=(x,y)
=(T\phi \circ \pi_{T^2M}\circ T^2(\phi^{-1}))(x,y,z,w)$.]
\end{numba}
\begin{proof}[Proof of Lemma~\ref{addTM}]
If $\Sigma\colon U\to M$ is a local addition for~$M$
and $\theta=(\pi_{TM},\Sigma)\colon U\to U'\subseteq M\times M$ (as above) the associated
$C^\infty$-diffeomorphism, then $TU$ is open in $T^2M$,
the set $T(U')$ is open in $T(M\times M)$ (which we identify with $TM\times TM$
via $(T\pr_1,T\pr_2)$) and
\[
T\theta\colon TU\to TU'
\]
is a $C^\infty$-diffeomorphism.
Then also
\[
(T\theta)\circ\kappa\colon \kappa(TU)\to TU'\subseteq TM\times TM
\]
is a $C^\infty$-diffeomorphism and
$(T\theta)\circ\kappa=(\pi_{T^2M},\Sigma_{TM})$
with
\[
\Sigma_{TM}:=(T\Sigma)\circ \kappa\colon \kappa(TU)\to TM.
\]
We shall readily check that $0_v\in \kappa(TU)$ for all $v\in TM$
and $\Sigma_{TM}(0_v)=v$,
whence $\Sigma_{TM}$ is a local addition for~$TM$.
[Given $p\in M$, let $\phi\colon U_\phi\to V_\phi\subseteq E_\phi$ be a chart for~$M$
such that $p\in U_\phi$ and $\phi(p)=0$. Set $P:=U\cap TU_\phi$ and $Q:=(T\phi)(P)$.
Since $0_p\in U$, we have
$(0,0)\in Q$,
whence $\{0\}\times\{0\}\times E_\phi\times E_\phi\subseteq TQ=(T^2\phi)(P)$
and thus $T^2\phi^{-1}(\{0\}\times E_\phi\times\{0\}\times E_\phi)
=\kappa(T^2\phi^{-1}(\{0\}\times\{0\}\times E_\phi\times E_\phi))
\subseteq \kappa(TU)$, entailing that $0_v\in \kappa(TU)$ for all $v\in T_pM$.\\[2.3mm]
To see that $\Sigma_{TM}(0_v)=v$ for all $v\in T_pM$,
note that $T^2\phi(\kappa(0_v))=(0,0,y,0)$ for some
$y\in E_\phi$.
Now
\begin{eqnarray*}
v&=&\pi_{T^2M}(0_v)=T(\pi_{TM})(\kappa(0_v))=T(\pi_{TM})T^2\phi^{-1}(0,0,y,0)\\
&=&
T(\pi_{TM}\circ T\phi^{-1})(0,0,y,0)=T\phi^{-1}(0,y)
\end{eqnarray*}
since $\pi_{TM}\circ T\phi^{-1}(z,0)=\phi^{-1}(z)$ for all $z\in V_\phi$
and thus $T(\pi_{TM}\circ T\phi^{-1})(0,0,y,0)$ $=T\phi^{-1}(0,y)$.
As a consequence,
\begin{eqnarray*}
\Sigma_{TM}(0_v)&=& T\Sigma (\kappa(0_v))=T(\Sigma\circ T\phi^{-1})(0,0,y,0)\\
&=& {\textstyle \frac{d}{dt}}\big|_{t=0}\Sigma T\phi^{-1}(ty,0)
={\textstyle \frac{d}{dt}}\big|_{t=0}\Sigma(0_{\phi^{-1}(ty)})\\
&=& {\textstyle \frac{d}{dt}}\big|_{t=0}\phi^{-1}(ty)
=T\phi^{-1}(0,y)=v,
\end{eqnarray*}
which completes the proof.]
\end{proof}
The following considerations prepare the proof of Theorem~\ref{thetabu}.
Let $M$ be a smooth manifold modelled on locally convex spaces.
In the following proofs, given $p\in M$ we write
$\lambda_p\colon T_pM\to TM$, $v\mto v$ for the inclusion map.
We abbreviate $T^2M:=T(TM)$ and let $\kappa\colon T^2M\to T^2M$
be the canonical flip.
The zero-section $0_M:=\{0_p\in T_pM\colon p\in M\}$ is a split submanifold of~$TM$.
As the bundle projection $\pi_{T^2M}\colon T^2(M)\to TM$ is a smooth submersion,
\cite[Theorem C]{SUB} shows that
the preimage
\[
\pi_{T^2M}^{-1}(0_M)
\]
is a split submanifold of $T^2M$. We can also see this by hand: If $\phi\colon U_\phi\to V_\phi\subseteq E_\phi$
is a chart for~$M$, then $T^2\phi\colon T^2U_\phi \to T^2 V_\phi=V_\phi\times E_\phi\times E_\phi\times E_\phi$
is a chart for $T^2M$ and
\begin{equation}\label{submcha}
T^2\phi(T^2(U_\phi)\cap \pi_{T^2M}^{-1}(0_M))=V_\phi\times \{0\}\times E_\phi\times E_\phi
=T^2(V_\phi)\cap (E_\phi\times \{0\}\times E_\phi\times E_\phi),
\end{equation}
where $E_\phi\times \{0\}\times E_\phi\times E_\phi$ is a complemented topological vector subspace
of $E_\phi\times E_\phi\times E_\phi\times E_\phi$.
Define
\[
\pi\colon \pi_{T^2M}^{-1}(0_M)\to M,\quad v\mto\pi_{TM}(\pi_{T^2M}(v)).
\]
For $p\in M$, we give $\pi^{-1}(\{p\})=T_{0_p}(TM)$ the vector space structure
as the tangent space of the smooth manfold~$TM$ at~$0_p$.
Then $d(T\phi)$ restricts to a linear isomorphism $\pi^{-1}(\{p\})=T_{0_p}(TM)\to E_\phi\times E_\phi$
for each chart $\phi$ as before and $p\in U_\phi$. As a consequence,
each of the the $C^\infty$-diffeomorphisms
$(\pi,d(T\phi))\colon \pi^{-1}(U_\phi)\to U_\phi\times
E_\phi\times E_\phi$ is a local trivialization and
$\pi_{T^2M}^{-1}(0_M)$ is a smooth vector bundle.
\begin{la}\label{infra}
In the preceding situation, the following holds:
\begin{itemize}
\item[\textup{(a)}]
$\Theta(v,w):=\kappa(T\lambda_p(v,w))\in T_{0_p}(TM)$ for all $p\in M$ and $v,w\in T_pM$;
\item[\textup{(b)}]
The map $\Theta\colon TM\oplus TM\to \pi_{T^2M}^{-1}(0_M)$
is an isomorphism of $C^\infty$-vector bundles over~$\id_M$;
\item[\textup{(c)}]
If $U\subseteq TM$ is an open subset and $U\times_M TM:=
\bigcup_{p\in M} (U\cap T_pM)\times T_pM\subseteq TM\oplus TM$,
then $\Theta(U\times_M TM)=\kappa(TU)\cap\pi_{T^2M}^{-1}(0_M)$.
\end{itemize}
Now let $K$ be a compact smooth manifold $($possibly with rough boundary$)$
and $\ell\in \N_0\cup\{\infty\}$. Let $\Gamma_f:=\{\tau\in C^\ell(K,TM)\colon (\forall x\in K)\;
\tau(x)\in T_{f(x)}M\}$,
\[
\Gamma_{0\circ f}:=\{\tau\in C^\ell(K,T^2M)\colon (\forall x\in K)\; \tau(x)\in T_{0_{f(x)}}TM\},
\]
$O_f:=\{\tau\in\Gamma_f\colon \tau(K)\subseteq U\}$ and $O_{0\circ f}:=\{\tau\in\Gamma_{0\circ f}\colon
\tau(K)\subseteq \kappa(TU)\}$. Then the following holds:
\begin{itemize}
\item[\textup{(d)}]
$\Theta\circ(\sigma,\tau)\in \Gamma_{0\circ f}$ for all $\sigma,\tau\in\Gamma_f$
and the map
\begin{equation}\label{ourmp}
\Gamma_f\times \Gamma_f\to \Gamma_{0\circ f},\quad (\sigma,\tau)\mto\Theta\circ(\sigma,\tau)
\end{equation}
is an isomorphism of topological vector spaces.
\item[\textup{(e)}]
The isomorphism from \emph{(d)} restricts to a $C^\infty$-diffeomorphism
$\Psi_f\colon O_f\times\Gamma_f\to O_{0\circ f}$.
\end{itemize}
\end{la}
\begin{proof}
(a) and (b):
Let $\phi\colon U_\phi\to V_\phi\subseteq E_\phi$ be a chart for~$M$
and $x\in V_\phi$. Abbreviate $p:=\phi^{-1}(x)$.
The map $\alpha \colon E_\phi\to T_pM$,
$y\mto (T\phi^{-1})(x,y)$ is an isomorphism of topological vector spaces.
Since
\[
(T\phi\circ \lambda_p\circ \alpha)(y)=(x,y)
\]
for all $y\in E_\phi$, we have
$(T^2\phi\circ T\lambda_p)(T\phi^{-1}(x,y),T\phi^{-1}(x,z))
=(T^2\phi\circ T\lambda_p\circ T\alpha)(y,z)=(x,y,0,z)$ for all $y,z\in E_\phi$
and hence
\begin{equation}\label{maybenee}
(T^2\phi\circ \kappa\circ T\lambda_p)(T\phi^{-1}(x,y),T\phi^{-1}(x,z)))=(x,0,y,z).
\end{equation}
Writing $v:=T\phi^{-1}(x,y)$ and $w:=T\phi^{-1}(x,z)$, we deduce
that $\kappa(T\lambda_p(v,w))=T^2\phi^{-1}(x,0,y,z)\in T_{0_p}(TM)$, establishing~(a).
It follows from~(\ref{maybenee})
that the map $T_pM\times T_pM\to T_{0_p}(TM)$, $(v,w)\mto \kappa(T\lambda_p(v,w))$
is a bijection. As a consequence, $\Theta$ is a bijection.
Given a chart $\phi\colon U_\phi\to V_\phi\subseteq E_\phi$ of~$M$, the map
\[
T\phi\oplus T\phi\colon (TM\oplus TM)|_{TU_\phi}\to V_\phi\times E_\phi\times E_\phi,\;\;
(v,w)\mto (T\phi(v),d\phi(w))
\]
is a chart for $TM\oplus TM$.
Using (\ref{maybenee}), we find that
\[
(T^2\phi\circ \Theta\circ (T\phi \oplus T\phi)^{-1})(x,y,z)=(T^2\phi\circ\Theta)(T\phi^{-1}(x,y),
T\phi^{-1}(x,z))=(x,0,y,z)
\]
for $(x,y,z)\in V_\phi\times E_\phi\times E_\phi$, which is a $C^\infty$-diffeomorphism
from $V_\phi\times E_\phi\times E_\phi$ onto $V_\phi \times\{0\}\times E_\phi\times E_\phi=
T^2 V_\phi\cap (E_\phi\times \{0\}\times E_\phi\times E_\phi)$.
As $T^2\phi$ restricts to a chart of the submanifold $\pi_{T^2M}^{-1}(0_M)$
(cf.\ (\ref{submcha})), we deduce that $\Theta$ restricts to a $C^\infty$-diffeomorphism
$(TM\oplus TM)|_{TU_\phi}\to (T^2 U_\phi)\cap \pi_{T^2M}^{-1}(0_M)$.
Now $\Theta((TM\oplus TM)_p)\subseteq \pi^{-1}(\{p\})$ for each $p\in M$, by (a).
Since $(d(T\phi)\circ \Theta\circ (T\phi \oplus T\phi)^{-1})(x,y,z)=(y,z)$
is linear in $(y,z)$, we deduce that the $C^\infty$-diffeomorphism $\Theta$ is an
isomorphism of smooth vector bundles over~$\id_M$.

(c)
Let $p\in M$ and $v,w\in T_pM$. Pick a chart $\phi\colon U_\phi\to V_\phi\subseteq E_\phi$
with $p\in U_\phi$ and set $x:=\phi(p)$.
Then $v=T\phi^{-1}(x,y)$ and $w=T\phi^{-1}(x,z)$ with suitable $y,z\in E_\phi$
and
\begin{align*}
&\Theta(v,w)\in \kappa(TU)\cap T_{0_p}(TM) \\
 \Leftrightarrow &
\Theta(v,w)\in \kappa(TU)
\Leftrightarrow  T\lambda_p(v,w)\in TU\cap T^2 U_\phi=T(U\cap TU_\phi)\\
\Leftrightarrow& (x,y,0,z)=T^2\phi T\lambda_p(v,w)\in T^2\phi(TU\cap T^2 U_\phi)
\Leftrightarrow(x,y)\in T\phi(U\cap TU_\phi)\;\Leftrightarrow \; v\in U.
\end{align*}
Thus $\Theta^{-1}(\kappa(TU)\cap \pi_{T^2M}^{-1}(0_M))\cap (T_pM\times
T_pM)=(T_pM\cap U)\times T_pM$ and the assertion follows.

(d)
We have $\Gamma_{0\circ f}:=\Gamma_{0\circ f}(T^2M)
=\Gamma_f(\pi_{T^2M}^{-1}(0_M))$
as a set and as a topological space, as a consequence of \ref{into-sub}
and Lemma~\ref{embpfwd}.
Since $\pi^{-1}(\{p\})=T_{0_p}(TM)$ as a vector space for each $p\in M$,
a pointwise calculation shows that $\Gamma_{0\circ f}(T^2M)=\Gamma_f(\pi_{T^2M}^{-1}(0_M))$
also as vector spaces, and hence as locally convex spaces.
Identifying $\Gamma_f\times\Gamma_f$ with $\Gamma_f(TM\oplus TM)$
as in Lemma~\ref{Gammfunct}\,(b), we have
\[
\Theta\circ (\sigma,\tau)=\Gamma_f(\Theta)(\sigma,\tau)\in \Gamma_f(\pi_{T^2M}^{-1}(0_M))
\]
for all $(\sigma,\tau)\in \Gamma_f\times\Gamma_f$, by Lemma~\ref{Gammfunct}\,(a).
The map (\ref{ourmp}) to $\Gamma_f(\pi_{T^2M}^{-1}(0_M))=\Gamma_{0\circ f}(T^2M)$
therefore coincides with $\Gamma_f(\Theta)$,
which is an isomorphism of locally convex spaces (with inverse $\Gamma_f(\Theta^{-1})$)
by Lemma~\ref{Gammfunct}\,(a).

(e) Given $\sigma,\tau \in \Gamma_f$, we have $\Theta\circ (\sigma,\tau)\in O_{0\circ f}$
if and only if $\Theta(\sigma(x),\tau(x))\in
\kappa(TU)$
for all $x\in K$. By (c), this holds if and only if $\sigma(x)\in U$ for all $x\in K$,
i.e., if and only if $\sigma\in O_f$. Thus $\{\Theta\circ (\sigma,\tau)\colon
(\sigma,\tau)\in O_f\times \Gamma_f\}=O_{0\circ f}$.
\end{proof}

\noindent
\begin{proof}[Proof of Theorem~\ref{thetabu}]
Given $f\in C^\ell(K,M)$, the map $\phi_f\colon O_f\to O_f'\subseteq C^\ell(K,M)$
is a $C^\infty$-diffeomorphism with $\phi_f(0)=f$, whence
$
T\phi_f(0,\cdot)\colon\Gamma_f\to T_f(C^\ell(K,M))
$
is an isomorphism of topological vector spaces. For $\tau\in\Gamma_f$,
we have for each $x\in K$
\begin{eqnarray*}
T\ve_xT\phi_f(0,\tau)&=&T\ve_x([t\mto \Sigma\circ (t\tau)])
=[t\mto \Sigma(t\tau(x))]\\
&=& [t\mto\Sigma|_{T_{f(x)}M}(t\tau(x))]
=T\Sigma|_{T_{f(x)}M}(\tau(x))=\tau(x),
\end{eqnarray*}
as $\Sigma$ is assumed normalized. Thus $\Phi(T\phi_f(0,\tau))=\tau\in\Gamma_f\subseteq
C^\ell(K,TM)$, whence $\Phi(v)\in \Gamma_f\subseteq C^\ell(K,TM)$
for each $v\in T_f(C^\ell(K,M))$ and $\Phi$ takes $T_f(C^\ell(K,M))$ bijectively
and linearly onto~$\Gamma_f$. As $T(C^\ell(K,M))$ and $C^\ell(K,TM)$
is the disjoint union of the sets $T_f(C^\ell(K,M))$ and $\Gamma_f=C^\ell(K,\pi_{TM})^{-1}(\{f\})$,
respectively, we see that $\Phi$ is a bijection. If we can show that $\Phi$ is
a $C^\infty$-diffeomorphism, it will also follow from the preceding that $C^\ell(K,\pi_{TM})\colon
C^\ell(K,TM)\to C^\ell(K,M)$ is a smooth vector bundle over $C^\ell(K,M)$ (like $T(C^\ell(K,M))$)
and $\Phi$ an isomorphism of
smooth vector bundles over~$\id_M$.\\[2.3mm]
The bijective map $\Phi$ will be a $C^\infty$-diffeomorphism
if we can show that\footnote{The sets $S_f:=T\phi_f(O_f\times \Gamma_f)$
form an open cover of $T(C^\ell(K,M))$ for $f\in C^\ell(K,M)$,
whence the sets $\Phi(S_f)$
form a cover of $C^\ell(K,TM)$ by sets which are open as $\Phi(S_f)
=(\phi_{0\circ f}\circ \Psi_f)(O_f\times \Gamma_f)=\phi_{0\circ f}(O_{0\circ f})$.}
\[
\Phi\circ T\phi_f=\phi_{0\circ f}\circ \Psi_f
\]
for each $f\in C^\ell(K,M)$, where
$\Psi_f\colon O_f\times\Gamma_f \to O_{0\circ f}$ is the $C^\infty$-diffeomorphism
from Lemma~\ref{infra}\,(e).
Recall that $\lambda_p\colon T_pM\to TM$, $z\mto z$
is the inclusion for $p\in M$. Now
\[
T\phi_f(\sigma,\tau)=[t\mto\Sigma\circ (\sigma+t\tau)]
\]
for all $(\sigma,\tau)\in O_f\times\Gamma_f$, and thus
\begin{align*}
\Phi(T\phi_f(\sigma,\tau))
=& ([t\mto \Sigma(\sigma(x)+t\tau(x))])_{x\in K}
= ([t\mto (\Sigma\circ \lambda_{f(x)})(\sigma(x)+t\tau(x))])_{x\in K}\\
=& (T(\Sigma\circ \lambda_{f(x)})(\sigma(x),\tau(x)))_{x\in K}
= (\Sigma_{TM}((\kappa \circ T\lambda_{f(x)})(\sigma(x),\tau(x))))_{x\in K}\\
=& ((\Sigma_{TM}\circ \Psi_f)(\sigma,\tau)(x))_{x\in K}
=(\phi_{0\circ f}\circ \Psi_f)(\sigma,\tau). \tag*{\qedhere}
\end{align*}
\end{proof}
\section{Current groupoids related to orbifold morphisms}\label{app:orbi}

In this section, the relation of current groupoids of proper \'{e}tale Lie groupoids with orbifolds and morphisms of orbifolds is discussed. 
An orbifold is a generalisation of a manifold allowing mild singularities; we recall from \cite{Sch,MaP,MaM}:

\begin{numba}[Orbifolds in local charts]
Let $Q$ be a Hausdorff topological space. An orbifold chart $(V,G,\pi)$ is a triple, where $V$ is a connected manifold, $G \subseteq \Diff (V)$ a finite subgroup and $\pi \colon V \rightarrow Q$ a continuous map with open image, which induces a homeomorphism $V/G \cong \pi (V)$. Two orbifold charts $(V,G,\pi)$, $(W,H,\psi)$ on $Q$ are \emph{compatible} if for every $\pi (x) = \psi (y)$ there exists a smooth diffeomorphism $\varphi \colon V_x \rightarrow V_y$, \emph{a change of charts} between $x$- and $y$-neighborhoods, such that $\psi \circ \varphi = \pi|_{V_x}$.\footnote{Contrary to manifolds, the change of charts is not given by $\psi^{-1} \circ \pi$ as $\psi$ might not be invertible.} An \emph{orbifold atlas} is a family of pairwise compatible orbifold charts whose images cover $Q$.
\end{numba}

One usually assumes that the manifolds appearing in orbifold charts are paracompact and finite dimensional, i.e.\ the orbifold atlas is finite dimensional. We do not assume this per se. However, recall that every
(finite-dimensional) orbifold atlas gives rise to an atlas groupoid which is a proper \'{e}tale Lie groupoid. 

\begin{numba}[Atlas groupoids]
Consider an orbifold atlas $\mathcal{V} \coloneq \{V_i,G_i,\varphi_i)\}_{i\in I}$ on a topological space $Q$ such that the manifolds $V_i, i\in I$ are finite dimensional. Then we construct a proper \'{e}tale Lie groupoid $\Gamma (\mathcal{V})$, called \emph{atlas groupoid}, as follows. Its space of arrows is given by the disjoint union $\sqcup_{i\in I} V_i$, while the arrows are germs of change of chart morphisms (with the germ topology turning $\Gamma (\mathcal{V})$ into a proper \'{e}tale groupoid). For details we refer to \cite[Theorem 4 (4) $\Rightarrow$ (1)]{MaP} and \cite{Poh}.
\end{numba}

Different (but equivalent) orbifold atlases give rise to different (but Morita equivalent) atlas groupoids. This construction can be reversed, as \cite{MaP,MaM} showed that the orbit space associated to the canonical right action of a (finite-dimensional) proper \'{e}tale Lie groupoid on its space of units gives rise to a topological space with an orbifold atlas. Again, Morita equivalent groupoids give rise to equivalent orbifold atlases. Hence at least in the finite-dimensional case, orbifolds correspond to proper \'{e}tale Lie groupoids. Currently, there seems to be no consensus on the definition of an infinite-dimensional orbifold, however, the Lie groupoid picture generalises with ease.

\begin{defn}
We call a proper \'{e}tale Lie groupoid $\mathcal{G}$ an \emph{orbifold groupoid}.
\end{defn}

It is currently unknown whether an orbifold groupoid modelled on an infinite-dimensional space
corresponds to an orbifold in (infinite-dimensional) local charts. The classical proof (see e.g.\ \cite[Theorem 4]{MaP}) requires a suitable version of a slice theorem for infinite-dimensional Lie group actions. No such theorem is known in general. 
As a special case, Theorem C entails that current groupoids of orbifold groupoids are again orbifold groupoids which are locally isomorphic to action groupoids by Proposition \ref{prop:loc:actgpd}. We conjecture that at least these orbifold groupoids correspond to orbifolds in local charts.\footnote{Since for Banach manifolds and tame \Frechet\, manifolds suitable slice theorems are known. Chen proves a similar statement for orbispaces, see \ref{Numba:develop} below.} However, this is beyond the present paper.

It is known that spaces of orbifold maps are infinite-dimensional orbifolds \cite{Chen,CaPaRaS,Wei,RaV}. Now, since a compact manifold $K$ is a trivial orbifold, does the current groupoid model the space of $C^\ell$-orbifold morphisms $C^\ell_{\text{Orb}} (K,Q)$? In general, this is not even the case if $\mathcal{G}$ represents a manifold.

\begin{exa}\label{ex: toofew}
Let $Q = \Sph = K$ and choose a manifold atlas $\mathcal{V}$ of $Q$ to construct $\Gamma (\mathcal{V})$. Its space of units $\Gamma (\mathcal{V})_0$ is the disjoint union of at least two smooth manifolds (as every atlas of the unit sphere must contain at least two charts). 
Since $K$ is connected, the image of every smooth map $K \rightarrow \Gamma (\mathcal{V})_0$ is contained in exactly one component of $\Gamma (\mathcal{V})_0$, i.e.\ in one chart domain. In particular, the identity $\id \colon \Sph \rightarrow \Sph$ is not contained in the current groupoid, but $C^\ell (\Sph , \Sph) = C^\ell_{\text{Orb}} (K,M)$ for all $\ell \in \N_0 \cup \{\infty\}$.
\end{exa}
 Note the contrast to Example \ref{exa:unitgpd} where we recovered the full space $C^\ell (\Sph,\Sph)$. However, in a specialised case we can avoid atlas groupoids to obtain current groupoids which encode orbifold morphisms.
 
\begin{numba}[Curves into developable orbifolds]\label{Numba:develop}
 Recall that a (smooth) orbifold $(Q,\mathcal{U})$ is \emph{developable} if there is a discrete subgroup $\Gamma \subseteq \Diff (M)$ such that $\Gamma \times M \rightarrow M$ is a proper action and as orbifolds $Q \cong M/\Gamma$ (see \cite[Section III.$\mathcal{G}$ 1.3]{BaH}). To every developable (smooth) orbifold one can associate a proper \'{e}tale (Lie) groupoid, by defining the action groupoid $\Gamma \ltimes M \toto M$ where $\Gamma$ is endowed with the discrete topology (i.e.\ is a $0$-dimensional manifold).\smallskip
 
Let now $Q = M /\Gamma$ be developable. Then every $C^\ell$-orbifold path $I \rightarrow (Q,\mathcal{U})$ from a compact interval $I\subseteq \mathbb{R}$ induces a $C^\ell$-map $I \rightarrow M$.\footnote{In local charts, a $C^\ell$-path is a continuous map $c\colon I \rightarrow Q$ which lifts locally to $C^\ell$-paths $\hat{c}_i \colon ]t_i,t_{i+1}[ \rightarrow V_i$ in orbifold charts such that the lifts $c_i,c_j$ are (locally) related via $\lambda \circ c_i = c_j$ by suitable change of orbifold charts. See \cite[Section 4.1 and Appendix E]{Sch} for details.}
 For $\ell =0$ this is recorded in \cite[III.$\mathcal{G}$ Example 3.9 (1)]{BaH}. For $\ell >0$ every $C^\ell$-orbifold path to $Q$ admits lifts in orbifold charts which embed as open sets of $M$ (due to developability of $Q$). Then \cite[Lemma F.1]{Sch} generalises to $\ell \in \N$ and yields an open cover of $I$ by intervals together with $C^\ell$-lifts such that every $x\in I$ is contained at most in two subintervals. Now gluing the lifts  together using changes of charts (which are induced by the $\Gamma$ action!) yields the desired $C^\ell$-map $I\rightarrow M$. 
 Of course, in general, many different $C^\ell$-curves lift the same $C^\ell$-orbifold path. One now identifies the $C^\ell$-orbifold paths with the orbit space $\text{Orb}^\ell (I,Q) \coloneq C^\ell (I,M)/C^\ell (I,\Gamma\ltimes M)$:
 \begin{enumerate}
 \item if $\ell=0$, the orbit space $\text{Orb}^\ell (I,Q)$ coincides with the space of continuous orbifold paths as is explained in \cite[III.$\mathcal{G}$ Example 3.9 (1)]{BaH}. Thus (up to homotopy of paths) the current groupoid encodes the so-called $\mathcal{G}$-paths and the fundamental group of a developable \'{e}tale groupoid (cf.\ \cite[Section III.$\mathcal{G}$ 3.]{BaH}). 
 \item For $\infty >\ell \geq 0$ \cite[Theorem 3.3.3 (ii)]{Chen} shows that $\text{Orb}^\ell (I,Q)$ coincides with the $C^\ell$-orbifold paths. Thus we recover Chen's orbispace structure on $\text{Orb}^\ell (I,Q)$. In this case, the Lie groupoid structure of $C^\ell (K,\Gamma \ltimes M) \toto C^\ell (K,M)$ is new as in loc.cit.\ only the orbispace structure of the quotient and a topological groupoid structure are discussed. However, the setting of \cite{Chen} is much more general as it allows spaces of orbifold maps \emph{between} orbifolds
to be treated.
 \end{enumerate}
 \end{numba} 

\addcontentsline{toc}{section}{References}
\bibliography{groupoid}

\end{document}